\documentclass{amsart}

\usepackage{amsthm}
\usepackage{amsmath}
\usepackage{amssymb}
\usepackage{enumerate}
\usepackage{graphicx}
\usepackage[hidelinks,pagebackref,pdftex]{hyperref}
\usepackage{booktabs}
\usepackage{color}
\usepackage[dvipsnames]{xcolor}
\usepackage{import}
\usepackage{tikz-cd}
\usepackage{microtype}
\usepackage[normalem]{ulem} 

\AtBeginDocument{%
   \def\MR#1{}
}

\usepackage{marginnote}
\long\def\@savemarbox#1#2{\global\setbox#1\vtop{\hsize\marginparwidth 
  \@parboxrestore\tiny\raggedright #2}}
\newcommand\marginref[1]{{\tiny \marginnote{{\color{blue} #1}}} \normalmarginpar} 

\renewcommand{\marginref}[1]{}

\renewcommand{\sout}[1]{}

\renewcommand*{\backref}[1]{}
\renewcommand*{\backrefalt}[4]{
  \ifcase #1
  [No citations.]
  \or [#2]
  \else [#2]
  \fi }

\numberwithin{equation}{section}
\theoremstyle{plain}
\newtheorem{theorem}[equation]{Theorem}
\newtheorem{corollary}[equation]{Corollary}
\newtheorem{lemma}[equation]{Lemma}
\newtheorem{conjecture}[equation]{Conjecture}
\newtheorem{proposition}[equation]{Proposition}

\newtheorem*{namedtheorem}{\theoremname}
\newcommand{\theoremname}{testing}

\theoremstyle{definition}
\newtheorem{definition}[equation]{Definition}
\newtheorem{remark}[equation]{Remark}

\newcommand{\refthm}[1]{Theorem~\ref{Thm:#1}}
\newcommand{\reflem}[1]{Lemma~\ref{Lem:#1}}
\newcommand{\refprop}[1]{Proposition~\ref{Prop:#1}}
\newcommand{\refcor}[1]{Corollary~\ref{Cor:#1}}

\newcommand{\refdef}[1]{Definition~\ref{Def:#1}}
\newcommand{\refsec}[1]{Section~\ref{Sec:#1}}
\newcommand{\reffig}[1]{Figure~\ref{Fig:#1}}

\newcommand{\calA}{{\mathcal{A}}}

\newcommand{\calH}{{\mathcal{H}}}
\newcommand{\calB}{{\mathcal{B}}}
\newcommand{\calT}{{\mathcal{T}}}
\newcommand{\calS}{{\mathcal{S}}}
\newcommand{\calC}{{\mathcal{C}}}

\newcommand{\bdy}{\partial}

\newcommand{\from}{\colon} 

\newcommand{\MCG}{\operatorname{MCG}}
\newcommand{\Tr}{\operatorname{Tr}}
\newcommand{\Sp}{\operatorname{Sp}}

\newcommand{\cut}{\backslash\backslash}

\title{The triangulation complexity of fibred 3-manifolds}

\author{Marc Lackenby}

\author{Jessica S. Purcell}

\begin{document}

\begin{abstract}
The \emph{triangulation complexity} of a closed orientable 3-manifold $M$ is the minimal number of tetrahedra in any triangulation of $M$. The main theorem of the paper gives upper and lower bounds on the triangulation complexity of any closed orientable hyperbolic 3-manifold that fibres over the circle. We show that the triangulation complexity of the manifold is equal to the translation length of the monodromy action on the mapping class group of the fibre $S$, up to a bounded factor, where the bound depends only on the genus of $S$. 
\end{abstract}

\maketitle

\section{Introduction}\label{Sec:Intro}
The \emph{triangulation complexity} $\Delta(M)$ of a closed orientable 3-manifold $M$ is the minimal number of tetrahedra in any triangulation of $M$. Despite its naive definition, it has some attractive properties. An obvious but important property is that only finitely many 3-manifolds have triangulation complexity less than a given number. It is also useful in normal surface theory, where it is a natural measure of complexity for $M$. However, like other invariants of manifolds that are defined as the minimum of some quantity, it is neither easy to compute nor of obvious topological significance. Indeed, precise values for $\Delta(M)$ are known for a few relatively small examples (see for example~\cite{MartelliPetronio}, \cite{Matveev:ComplexitySurvey}) and for a few infinite families  \cite{MatveevPetronioVesnin, PetronioVesnin, JacoRubinsteinTillmann:LensSpaces, JacoRubinsteinTillmann:Coverings, JacoRubinsteinTillmann:Z2_1, JacoRubinsteinSpreerTillmann:Z2_2, JacoRubinsteinSpreerTillmann:MinimalCusped}.

It is the main goal of this paper to establish that triangulation complexity is a good invariant of 3-manifolds that relates to many other key topological and geometric quantities. Our focus will be on 3-manifolds $M$ that fibre over the circle. We will relate $\Delta(M)$ to the geometry of the mapping class group of the fibre and to its Teichm\"uller space. As a consequence, the modern theory of mapping class groups
can be applied to compute $\Delta(M)$ to within a bounded factor, where the bound depends on the topology of the fibre.

\subsection{Translation length and analogous results}
There is an obvious analogy between the triangulation complexity of a hyperbolic 3-manifold $M$ and its volume. Indeed, there is a well-known inequality due to Gromov and Thurston \cite{Thurston:Notes}, which states that the hyperbolic volume of $M$ is at most $v_3 \Delta(M)$, where $v_3$ is the volume of a regular hyperbolic ideal 3-simplex. A beautiful and important theorem of Brock \cite{Brock:MappingTori} relates the hyperbolic geometry of a fibred hyperbolic 3-manifold with the Weil-Petersson geometry of Teichm\"uller space. Specifically, suppose that $M$ fibres over the circle with fibre $S$ and monodromy $\phi \colon S \rightarrow S$. The monodromy induces an action on the Teichm\"uller space of $S$ that is an isometry with respect to both the Weil-Petersson and Teichm\"uller metrics.

Whenever one has an isometry $h$ of a metric space $(X,d)$, one can consider its \emph{translation length} $\ell_X(h)$, which is defined to be 
\[ \ell_X(h) = \inf \{ d(h(x), x) : x \in X \}. \]
One can also define its \emph{stable translation length} $\overline{\ell}_X(h)$ by
\[ \overline{\ell}_X(h) = \inf \{ d(x,h^N(x)) /N : N \in \mathbb{Z}_{>0} \}, \]
where $x \in X$ is chosen arbitrarily. This is independent of $x \in X$, the infimum is in fact a limit as $N \rightarrow \infty$, and it is at most the translation length; see \cite[II.6.6]{BridsonHaefliger}.
We denote the translation length of the action of $\phi$ on the Teichm\"uller space of $S$ with the Weil-Petersson metric by $\ell_{\mathrm{WP}(S)}(\phi)$.

There are also many simplicial complexes associated with the surface $S$. Brock considers the \emph{pants complex} $\mathcal{P}(S)$. This has a vertex for each collection of disjoint simple closed curves on $S$ that divide it into a union of pairs of pants, and two vertices are joined by an edge when the associated collections of curves are related by a simple type of move. One can assign a path metric to this complex by declaring that each edge has length one, and then the monodromy $\phi$ acts on it by an isometry. One can again therefore define its translation length $\ell_{\mathcal{P}(S)}(\phi)$. The following is Brock's theorem \cite[Theorem 1.1]{Brock:MappingTori}.

\begin{theorem}[Brock]
Let $S$ be a compact orientable surface. Then the following quantities are all within bounded ratios of each other, where the bounds only depend on the Euler characteristic of $S$, for a pseudo-Anosov 
homeomorphism $\phi$ of $S$:
\begin{enumerate}
\item the hyperbolic volume of the fibred manifold $(S \times I)/\phi$;
\item the Weil-Petersson translation length $\ell_{\mathrm{WP}(S)}(\phi)$;
\item the translation length $\ell_{\mathcal{P}(S)}(\phi)$ in the pants complex of $S$;
\item the stable translation length $\overline{\ell}_{\mathcal{P}(S)}(\phi)$.
\end{enumerate}
\end{theorem}

There are many other interesting invariants of hyperbolic 3-manifolds and many other complexes associated with a compact orientable surface. One is naturally led to ask whether any of these other quantities are related as in Brock's theorem. A theorem of Futer and Schleimer asserts that there is another relationship of this form \cite{FuterSchleimer}. This involves the \emph{arc complex} $\mathcal{A}(S)$, which has a vertex for each isotopy class of properly embedded essential arc in $S$ and where two vertices are joined by an edge if their associated arcs can be isotoped to be disjoint.

\begin{theorem}[Futer--Schleimer]
Let $S$ be a compact orientable surface with non-empty boundary. Then the following quantities are all within bounded ratios of each other, where the bounds only depend on the Euler characteristic of $S$,  for a pseudo-Anosov homeomorphism $\phi$ of $S$:
\begin{enumerate}
\item the volume of a maximal collection of cusps for the hyperbolic manifold $(S \times I) / \phi$;
\item the stable translation length $\overline{\ell}_{\mathcal{A}(S)}(\phi)$ in the arc complex of $S$.
\end{enumerate}
\end{theorem}

\subsection{Main results}
In this paper, we continue this theme by relating the triangulation complexity $\Delta(M)$ with the discrete geometry of the mapping class group of $S$ and with the Teichm\"uller space of $S$. The following is our main theorem.

\begin{theorem}\label{Thm:Main}
Let $S$ be a closed orientable surface of genus at least two, and let $\phi\from S\to S$ be a pseudo-Anosov homeomorphism. Then the following quantities are within bounded ratios of each other, where the bounds only depend on the genus of $S$
and a choice of finite generating set for the mapping class group of $S$:
\begin{enumerate}
\item the triangulation complexity of $(S \times I)/ \phi$;
\item the translation length of $\phi$ in the thick part of the Teichm\"uller space of $S$;
\item the translation length of $\phi$ in the mapping class group of $S$;
\item the stable translation length of $\phi$ in the mapping class group of $S$.
\end{enumerate}
\end{theorem}

We now explain these terms in a bit more detail. The \emph{mapping class group} of $S$, denoted $\MCG(S)$, is well known to be a finitely generated group. Once one picks a finite set of generators, it inherits a word metric. Different choices of finite generating sets give different metrics, but any two such metrics remain within a bounded ratio of each other. For the sake of being definite, we can pick a standard generating set, for example as in~\cite{Lickorish}, which then determines the metric on $\MCG(S)$.

The Teichm\"uller space of $S$ can be viewed as the space of marked hyperbolic structures on $S$. Its \emph{thick part} is the subset consisting of hyperbolic structures where every geodesic has length at least some $\epsilon > 0$. A suitable choice of $\epsilon$ must be made. One normally uses a version of the Margulis constant, so that the union of geodesics with length at most $\epsilon$ is a union of disjoint simple closed curves. We will choose $\epsilon$ to have the additional property that the thick part of Teichm\"uller space is path-connected. When Teichm\"uller space is given one of its usual metrics, say the Weil-Petersson metric or the Teichm\"uller metric, the thick part inherits a path metric. We use either of these metrics.

The mapping class group of $S$ acts properly discontinuously and cocompactly on the thick part. Thus it is a standard consequence of the \v{S}varc-Milnor lemma that $\MCG(S)$ with its word metric and the thick part of Teichm\"uller space are quasi-isometric; see \cite[Proposition~8.19]{BridsonHaefliger}. Hence, for any orientation-preserving homeomorphism $\phi$ of $S$, the translation lengths of $\phi$ on $\MCG(S)$ and on the thick part of Teichm\"uller space are within a bounded ratio of each other. So, the relationship between the second and third quantities in \refthm{Main} is easy and well known. The relationship between the third and fourth quantities is also probably well known; we give a proof in \refsec{StableTranslationDistance}.

It is the relationship with the triangulation complexity of the fibred manifold $(S \times I)/\phi$ that is new.
One of the main consequences of \refthm{Main} is that it is now possible to compute the triangulation complexity of fibred 3-manifolds up to a multiplicative error. This is because the translation length in $\MCG(S)$ is computable up to a bounded factor, by a theorem of Masur, Mosher and Schleimer~\cite{MasurMosherSchleimer}; see \refsec{StableTranslationDistance} for more details.

\medskip

Analagous to his theorem about the volume of fibred 3-manifolds, Brock \cite{Brock:ConvexCores} also proved a result about geometrically finite hyperbolic structures on $S \times [0,1]$ . He showed that the volume of the convex core of such a 3-manifold is bounded above and below by linear functions of the Weil-Petersson distance between the points in Teichm\"uller space associated to $S \times \{ 0 \}$ and $S \times \{ 1 \}$. We also have a result in this spirit.

\begin{theorem}\label{Thm:TriangulationProductOneVertex}
  Let $S$ be a closed orientable surface of genus at least two
  and let $\calT_0$ and $\calT_1$ be non-isotopic 1-vertex triangulations of $S$. Then the following are within a bounded ratio of each other, the bound only depending on the genus of $S$:
\begin{enumerate}
\item the minimal number of tetrahedra in any triangulaton $\calT$ of $S \times [0,1]$ such that the restriction of $\calT$ to $S\times\{0\}$ equals $\calT_0$, and the restriction of $\calT$ to $S\times\{1\}$ equals $\calT_1$;
\item the minimal number of 2-2 Pachner moves relating $\calT_0$ and $\calT_1$.
\end{enumerate}
\end{theorem}

In fact, a version of this theorem is the main technical result of the paper, and forms the core of the proof of \refthm{Main}.

\subsection{Applications}

Implicit in the statement of \refthm{Main} is that for a pseudo-Anosov homeomorphism $\phi \colon S \rightarrow S$, its translation length $\ell_{\MCG(S)}(\phi)$ and its stable translation length $\overline{\ell}_{\MCG(S)}(\phi)$ lie within a bounded ratio of each other, where the bound depends only on the genus of $S$. This is in fact a rapid consequence of the theorem of Masur, Mosher and Schleimer mentioned above~\cite{MasurMosherSchleimer}, which interprets translation length in terms of splitting sequences for train tracks. However, a consequence of \refthm{Main} is that the triangulation complexity of a fibred $3$-manifold and its `stable' triangulation complexity (when this term is defined in the obvious way using powers of the monodromy) lie within a bounded ratio of each other. Specifically, we have the following result.

\begin{corollary}\label{Cor:StableComplexity}
Let $S$ be a closed orientable surface. Then there is a constant $k > 0$ (depending only on the genus of $S$) such that, for any pseudo-Anosov homeomorphism $\phi$ of $S$ and any positive integer $N$,
\[ k \, \Delta((S \times I)/ \phi) \leq \frac{\Delta((S \times I)/ \phi^N)}{N} \leq \Delta((S \times I)/ \phi).\]
\end{corollary}

Our methods can also determine the triangulation complexity of lens spaces. These have also been extensively studied by Jaco, Rubinstein and Tillmann \cite{JacoRubinsteinTillmann:LensSpaces}. They conjectured that the triangulation complexity of the lens space $L(p,q)$ can be computed in terms of the continued fraction expansion of $p/q$. Using \refthm{TriangulationProductOneVertex}, we confirm that their conjecture is true, up to a bounded multiplicative error. This will
appear in a forthcoming paper \cite{LackenbyPurcell:Lens}.

\begin{theorem}\label{Thm:LensSpaces}
Let $L(p,q)$ be a lens space, where $p$ and $q$ are coprime integers satisfying $0< q < p$. Let $[a_1, \dots, a_n]$ be the continued fraction expansions of $p/q$ where each $a_i > 0$. Then there is a universal constant $k>0$ such that
\[ k \sum a_i \leq \Delta(L(p,q)) \leq \sum a_i. \]
\end{theorem}

\subsection {Further work}

We conjecture that the main result of this paper, \refthm{Main}, should also be true for compact orientable fibred 3-manifolds with non-empty boundary. It seems likely that the methods introduced in this paper might lead to a proof of this. However, the generalisation does not seem to be immediate. One might be tempted to double the bounded manifold to obtain a closed one, but the resulting monodromy is not pseudo-Anosov. Alternatively, one might attach solid tori to the boundary to extend the manifold to a closed one, but again the resulting monodromy need not be pseudo-Anosov and even if it is, then its translation length in the mapping class group of the closed surface may be much less than the original translation length in the bounded surface. One might alternatively try to adapt the proof of the main theorem. However, various aspects of the argument require the fibre to be closed. Nevertheless, this seems to be a promising area for further research, which would have some attractive applications.

We also believe that the main theorem should hold for arbitrary orientation-preserving homeomorphisms $\phi \from S \to S$ (other than those that are isotopic to the identity),
not just pseudo-Anosov ones. One might attempt to prove such a result by cutting the surface along some disjoint essential simple closed curves $C$ into pieces, where on each piece $\phi$ is either the identity or pseudo-Anosov. This is possible after possibly passing to a power of the monodromy. These curves in $S$ then determine a collection of disjoint incompressible tori in the fibred manifold $M$, and if we cut along these tori, then we obtain a lower bound on the triangulation complexity of each of the pieces, using the conjectured version of our main theorem in the case with non-empty boundary. However, this is not enough to obtain the correct lower bound on the triangulation complexity of the original manifold $M$. This is because one loses track of possible Dehn twists along the curves $C$ in $S$. Thus, to prove the main theorem for general homeomorphisms $\phi$, one either needs to work with a version of the mapping class group for bounded manifolds, where isotopies of the boundary are not allowed, or one needs to adapt the techniques of this paper to deal with homeomorphisms that may have invariant multi-curves.

Another useful direction that one might take is to consider manifolds $M$ that are not given as a fibration, but that are given using a Heegaard splitting. Specifically, one can fix a Heegaard surface $S$ that separates $M$ into two handlebodies $V$ and $W$. In this paper, we consider the \emph{spine graph} $\mathrm{Sp}(S)$, which has a vertex for each spine of $S$ up to isotopy and where two vertices are joined by an edge if and only if the corresponding spines are related by the contraction or expansion of an edge; see Definitions~\ref{Def:Spine} and~\ref{Def:SpineGraph} for more details.
Associated to each of the handlebodies $V$ and $W$, there are subsets $\mathcal{D}_V$ and $\mathcal{D}_W$ of $\mathrm{Sp}(S)$, which we call \emph{disc subsets} defined as follows. We say that $\mathcal{D}_V$ consists of those spines $\Gamma$ in $S$ with the property that some subset of $\Gamma$ forms the boundary of a union of disjoint properly embedded discs $D$ in $V$, such that $V \cut D$ is a ball. We define $\mathcal{D}_W$ similarly.

It is reasonable to make the following conjecture.

\begin{conjecture}\label{Conj:TComplexityHeegaard}
  Let $M$ be a closed orientable manifold with a Heegaard surface $S$ that divides $M$ into handlebodies $V$ and $W$. Then the complexity $\Delta(M)$ is bounded above and below by linear functions of the distance in $\mathrm{Sp}(S)$ between the disc subsets $\mathcal{D}_V$ and $\mathcal{D}_W$.
\end{conjecture}

Our result about lens spaces, \refthm{LensSpaces}, is a confirmation of this conjecture in the case where $S$ is a torus.
This is because in this case, the distance between $\mathcal{D}_V$ and $\mathcal{D}_W$ is coarsely the sum of terms in the continued fraction expansion. See \cite{LackenbyPurcell:Lens} for more details.
Given that the techniques of this paper rely so heavily on almost normal surface theory, which was specifically designed to deal with Heegaard splittings, it is tempting to think that this conjecture may be within reach.

\subsection {Outline of proof}

The first step is to transfer the problem from considering distances in the mapping class group to considering spines of surfaces, which are more geometric and easier to work with. In \refsec{SpineTriangulationGraphs}, we introduce the \emph{spine graph} $\mathrm{Sp}(S)$ for the surface $S$ mentioned above. We show that the spine graph $\mathrm{Sp}(S)$ is quasi-isometric to $\MCG(S)$. So, instead of considering $\ell_{\MCG(S)}(\phi)$ and $\overline{\ell}_{\MCG(S)}(\phi)$ in \refthm{Main}, we can consider $\ell_{\mathrm{Sp}(S)}(\phi)$ and $\overline{\ell}_{\mathrm{Sp}(S)}(\phi)$. The central part of the proof is to show that these quantities are within a bounded factor of $\Delta(M)$, where $M = (S\times [0,1]) / \phi$.

It is fairly straightforward to bound $\Delta(M)$ linearly above in terms of $\ell_{\mathrm{Sp}(S)}(\phi)$. We achieve this in \refsec{SpineTriangulationGraphs} by building a triangulation of $M$ from a path in $\mathrm{Sp}(S)$ joining a spine to its image under $\phi$.
The main work is in showing that $\ell_{\mathrm{Sp}(S)}(\phi)$ is bounded above by a linear function of $\Delta(M)$.

The broad idea is as follows. Start with a triangulation $\calT$ for $M$ where the number of tetrahedra $\Delta(\calT)$ equals $\Delta(M)$. Put a fibre $S$ in normal form, taking one with a minimal \emph{weight}, defined in \refsec{NormAlmostNorm}. When we cut along $S$ we obtain a copy of $S\times [0,1]$ that inherits a handle structure $\calH$ from the triangulation $\calT$. The surfaces $S \times \{ 0 \}$ and $S \times \{ 1 \}$ inherit handle structures, where each handle is a component of intersection between $S$ and a handle of $\calH$. These handle structures determine cell structures on $S \times \{ 0 \}$ and $S \times \{ 1 \}$ by declaring that each handle is a 2-cell. The gluing map $\phi$ from $S \times \{ 1 \}$ to $S \times \{ 0 \}$ preserves these handle structures and so is cell-preserving. Pick a spine $\Gamma$ in $S \times \{ 0 \}$ that is a subcomplex of the 1-skeleton of the cell structure; such a spine is called \emph{cellular}.

The main part of the proof is to isotope $\Gamma$ across $S \times [0,1]$, making modifications to the spine as we go, until it is a subcomplex $\Gamma'$ of the cell structure on $S \times \{ 1 \}$. The key claim is that the number of steps in $\Sp(S)$ between $\Gamma$ and $\Gamma'$ is at most a linear function of $\Delta(\calT)$. Of course, the image of $\Gamma'$ under the gluing map $\phi$ will probably not be equal to $\Gamma$. But both it and $\Gamma$ are subcomplexes of the same cell structure on $S \times \{ 0 \}$, and this allows us to relate them by a bounded number of steps in $\Sp(S)$, 
done in \refsec{SpineTriangulationGraphs}. By the end of this process, we have found a path in $\mathrm{Sp}(S)$ relating $\Gamma$ and $\phi(\Gamma)$, with length that is at most a linear function of $\Delta(\calT)$. This will prove the main theorem.

We now explain how to find the sequence of steps in $\Sp(S)$ taking $\Gamma$ to $\Gamma'$. Using the machinery of normal and almost normal surface theory, between $S\times\{0\}$ and $S\times\{1\}$ lies a collection of normal surfaces, almost normal surfaces, and surfaces interpolating between the two. 
To form this collection, start with a collection of disjoint non-parallel normal fibres that satisfy a maximality property. 
Between each of these, there is an almost normal fibre. Start with any almost normal fibre and apply weight-reducing isotopies, all in the same transverse direction. 
We define these isotopies in \refsec{Normalising}, having introduced the basic definitions and results about normal and almost normal surfaces in \refsec{NormAlmostNorm}. If we apply these isotopies to an almost normal surface in our collection, then we show that we end up at a normal surface. By the maximality of our initial collection of normal surfaces, this normal surface must be parallel to one in our collection. In this way, we get a collection of fibres interpolating between $S \times \{ 0 \}$ and $S\times \{ 1 \}$. Some of these are normal, some are almost normal and, between these, we have a type of surface that we call \emph{nearly normal}.

Using usual isotopy moves to simplify surfaces, it will not be the case that we obtain an upper bound on the number of surfaces in our collection. Thus, even if we were able to bound the number of steps to transfer a spine between any pair using these usual moves, it would still not be the case that we could obtain a linear upper bound on the distance in $\Sp(S)$ between $\Gamma$ and $\Gamma'$. 

In order to deal with this problem, we introduce isotopies, known as generalised isotopy moves, that have a more drastic effect on the fibres. These are defined in terms of the \emph{parallelity bundles} of the normal and almost normal surfaces. This terminology is recalled in \refsec{ParallelityBundles}. A large portion of the argument of this paper concerns parallelity bundles, and analysing their topology. 

For a normal or almost normal surface $S$, the space between two parallel normal discs of $S$ is an $I$-bundle. These $I$-bundles patch together to form the parallelity bundle for the exterior of $S$. Its horizontal boundary is a subsurface of the boundary of the exterior of $S$, and its vertical boundary is a collection of properly embedded annuli. 
It was shown in \cite{Lackenby:CrossingNo} that this $I$-bundle may be extended to a possibly larger $I$-bundle $\calB$ with incompressible horizontal boundary, called a \emph{generalised parallelity bundle}. Each component of $\calB$ either is an $I$-bundle over a disc,
an $I$-bundle over an annulus, or has incompressible vertical boundary (see \refthm{IncompressibleHorizontalBoundary}).
In particular, when a component of $\calB$ has its entire horizontal boundary in the same side of the same fibre, then it is an $I$-bundle over a disc or an annulus. 
This is a consequence of the way that the vertical annuli can be embedded in $S \times [0,1]$. In our sequence of isotopies, when a fibre starts to enter such a component of $\calB$, we can perform a generalised isotopy move, which moves it across this component in a single step. By allowing these larger moves, we can ensure that the number of moves is at most a linear function of $\Delta(\calT)$. Properties of such moves are given in \refsec{GenIsotopy}.

We need to control the number of modifications to the spine when a generalised isotopy move is performed. This is done for all but one case in \refsec{Spines}. The difficult part is the move involving an $I$-bundle over an annulus. This is called an \emph{annular simplification}. What is important here is the width of this annulus, which is the distance between its boundary components, as measured using the cell structure. Annular simplifications are analysed in \refsec{AnnularSimplification}.

In \refsec{InterpolatingSpines}, we combine these results to give a proof of \refthm{TriangulationProductOneVertex}, using a delicate inductive argument to pick normal fibres carefully. 

We give the proof of the main theorem in \refsec{MainProof}. This uses bounds on modifications as we pass a spine $\Gamma$ from $S\times\{0\}$ to a spine $\Gamma'$ in $S\times\{1\}$. We then need to convert this to $\phi(\Gamma)$ in $\Sp(S)$. Using the results in Section \ref{Sec:SpineTriangulationGraphs}, the number of moves is bounded above by the number of 1-cells in $\Gamma$. But unfortunately we do not have a good upper bound on this quantity. In order to circumvent this difficulty, we consider not the fibred manifold with monodromy $\phi$, but rather some finite cover with monodromy some high power $\phi^n$. In order to compare the translation length $\ell_{\mathrm{Sp}(S)}(\phi)$ and $\ell_{\mathrm{Sp}(S)}(\phi^n)$, we relate them both to the stable translation length $\overline{\ell}_{\mathrm{Sp}(S)}(\phi)$. We do this in \refsec{StableTranslationDistance}, using the machinery of train tracks.

\subsection{Acknowledgements}
J.~Purcell was partially supported by the Australian Research Council. M.~Lackenby was partially supported by the EPSRC.

\section{The spine graph and triangulation graphs}\label{Sec:SpineTriangulationGraphs}

In this section, we show that we can transfer the problem of proving \refthm{Main} from considering distances in the mapping class group to considering spines of surfaces, which are more geometric and easier to work with. This is done in \refprop{SpineGraphMCG}. Using this, we bound the triangulation complexity $\Delta((S \times I)/\phi)$ from above by a linear function of the translation length of $\phi$ in the mapping class group, in \refprop{UpperBound}.

\subsection{The spine graph}\label{Sec:SpineGraph}

Throughout this paper, our graphs may have multiple edges between vertices and may have edge loops. 

\begin{definition}\label{Def:Spine} 
A \emph{spine} for a closed surface $S$ is an embedded graph $\Gamma$ such that $S \cut \Gamma$ is a disc, and $\Gamma$ has no vertices with valence $0$, $1$ or $2$. It is a \emph{trivalent} spine if every vertex has valence $3$.
\end{definition}

When an embedded graph satisfies the first of the above conditions, but has some vertices of valence $1$ or $2$, then there are some easy modifications that one can apply to turn it into a spine. If there is a valence $1$ vertex,
then we remove it and its incident edge. We continue until there are no valence $1$ vertices. We can then remove each valence $2$ vertex by amalgamating its incident edges into a single edge. In some of the arguments below, the changes that we make to a spine may create vertices with valence $2$, but in this
case, we may immediately perform the above modification to turn it back into a spine.

Dual to a spine is a cell structure for the surface with a single vertex, in which every 2-cell has at least 3 sides. When the spine is trivalent, then its dual is a 1-vertex triangulation of the surface. 

\begin{lemma}[Bound on vertices and edges, spine]\label{Lem:BoundOnVerticesAndEdges}
Let $\Gamma$ be a spine for a closed orientable surface $S$. Then $\Gamma$ has at most $4g(S)-2$ vertices and at most $6g(S)-3$ edges.
\end{lemma}

\begin{proof}
Let $d(v)$ denote the degree of a vertex $v$. Let $V$ and $E$ denote the number of vertices and edges of the spine. Then, because $S \cut \Gamma$ is a single disc, we deduce that
\[ \chi(S) = V - E + 1 = 1 + \sum_{v} \left (1 - \frac{d(v)}{2} \right ) \leq 1 -\frac{V}{2},\]
where the sum ranges over all vertices $v$.
Hence, $V \leq 2 - 2\chi(S) = 4g(S)-2$. 
So
\[ E = V + 1 -\chi(S) \leq 3 - 3\chi(S) = 6g(S)-3. \qedhere \]
\end{proof}

This implies that the sphere does not have a spine, since it cannot have a negative number of vertices. In a similar spirit, we have the following.

\begin{lemma}[Bound on triangles and edges, triangulation] \label{Lem:NumberTrianglesAndEdges}
Let $\calT$ be a triangulation of a closed orientable surface with $V$ vertices. Then the number of triangles is $2V + 4g(S) - 4$ and the number of edges is $3V + 6g(S) - 6$. \end{lemma}

\begin{proof} Let $E$ and $F$ be the number of edges and triangles. Then $3F = 2E$, and so $\chi(S) = V - E + F = V - (F/2)$. Rearranging gives the required formulae. \end{proof}

We now describe some modifications to spines.

\begin{definition}
In an \emph{edge contraction} on a spine $\Gamma$, one collapses an edge that joins distinct vertices, thereby amalgamating these vertices into a single vertex. An \emph{edge expansion} is the reverse of this operation. See \reffig{EdgeContraction}.
\end{definition}

\begin{figure}
  \includegraphics{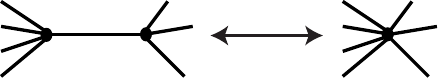}
  \caption{An edge contraction/expansion}
  \label{Fig:EdgeContraction}
\end{figure}

\begin{definition} \label{Def:SpineGraph} 
The \emph{spine graph} $\mathrm{Sp}(S)$ for a closed orientable surface $S$ other than a 2-sphere is a graph defined as follows. It has a vertex for each spine of $S$, up to isotopy of $S$. Two vertices are joined by an edge if and only if their spines differ by an edge contraction or expansion.
\end{definition}

Closely related is the following concept.

\begin{definition}
  The \emph{triangulation graph} $\mathrm{Tr}(S)$ for a closed orientable surface $S$ is a graph defined as follows. It has a vertex for each 1-vertex triangulation of $S$, up to isotopy of $S$. Two vertices are joined by an edge if they differ by a 2-2~Pachner move.

Recall that a \emph{2-2~Pachner move} on a triangulation removes an edge with distinct triangles on each side of it, thereby forming a quad, and then it introduces the edge that is the other diagonal of this quad. 
Similarly, a \emph{1-3~Pachner move} subdivides one triangle into three; see \reffig{PachnerMove}. A \emph{3-1~Pachner move} is the reverse of a 1-3~Pachner move.

\begin{figure}
  \includegraphics{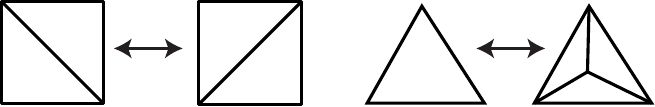}
  \caption{2-2 and 1-3 Pachner moves}
  \label{Fig:PachnerMove}
\end{figure}

More generally, for any positive integer $v$, define the graph $\Tr(S;v)$ as follows. It has a vertex for each triangulation of $S$ with at most $v$ vertices, up to ambient isotopy.
Two vertices of $\Tr(S;v)$ are joined by an edge if the corresponding triangulations differ by a 1-3, 2-2 or 3-1 Pachner move.
\end{definition}

A well known application of the \v{S}varc-Milnor lemma gives the following result.

\begin{proposition}[Quasi-isometric to $\mathrm{MCG}(S)$]\label{Prop:SpineGraphMCG}
For a closed orientable surface $S$ other than a 2-sphere and a positive integer $v$, its spine graph $\mathrm{Sp}(S)$ and triangulation graphs $\mathrm{Tr}(S)$ and $\Tr(S;v)$ are all quasi-isometric to the mapping class group of $S$.
\end{proposition}

\begin{proof}
This is an application of the \v{S}varc-Milnor lemma; see \cite[Proposition~8.19]{BridsonHaefliger}. We just need to verify that the hypotheses of this lemma hold.

The triangulation graph $\mathrm{Tr}(S)$ is well-known to be connected.
For example, given two 1-vertex triangulations of $S$, we can remove their vertices to obtain ideal triangulations of the surface $S'$ that has a single puncture. Two such ideal triangulations differ by a sequence of 2-2 Pachner moves, by \cite[Lemma~6]{Lackenby:Taut}. Hence, the original 1-vertex triangulations differ by a sequence of 2-2 Pachner moves.

A similar argument gives that, for any positive integer $v$, the triangulation graph $\Tr(S;v)$ is connected. Suppose that we are given two triangulations of $S$, each with at most $v$ vertices. If they have different numbers of vertices, then we can apply 1-3 Pachner moves to one of them until they have the same  number of vertices. We can then remove their vertices, creating ideal triangulations of the same punctured surface $S'$. As above, these are related by a sequence of 2-2 Pachner moves.

The spine graph is connected, for a similar reason. Given two spines for $S$, we may apply edge expansions to each of them until each of their vertices is 3-valent. The resulting spines are dual to 1-vertex triangulations of $S$. These are related by a sequence of 2-2 Pachner moves. Dual to a 2-2 Pachner move is an operation on a spine that can be achieved by an edge contraction followed by an edge expansion.

Thus the spine graph and triangulation graphs are geodesic metric spaces. The mapping class group acts on them properly, using the fact that if a homeomorphism of a surface fixes a spine or triangulation pointwise, then it is isotopic to the identity. Hence, the subgroup of the mapping class group that sends a spine or triangulation back to itself, up to isotopy, is finite.

Finally, the actions of the mapping class group on the spine graph and triangulation graphs are cocompact. That is, up to orientation-preserving homeomorphism of $S$, there are only finitely many spines for $S$ and only finitely many triangulations with at most $v$ vertices. This follows from Lemmas~\ref{Lem:BoundOnVerticesAndEdges} and~\ref{Lem:NumberTrianglesAndEdges}.

Thus, all the requirements of the \v{S}varc-Milnor are verified, and so the mapping class group of $S$ is quasi-isometric to the spine graph and to the triangulation graphs.
\end{proof}

We are now in a position to prove one direction in \refthm{Main}, which is that the triangulation complexity $\Delta(M)$ of a fibred 3-manifold $M$ with monodromy $\phi$ is bounded above by a linear function of the translation length of $\phi$ in the mapping class group of $S$.

\begin{proposition}[Upper bound on complexity]\label{Prop:UpperBound}
  Let $S$ be a closed orientable surface, $\phi$ an orientation preserving homeomorphism of $S$, and let $M$ denote the mapping torus $(S\times I)/\phi$. There exist constants $A$ and $B$, depending only on $S$, such that the triangulation complexity $\Delta(M)$ is bounded above by
  \[ \Delta(M) \leq A \cdot \ell_{\MCG(S)}(\phi) + B. \]
\end{proposition}

\begin{proof}
It suffices to focus on the case where $S$ is not a 2-sphere, since any orientation-preserving homeomorphism of the 2-sphere is isotopic to the identity. Thus, we may assume that $S$ has a 1-vertex triangulation. Since $\MCG(S)$ is quasi-isometric to $\mathrm{Tr}(S)$, it suffices to bound $\Delta(M)$ above in terms of the translation length in $\mathrm{Tr}(S)$. Let $\calT$ be a 1-vertex triangulation of $S$ such that $d(\phi(\calT), \calT)$ equals the translation length $\ell_{\mathrm{Tr}(S)}(\phi)$. First build a triangulation of $S \times [0,1]$ so that the induced triangulations of $S \times \{ 0 \}$ and $S \times \{ 1 \}$ are both isotopic to $\calT$. For example, for each triangle $T$ of $\calT$, take the prism $T\times[0,1]$, then triangulate it by coning to a central vertex. This can be achieved so that the number of tetrahedra is bounded above by a linear function of the genus of $S$.

A shortest path in $\mathrm{Tr}(S)$ joining $\calT$ to $\phi(\calT)$ specifies a sequence of 1-vertex triangulations of $S$, each obtained from its predecessor by a 2-2 Pachner move. At each such move, attach a tetrahedron onto $S \times \{ 1 \}$ so that in the new triangulation of $S \times [0,1]$ the induced triangulation of $S \times \{ 1 \}$ inherits the new triangulation. In this way, we build a triangulation of $S \times [0,1]$ where the bottom $S\times\{0\}$ is triangulated using $\calT$ and the top $S\times\{1\}$ is triangulated using $\phi(\calT)$ and where the number of tetrahedra is bounded above by the translation length of $\phi$ in $\mathrm{Tr}(S)$ plus a constant. Glue bottom to top via $\phi$ and we thereby construct our required triangulation of $M$.
\end{proof}

\subsection{Edge contractions and expansions on the same surface}\label{Sec:SameSurface}

\refprop{UpperBound} gives an upper bound on triangulation complexity in terms of translation length in the mapping class group. We need a lower bound to finish the proof of \refthm{Main}. We will obtain the lower bound by transferring spines of surfaces through the triangulation of the manifold $(S\times I)/\phi$, bounding the number of edge contractions and expansions along the way. There are two steps for this transfer. First, we will transfer a spine $\Gamma$ from $S\times\{0\}$ to a spine $\Gamma'$ in $S\times\{1\}$. This step is difficult, and will require most of the rest of the work in the paper. 
The second step is to bound the number of edge contractions and expansions to transfer the spine $\Gamma'$ in $S\times\{1\}$ to $\phi(\Gamma)$. This second step does not require much additional work, and so we give a bound in this subsection.

We begin with a lemma that bounds the number of edge expansions and contractions taking a given spine on a surface with boundary to the dual of a fixed ideal triangulation. This is not exactly the setting we need, since our surfaces are closed, but we can easily convert our problem to this one, in \reflem{SpinesSameComplex}. The ideas of the proof are taken from \cite{Lackenby:Taut}. 

Recall that a \emph{spine} for a compact orientable surface $S$ with non-empty boundary is a graph $\Gamma$ embedded in the interior of $S$ with no vertices of valence 0, 1 or 2, and such that $S \cut \Gamma$ is a regular neighbourhood of $\partial S$.

\begin{lemma}[Replacing spine with the dual of an ideal triangulation]\label{Lem:SpineToTriangulation}
Let $S$ be a compact orientable surface with non-empty boundary. Let $\Gamma$ be a spine for $S$ and let $T$ be an ideal triangulation with edges $T^1$. Then there is a sequence of at most $4|\chi(S)| \cdot  (|\Gamma \cap T^1|+1)$ edge expansions and contractions taking $\Gamma$ to the dual of $T$.
\end{lemma}

\begin{proof}
In \cite{Lackenby:Taut}, it was shown how to perform edge expansions and contractions to transform $\Gamma$ to the dual of $T$. If one follows this proof, one readily obtains the bound in the lemma. We briefly sketch the argument. We first convert $\Gamma$ to a spine with only trivalent vertices using edge expansions, without changing its intersection with $T^1$. Each edge expansion increases the number of edges by $1$. By an argument analogous to that in Lemma \ref{Lem:BoundOnVerticesAndEdges}, the number of edges of a spine of $S$ with trivalent vertices is $3 |\chi(S)|$. So, the number of edge expansions is less than $3|\chi(S)|$.

The surface $S \cut \Gamma$ is a collection of annuli, one for each component of $\partial S$. The 1-skeleton of $T$ intersects these annuli in arcs. We may assume that these arcs miss the vertices of $\Gamma$. No arc can run from $\partial S$ back to $\partial S$ (without meeting $\Gamma$), since every edge of an ideal triangulation is essential in $S$. If every arc runs from $\partial S$ to $\Gamma$, then $T$ is dual to $\Gamma$. On the other hand, if some arc runs from $\Gamma$ to $\Gamma$, then an outermost one in $S \cut \Gamma$ separates off a disc. This disc lies in a triangle of $T$, and has both endpoints on the same edge. Using this disc, we may apply at most $4|\chi(S)|$ edge contractions and expansions to $\Gamma$ and reduce $|T^1 \cap \Gamma|$. So after at most $4 |\chi(S)| \cdot  |\Gamma \cap T^1|$ edge expansions and contractions, we end with the spine dual to $T$.
\end{proof}

The following lemma will be used in the proof of \refthm{Main}, to transfer a spine coming from $\phi(S\times\{1\})$ to a spine on $S\times\{0\}$. 

\begin{lemma}[Spines in the same cell structure]\label{Lem:SpinesSameComplex}
Let $S$ be a closed orientable surface with a cell structure $\mathcal{C}$. Let $\Gamma$ and $\Gamma'$ be spines for $S$ that are both subcomplexes of $\mathcal{C}$. Let $n$ be the number of 1-cells in $\Gamma$. Then $\Gamma$ and $\Gamma'$ differ by a sequence of edge expansions and contractions with length at most $48 g(S)^2 n$.
\end{lemma}

Observe that $n$ in the above lemma can be arbitrary, even for a fixed surface, since the 1-cells are counted in $\mathcal{C}$.

\begin{proof}[Proof of \reflem{SpinesSameComplex}]
Form a 1-vertex triangulation $T$ dual to $\Gamma'$, as follows. Pick a point $p$ in the interior of a 2-cell to be the vertex. For each edge of $\Gamma'$, pick a point in a 1-cell lying in that edge. Construct arcs coming out from this point, from each side of the edge, and ending at $p$. The union of these two arcs will form edges of $T$. Since the vertices of $\Gamma'$ need not be trivalent, we may need to add further edges to $T$ to make it a triangulation. We can ensure that the edges of $T$ are transverse to the 1-cells of $\mathcal{C}$. We can further ensure that each edge of $T$ intersects each 1-cell of
$\Gamma$
at most once. The number of edges of $T$ is $6g(S) - 3$ by Lemma~\ref{Lem:NumberTrianglesAndEdges}. So the number of intersections between the edges of $T$ and $\Gamma$ is at most $n(6g(S)-3)$.

Now apply \reflem{SpineToTriangulation} to the ideal triangulation obtained from $T$ by removing $p$. We deduce that there is a sequence of at most $4|\chi(S \cut N(p))| \  (|\Gamma \cap T^1|+1) \leq 48g(S)^2 n$ edge expansions and contractions taking $\Gamma$ to $\Gamma'$.
\end{proof}

\section{Stable translation length in the mapping class group}\label{Sec:StableTranslationDistance}

In this section, we will collate some useful facts about the mapping class group $\MCG(S)$. We will discuss results of Agol~\cite{Agol:IdealTriangulations}, Hamenstadt~\cite{Hamenstadt} and Masur, Mosher, and Schleimer~\cite{MasurMosherSchleimer}, which will demonstrate that the translation length in $\MCG(S)$ of a pseudo-Anosov $\phi$ is readily calculable up to a bounded factor, where the bound only depends on the genus of $S$. Hence, as a consequence of our main theorem, the triangulation complexity of $(S \times I)/\phi$ is also calculable up to such a factor. 

\subsection{Train tracks}

We recall some terminology related to train tracks. A \emph{pre-track} $\tau$ is a graph smoothly embedded in the closed orientable surface $S$ where every vertex has degree three and such that at each vertex of the graph, all three edges coming into that vertex share the same tangent line, with one edge coming along the line in one direction, and the other two edges entering from the other direction. The edges of the graph are called \emph{branches} and the vertices are called \emph{switches}. See \reffig{SplitSlide} for some pictures of pre-tracks.

\begin{figure}
\begingroup%
  \makeatletter%
  \providecommand\color[2][]{%
    \errmessage{(Inkscape) Color is used for the text in Inkscape, but the package 'color.sty' is not loaded}%
    \renewcommand\color[2][]{}%
  }%
  \providecommand\transparent[1]{%
    \errmessage{(Inkscape) Transparency is used (non-zero) for the text in Inkscape, but the package 'transparent.sty' is not loaded}%
    \renewcommand\transparent[1]{}%
  }%
  \providecommand\rotatebox[2]{#2}%
  \newcommand*\fsize{\dimexpr\f@size pt\relax}%
  \newcommand*\lineheight[1]{\fontsize{\fsize}{#1\fsize}\selectfont}%
  \ifx\svgwidth\undefined%
    \setlength{\unitlength}{210.53871346bp}%
    \ifx\svgscale\undefined%
      \relax%
    \else%
      \setlength{\unitlength}{\unitlength * \real{\svgscale}}%
    \fi%
  \else%
    \setlength{\unitlength}{\svgwidth}%
  \fi%
  \global\let\svgwidth\undefined%
  \global\let\svgscale\undefined%
  \makeatother%
  \begin{picture}(1,0.72668617)%
    \lineheight{1}%
    \setlength\tabcolsep{0pt}%
    \put(0,0){\includegraphics[width=\unitlength,page=1]{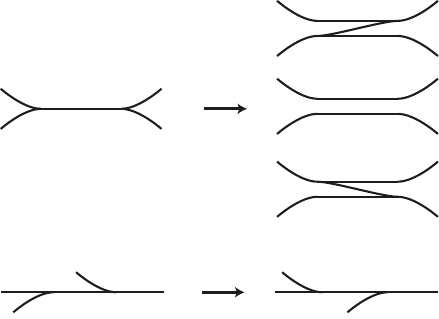}}%
    \put(0.46284952,0.41901961){\color[rgb]{0,0,0}\makebox(0,0)[lt]{\lineheight{1.25}\smash{\begin{tabular}[t]{l}split\end{tabular}}}}%
    \put(0.4554366,0.00071895){\color[rgb]{0,0,0}\makebox(0,0)[lt]{\lineheight{1.25}\smash{\begin{tabular}[t]{l}slide\end{tabular}}}}%
  \end{picture}%
\endgroup%

  \caption{Top: The three types of split; Bottom: A slide}
  \label{Fig:SplitSlide}
\end{figure}

Now let $\tau$ be a pre-track, and consider a component $F$ of $S \cut \tau$. Its boundary consists of a union of arcs, where each arc is identified with a branch of $\tau$. When two such arcs meet in $\partial F$, it might or might not be possible to join them to form a single smooth arc. If they cannot be joined in this way at a point $p$ in $\partial F$, we say that $p$ is a \emph{cusp} of $F$. Define the \emph{index} $I(F)$ of $F$ to be $\chi(F)$ minus half the number of cusps in $\partial F$. The pre-track $\tau$ in the closed orientable surface $S$ is said to be a \emph{train track} if each component of $S \cut \tau$ has negative index. It is \emph{filling} if each component of $S \cut \tau$ is, in addition, a disc.

\begin{lemma}[Complement of train track]\label{Lem:ComplementOfTrainTrack}
A train track $\tau$ in a closed orientable surface $S$ has at most $4g(S) - 4$ complementary regions.
\end{lemma}

\begin{proof}
  The sum of indices $I(F)$ over all components $F$ of $S\cut \tau$ is equal to $\chi(S)$. Because $\tau$ is a train track, each component has index at most $-1/2$. Thus the number of components of $S \cut \tau$ is at most $-2\chi(S)$. The lemma follows.
\end{proof}

A train track $\tau$ has a natural thickening $N(\tau)$ called its \emph{fibred neighbourhood}, as follows. Each branch $e$ is thickened to $e \times [-1,1]$. The intervals $\{ \ast \} \times [-1,1]$ are called \emph{fibres}. At a switch $v$, two branches $e_1$ and $e_2$ approach the switch from one direction and the third branch $e_3$ approaches from the other direction. To form $N(\tau)$, we attach $e_1 \times [-1,1]$ and $e_2 \times [-1,1]$ to $e_3 \times [-1,1]$ by identifying $v \times [-1,1] \subset e_1 \times [-1,1]$ and $v \times [-1,1] \subset e_2 \times [-1,1]$ with $v \times [-1,0]$ and $v \times [0,1]$ in $e_3 \times [-1,1]$. The resulting subset $N(\tau)$ of $S$ contains $\tau$ but is not actually a regular neighbourhood of $\tau$. However, it could become a regular neighbourhood after applying a small isotopy.

We say that a lamination $\mathcal{L}$ is \emph{carried} by the train track $\tau$ if it lies within the fibred neighbourhood $N(\tau)$ and is transverse to the fibres. If, in addition, $\mathcal{L}$ intersects every fibre of $N(\tau)$, then it is \emph{fully carried} by $\tau$. When $\mathcal{L}$ has a transverse measure, each fibre inherits a measure which is a non-negative real number called its \emph{weight}.

A train track $\tau$ is known as \emph{recurrent} if it fully carries a measured lamination. It is known as \emph{transversely recurrent} if there is a union of disjoint simple closed curves that intersects $\tau$ transversely away from its switches, and that intersects each branch at least once. The train track is said to be \emph{birecurrent} if it is recurrent and transversely recurrent.

\subsection{Splits and slides}\label{Sec:SplitsSlides}

In this subsection, we present some well known modifications that one can make to a train track $\tau$. 

If we pick a point $p$ in the interior of a branch $b$ of $\tau$, the closure of each component of $b - p$ is a \emph{half-branch}. Make such a choice of point for each branch, thereby expressing every branch as the union of two half-branches. We say that a half-branch is \emph{large} if at the switch to which it is incident, there is no other half-branch coming in from the same direction. Otherwise, the half-branch is \emph{small}.

A branch is \emph{large} if both its half-branches are large. It is \emph{small} if both its half-branches are small. Otherwise, it is \emph{mixed}.

A \emph{split} or a \emph{slide} is one of the modifications to a train track $\tau$ shown in \reffig{SplitSlide}. There are three possible splits that can take place at a large branch; the middle case is called a \emph{central split}.

If we forget the tangential structure at each switch of a train track $\tau$, we get a trivalent graph embedded in $S$. When $\tau$ is filling, this graph is dual to a triangulation of $S$. By \reflem{ComplementOfTrainTrack}, the dual triangulation has at most $k = 4g(S) - 4$ vertices. When two filling train tracks differ by a slide, the associated triangulations differ by a 2-2 Pachner move. The same is true of non-central splits. When two filling train tracks differ by a central split, their dual triangulations also differ by a bounded number of Pachner moves, where the bound only depends on the genus of $S$. However, we will not need to compute this bound.
Hence, a sequence of slides and splits can be viewed as determining a sequence of points in $\Tr(S;k)$. The following important result is \cite[Theorem~6.2]{MasurMosherSchleimer}.

\begin{theorem}[Masur--Mosher--Schleimer] \label{Thm:TrainTrackQuasiGeodesic}
There is a constant $Q>0$, depending only on the genus of $S$, with the following property. Let $\tau_i$ be a sequence of filling birecurrent train tracks in $S$, each obtained from its predecessor from a slide or a split. Suppose that, in all subsequences consisting only of slides, no two train tracks in the subsequence are isotopic. Then this sequence of train tracks determines a $Q$-quasi-geodesic in $\Tr(S;k)$, where $k = 4g(S) - 4$.
\end{theorem}

The statement of Theorem \ref{Thm:TrainTrackQuasiGeodesic} above is not stated identically to that of \cite[Theorem~6.2]{MasurMosherSchleimer}, but it follows immediately from that theorem. First, the fact that the train tracks are filling and birecurrent is implicit in \cite{MasurMosherSchleimer}; they define a train-track graph to have vertices consisting only of such train tracks. Second, in \cite[Theorem~6.2]{MasurMosherSchleimer}, it is shown that a sequence of trains tracks as in Theorem \ref{Thm:TrainTrackQuasiGeodesic} determines a $Q'$-quasi-geodesic in this train-track graph, for some $Q'>0$ depending only on the genus of $S$. As discussed above, the sequence determines a path in $\Tr(S;k)$. Similar to the case of the triangulation graph and the spine graph, the \v{S}varc-Milnor lemma shows that the train-track graph is quasi-isometric to $\MCG(S)$. The map from the train-track graph to $\Tr(S; k)$ that sends a train track to its dual triangulation is $\MCG(S)$-equivariant. Hence it is a quasi-isometry. Therefore, a sequence in the train-track graph that forms a $Q'$-quasi-geodesic gives rise to a $Q$-quasi-geodesic in $\Tr(S; k)$, for some $Q>0$ depending only on the genus of $S$.

\subsection{pseudo-Anosov homeomorphisms}

Let $\phi$ be a pseudo-anosov homeomorphism of the closed orientable surface $S$. Recall that there are two associated measured singular foliations $(\mathcal{F}_s,\mu_s)$ and $(\mathcal{F}_u, \mu_u)$ and a real number $\lambda > 1$ such that 
$\phi(\mathcal{F}_s,\mu_s) = (\mathcal{F}_s, \lambda \mu_s)$ and  $\phi(\mathcal{F}_u,\mu_u) = (\mathcal{F}_u, \lambda^{-1} \mu_u)$. The measured singular foliations are known as the \emph{stable} and \emph{unstable} measured singular foliations for $\phi$, and $\lambda$ is the \emph{dilatation}.

\begin{remark}
There are differing conventions in the literature \cite{Thurston:GeometryDynamics, FarbMargalit, CassonBleiler, FLP, Agol:IdealTriangulations} as to whether the stable lamination satisfies $\phi(\mathcal{F}_s,\mu_s) = (\mathcal{F}_s, \lambda \mu_s)$ or $(\mathcal{F}_s, \lambda^{-1} \mu_s)$. We chose to adopt the former convention, since this makes $(\mathcal{F}_s,\mu_s)$ an attracting fixed point in projective measured lamination space for the action of $\phi$. This follows Casson and Bleiler's presentation \cite{CassonBleiler}, and is consistent with the conventions used by Agol \cite{Agol:IdealTriangulations}.
\end{remark}

We only focus on the stable foliation $\mathcal{F}_s$. One may split its leaves to form a measured lamination $(\mathcal{L}, \mu)$, which is then fully carried by a train track $\tau$. The transverse measure of each branch is a positive real number. These numbers satisfy simple linear constraints, which assert that at each switch, the weight of the large half-branch coming into it is equal to the sum of the weights of two small half-branches. Such an assignment of positive real numbers to the branches of $\tau$ makes it a \emph{measured train track} $(\tau, \mu)$.
This data can be computed, via the following result of Bestvina and Handel~\cite{BestvinaHandel:Surface}.

\begin{theorem}[Bestvina--Handel]\label{Thm:ConstructibleTrainTrack}
Let $\phi$ be an element of the mapping class group of $S$, given as a product of standard generators. If $\phi$ has a pseudo-Anosov representative, then there is an algorithm to compute a measured train track $(\tau, \mu)$ for the stable measured foliation of $\phi$ and to compute the dilatation of $\phi$, as an exact algebraic number.
\end{theorem}

Indeed, it is worth emphasing that $\tau$ and $\mu$ are readily computable in practice. For example, in the closely related case of a pseudo-Anosov on a punctured surface, the program \emph{flipper} \cite{Bell:Flipper} is a helpful practical method of doing this.

\subsection{Eventually periodic train track sequences}

The stable measured lamination satisfies the condition $\phi(\mathcal{L},\mu) = (\mathcal{L}, \lambda \mu)$. Hence, the measured train tracks $(\tau, \mu)$ and $(\phi (\tau), \lambda^{-1} \mu)$ both carry the same measured lamination. When two measured train tracks carry the same measured lamination, then they differ by a sequence of measured splits, slides and their inverses, up to isotopy \cite[Theorem 2.8.5]{HarerPenner}. This terminology is defined as follows. Suppose that $(\tau, \mu)$ is a measured train track and that $b$ is a mixed branch. Then one may slide along $b$, as defined in \refsec{SplitsSlides}, and the new train track $\tau'$ inherits a measure. If $b$ is a large branch, then one may split $\tau$ at $b$, but in order that the new train track naturally inherits a measure, only one of the three ways of doing this is permitted; this is determined by the weights of the incoming small half-branches. We call this a \emph{measured split or slide}.

A \emph{maximal split} of a measured train track $(\tau, \mu)$ is the operation of simultaneously performing a measured split along all the branches of maximal weight. Note that the branches of maximal weight are necessarily large branches. As any two large branches cannot share any switches, there is no concern about performing all these measured splits simultaneously. 

It is important to note that there is one and only one possible maximal split that can be applied to a measured train track.

The following was proved by Agol~\cite[Theorem 3.5]{Agol:IdealTriangulations}, based on work of Hamenstadt~\cite{Hamenstadt}.

\begin{theorem}[Agol]\label{Thm:PeriodicSplittingSequence}
Let $\phi$ be a pseudo-Anosov homeomorphism of a closed orientable surface $S$. Let $(\tau, \mu)$ be a measured train track that fully carries the stable measured lamination of $\phi$. Then there exists a sequence of measured train tracks $(\tau, \mu) = (\tau_0,\mu_0), (\tau_1,\mu_1), \dots$, each obtained from the predecessor by a maximal split, such that the sequence is eventually periodic. That is, there exist $n \geq 0, m>0$ such that
$\tau_{n+m} = \phi(\tau_n)$ and $\mu_{n+m} = \lambda^{-1} \mu_n$, with $\lambda$ the dilatation of $\phi$.
\end{theorem}

It was also shown by Agol \cite[Corollary 3.4]{Agol:IdealTriangulations} that if $(\tau', \mu')$ and $(\tau, \mu)$ are measured train track fully carrying the stable measured lamination, then after sufficiently many maximal splits applied to $(\tau', \mu')$, the resulting measured train track ends up in the periodic sequence for $(\tau, \mu)$. This implies that the train tracks in the periodic sequence are all transversely recurrent. This is because there is certainly some measured train track $(\tau', \mu')$ that carries the stable measured lamination and that is transversely recurrent \cite{HarerPenner}. Transverse recurrence is preserved when a slide or split is performed on a train track. Hence, the periodic sequence consists of transversely recurrent train tracks. They are also recurrent, because they fully carry the stable measured lamination. Hence, they are birecurrent.

Note also that the train tracks in Theorem \ref{Thm:PeriodicSplittingSequence} are all filling. For if not, then there is some essential simple closed curve in their complement. This curve therefore lies in the complement of the stable lamination. But it is well known that each complementary component of the stable lamination is simply-connected (see \cite[Lemma 5.3]{CassonBleiler} for instance).

\subsection{Stable translation length}

Using the results in the previous subsections, we will prove the following result, which is presumably well known.

\begin{theorem}[Translation length and stable translation length]\label{Thm:StableTranslationLength}
Let $S$ be a closed orientable surface. Then, there is a constant $k > 0$, depending only on the genus of $S$
and a choice of finite generating set for $\MCG(S)$,
such that for any pseudo-Anosov homeomorphism $\phi$ of $S$,
\[ k \, \ell_{\MCG(S)}(\phi) \leq \overline{\ell}_{\MCG(S)}(\phi) \leq \ell_{\MCG(S)}(\phi). \]
As a consequence, the translation length in the spine graph $\ell_{\mathrm{Sp}(S)}(\phi)$ and the stable translation length in the spine graph $\overline{\ell}_{\mathrm{Sp}(S)}(\phi)$ also lie within a bounded ratio of each other.
\end{theorem}

\begin{proof}
The stable translation length of an isometry is always at most the translation length; see \cite[II.6.6]{BridsonHaefliger}. So we focus on the first inequality. 

By \refprop{SpineGraphMCG}, for each positive integer $v$, $\mathrm{Tr}(S;v)$ and $\mathrm{Sp}(S)$ are quasi-isometric to $\MCG(S)$.
It  therefore suffices to prove that $\overline{\ell}_{\mathrm{Tr}(S;v)}(\phi) \geq k' \, \ell_{\mathrm{Tr}(S;v)}(\phi)$ for some positive integer $v$ and some $k' > 0$
depending only on the genus of $S$.

Let $(\tau,\mu)$ be a measured train track carrying the stable measured lamination for $\phi$. 
Let $(\tau, \mu) = (\tau_0,\mu_0), (\tau_1,\mu_1), \dots$ be the sequence of measured train tracks, each obtained from its predecessor by a maximal split.
By \refthm{PeriodicSplittingSequence}, there are integers $n \geq 0$ and $m > 0$ such that $\tau_{n+m} = \phi(\tau_n)$ and $\mu_{n+m} = \lambda^{-1} \mu_n$.
The train tracks in the periodic sequence are birecurrent and filling.
By forgetting the smoothing information at each vertex, and taking the dual of the resulting graph, each train track determines an element of $\mathrm{Tr}(S;v)$ for $v = 4g-4$, by \reflem{ComplementOfTrainTrack}. Thus, we obtain a sequence $x_0, \dots, x_{n+m}$ in $\mathrm{Tr}(S;v)$ where $x_{n+m} = \phi(x_n)$. Each train track is obtained from its predecessor by a maximal split, which is a composition of splits, each of which occurs at a branch, dual to an edge. By \reflem{NumberTrianglesAndEdges}, there are at most $18g(S)-18$ edges in the triangulation. Hence this is a composition of at most $18g(S) - 18$ splits.

From $\tau_n$ and beyond, no central splits are performed, because any central split reduces the number of complementary regions by $1$, whereas the equality $\tau_{n+m} = \phi(\tau_n)$ implies that the number of complementary regions is unchanged. Hence, for $i \geq n$, $x_{i+1}$ is obtained from $x_i$ by at most $18g(S)-18$ Pachner moves.
If we continue this sequence by setting $x_{m+i} = \phi(x_i)$, we can obtain a sequence of elements of $\mathrm{Tr}(S;v)$ of length $N m$, say, joining $x_n$ to $\phi^N(x_n)$. By \refthm{TrainTrackQuasiGeodesic}, this is a $Q$-quasi-geodesic, for a constant $Q$ that depends only on the genus of $S$. So, $d(x_n, \phi^N(x_n)) \geq (N m /Q) - Q$. So, letting $N$ tend to infinity, we deduce that the stable translation length satisfies
\[ \overline{\ell}_{\mathrm{Tr}(S;v)}(\phi) \geq m /Q. \]
But $\ell_{\mathrm{Tr}(S;v)}(\phi) \leq (18g(S) - 18) m$, and so setting $k' = (18g(S) - 18)^{-1}Q^{-1}$, we have proved
\[ \overline{\ell}_{\mathrm{Tr}(S;v)}(\phi) \geq k' \ell_{\mathrm{Tr}(S;v)}(\phi). \qedhere
\]
\end{proof}

Note that the integer $m$ in the previous proof is calculable, as follows. By \refthm{ConstructibleTrainTrack}, $(\tau, \mu)$ may be constructed from an expression of $\phi$ as a product of standard generators of $\MCG(S)$.
We start with the train track $(\tau, \mu) = (\tau_0, \mu_0)$ and then start to perform maximal splits, creating a sequence of measured train tracks $(\tau_i, \mu_i)$. Each time we create a new measured train track $(\tau_i, \mu_i)$, we search for all possible homeomorphisms $h \colon S \rightarrow S$ taking $\tau_i$ to some earlier $\tau_j$. This is clearly algorithmically possible, because $\tau_i$ and $\tau_j$ induce cell structures on $S$, and the search for $h$ is the search for homeomorphisms that preserves this cell structure and that respects the smoothing information at each switch. Hence, up to cell-preserving isotopy, there are only finitely many such $h$. For each one, we check whether $h$ is isotopic to our given homeomorphism $\phi$ and whether its action on the measure $\mu$ is just multiplication by $\lambda^{-1}$. According to \refthm{PeriodicSplittingSequence}, for some train tracks $(\tau_{n+m}, \mu_{n+m})$ and $(\tau_n, \mu_n)$ in our sequence, such a homeomorphism $h$ with these properties will eventually be found. Since $m$ is therefore calculable, $\overline{\ell}_{\mathrm{Tr}(S;v)}(\phi)$ and ${\ell}_{\mathrm{Tr}(S;v)}(\phi)$ are also calculable to within a bounded factor. As $\mathrm{Tr}(S;v)$ and $\MCG(S)$ are quasi-isometric, this also implies that $\overline{\ell}_{\MCG(S)}(\phi)$ and $\ell_{\MCG(S)}(\phi)$ are also calculable to within a bounded factor.
Unfortunately, this factor itself is not known to be calculable.

\section{Normal and almost normal surfaces}\label{Sec:NormAlmostNorm}

\textbf{Road map:} For $M$ the mapping torus of $\phi$, we are trying to prove that there exist constants $A$, $B$ depending only on the genus of $S$ such that $\Delta(M)\geq A\, \ell_{\Sp(S)}(\phi) + B$. The idea of the proof will be to cut $M$ along a (nice) fibre $S$ into $S\times[0,1]$, then transfer a spine $\Gamma$ from $S\times\{0\}$ to a spine $\Gamma'$ from $S\times\{1\}$, and then use \reflem{SpinesSameComplex} to bound the distance in $\mathrm{Sp}(S)$ between $\phi(\Gamma)$ and $\Gamma'$. Of course, $S \times [0,1]$ is a product and so we could simply isotope $\Gamma$ onto $\Gamma'$, but we will retain greater control over these spines by ensuring that they are subcomplexes of suitable cell structures on $S$. We will transfer the spine in $S \times \{ 0 \}$ to the spine in $S \times \{ 1 \}$ by passing it along a sequence of normal surfaces, almost normal surfaces, and surfaces interpolating between them. This section reviews definitions and results from normal and almost normal surface theory.

Throughout, $M$ will be a compact orientable 3-manifold with a triangulation $\calT$. To simplify later arguments, we will also introduce a handle structure  on $M$ dual to $\calT$. We will therefore extend results on normal and almost normal surfaces in triangulations to those in a handle structure $\mathcal{H}$.

\begin{definition}[Normal surface, triangulation]\label{Def:NormalSurface}
An arc properly embedded in a 2-simplex is \emph{normal} if it is disjoint from the vertices and has endpoints on distinct edges.

A disc properly embedded in a tetrahedron is a \emph{triangle} if its boundary consists of three normal arcs. It is a \emph{square} if its boundary consists of four normal arcs. See \reffig{NormalAlmostNormal}, left.

A surface properly embedded in $M$ is \emph{normal} if its intersection with each tetrahedron of $\calT$ is a union of disjoint triangles and squares.
\end{definition}

\begin{figure}
  \includegraphics{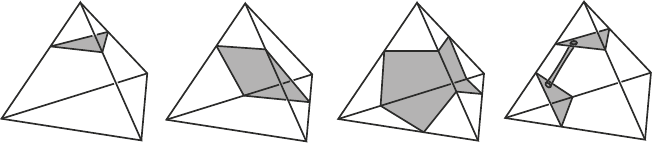}
  \caption{Normal and almost normal pieces, left to right: triangle, square, octagon, tubed piece.}
  \label{Fig:NormalAlmostNormal}
\end{figure}

\begin{definition}[Almost normal surface, triangulation]\label{Def:AlmostNormalPiece}
An \emph{almost normal piece} properly embedded in a tetrahedron is either a disc with boundary equal to eight normal arcs (known as an \emph{octagon}) or an annulus that is obtained from two disjoint normal triangles or squares by attaching a tube that runs parallel to an edge of the tetrahedron (known as a \emph{tubed piece}). See \reffig{NormalAlmostNormal}.

A surface properly embedded in $M$ is \emph{almost normal} if it intersects each tetrahedon in a collection of triangles and squares, except in precisely one tetrahedron, where it is a collection of triangles and squares and exactly one almost normal piece. 
\end{definition}

For many of our results it is actually easier to consider a handle structure arising from a triangulation $\calT$ rather than the triangulation itself.

\begin{remark}[Handle structure]
In this paper, a handle structure on a 3-manifold or a 2-manifold will always satisfy the following conditions:
\begin{enumerate}
\item each $i$-handle $D^i \times D^{k-i}$ ($k = 2$ or $3$) intersects the handles of lower index in $\partial D^i \times D^{k-i}$;
\item any two $i$-handles are disjoint;
\item in the case of a 3-manifold, the intersection of any $1$-handle $D^1 \times D^2$ with any 2-handle $D^2 \times D^1$ is of the form $D^1 \times \alpha$ in $D^1 \times D^2$ where $\alpha$ is a collection of arcs in $\partial D^2$, and of the form $\beta \times D^1$ in $D^2 \times D^1$ where $\beta$ is a collection of arcs in $\partial D^2$;
\item any 2-handle runs over at least one 1-handle.
\end{enumerate}
\end{remark}

\begin{definition}[Associated handle structure on a 0-handle, or $\bdy M$]
Let $\mathcal{H}$ be a handle structure of a 3-manifold $M$. Suppose each 0-handle has connected intersection with the union of the 1-handles and the 2-handles.
\begin{enumerate}
\item The boundary of each 0-handle inherits an \emph{associated handle structure}. The 0-handles of this structure are the components of intersection between the 0-handles and the 1-handles. The 1-handles of the associated handle structure are the components of intersection between the 0-handles and the 2-handles. The 2-handles of the associated handle structure are the remainder of the boundary of the 0-handles, which are discs by assumption.
\item The boundary of the 3-manifold also inherits a handle structure, where the $i$-handles are the components of intersection between $\partial M$ and the $i$-handles of $\mathcal{H}$.
\end{enumerate}
\end{definition}

\begin{definition}[Dual handle structure to triangulation]
Given a triangulation $\calT$ of a compact 3-manifold $M$, there is a handle structure for $M$ with a thin open collar neighbourhood of $\partial M$ removed.
This is called the \emph{dual handle structure} and is obtained as follows. Each tetrahedron gives rise to a 0-handle, which is obtained from the tetrahedron by removing a thin open regular neighbourhood of its boundary. Each face of $\calT$ not lying wholly in $\partial M$ gives rise to 1-handle, which we take to be a thin open regular neighbourhood of the face with a thin open regular neighbourhood of the edges of $\calT$ removed. Each edge of $\calT$ not lying in $\partial M$ forms a 2-handle, by taking a regular neighbourhood of the edge and removing a small open regular neighbourhood of the endpoints of the edge. Finally, regular neighbourhoods of the vertices in the interior of $M$ correspond to 3-handles. See \reffig{TetrToHandleStruct}.
\end{definition}

\begin{figure}
  \includegraphics{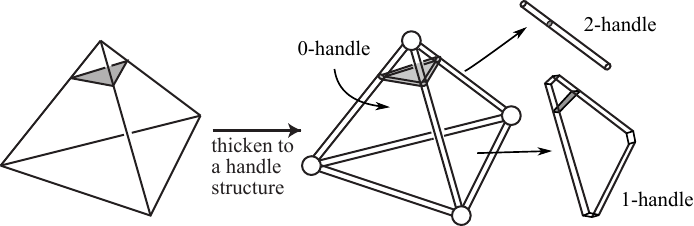}
  \caption{A handle structure arising from a triangulation. Shown also is a shaded surface that respects the handle structure.}
  \label{Fig:TetrToHandleStruct}
\end{figure}

Consider a normal surface in a triangulation, intersecting the tetrahedra in triangles and squares. When we cut along the normal surface, tetrahedra are not necessarily split into tetrahedra; the cut also creates pieces with quadrilateral faces, or parallel triangle faces. However, normal surfaces can be defined in handle structures, and cutting along a normal surface in a handle structure will still give rise to a handle structure.
The boundaries of the 0-handles will inherit handle structures, such those shown in \reffig{PreTetrahedral}.
See also \reffig{SliceTetrahedron}. We now define these more formally.

\begin{figure}
  \includegraphics{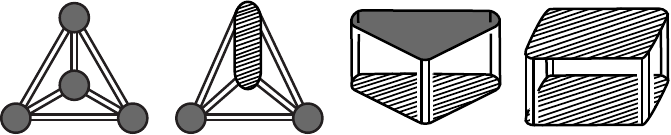}
  \caption{Pre-tetrahedral handle structures. Left to right: tetrahedral, semi-tetrahedral, product handle of length three, parallelity handle of length four. The shaded handles arise either from 3-handles or $\bdy M$. The striped arise from $\bdy M$ only. }
  \label{Fig:PreTetrahedral}
\end{figure}

\begin{definition}[Pre-tetrahedral handle structure] 
Let $M$ be a compact 3-manifold, possibly with boundary, with a handle structure $\calH$. We will define various types of 0-handle of $\calH$. Each 0-handle will have connected intersection with the union of the 1-handles and 2-handles and so its boundary will have an associated handle structure. Let $H_0$ be a 0-handle of $\calH$. 
\begin{enumerate}
\item The handle $H_0$ is \emph{tetrahedral} if $\bdy H_0$ inherits a handle structure as shown on \reffig{PreTetrahedral} left; that is, it inherits the same handle structure as the boundary of a 0-handle coming from a tetrahedron. The 2-handles in $\bdy H_0$, which are shaded in the figure, may arise either as intersections of $\bdy H_0$ with 3-handles of $\calH$, or as intersections of $\bdy H_0$ with $\bdy M$.

\item $H_0$ is \emph{semi-tetrahedral} if the intersection of $\bdy H_0$ with 1-handles, 2-handles, 3-handles, and $\bdy M$ is as shown in \reffig{PreTetrahedral} second from left; that is, it has the handle structure of the boundary of a thickened tetrahedron cut along a single square. Two of the 2-handles, shaded in the figure, may arise either as intersections of $\bdy H_0$ with 3-handles or with $\bdy M$. The third 2-handle, which meets four 1-handles, arises only as a component of intersection between $\bdy H_0$ and $\bdy M$.

\item $H_0$ is a \emph{product handle of length 3}, if it is as shown in the middle right of \reffig{PreTetrahedral}. That is, $\bdy H_0$ meets exactly three 1-handles of $\calH$ and exactly three 2-handles, connected in a cycle. Furthermore, it is a \emph{parallelity handle of length 3} if it meets $\bdy M$ only and no 3-handles.

\item $H_0$ is a \emph{parallelity handle of length 4} if it is as shown in \reffig{PreTetrahedral} right; that is it intersects four 1-handles and 2-handles of $\calH$ in a cycle, and meets $\bdy M$ on either side of the cycle.
\end{enumerate}
Parallelity handles arise, for example, when a 0-handle dual to a tetrahedron is cut along parallel normal triangles or squares. Again see \reffig{SliceTetrahedron}, and \reflem{ComplexityDecomposition} below.

We say that the handle structure $\calH$ is \emph{tetrahedral} if each of its 0-handles is tetrahedral. We say it is \emph{pre-tetrahedral} if each of its 0-handles is tetrahedral, semi-tetrahedral, a product handle or a parallelity handle, as above.
\end{definition}

\begin{figure}
  \includegraphics{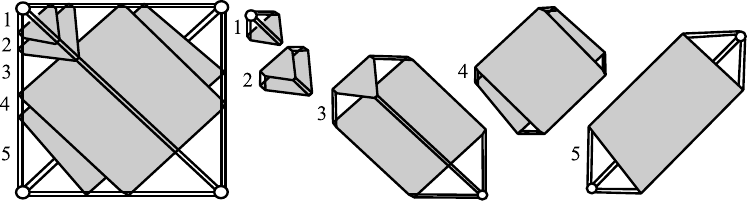}
  \caption{Shows a tetrahedral handle cut along some normal discs, and the resulting handles. 1 is a product handle of length three; 2 is a parallelity handle of length three; 3 is semi-tetrahedral; 4 is a parallelity handle of length four; and 5 is semi-tetrahedral. }
  \label{Fig:SliceTetrahedron}
\end{figure}

The following definitions give analogues to \refdef{NormalSurface}.

\begin{definition}[Surface respects a handle structure]
A closed surface $S$ embedded in $M$ \emph{respects} a handle structure $\mathcal{H}$ of $M$ if
\begin{enumerate}
\item its intersection with each 1-handle $D^1 \times D^2$ is of the form $D^1 \times \alpha$, where $\alpha$ is a properly embedded 1-manifold in $D^2$;
\item it intersects each 2-handle $D^2 \times D^1$ in a collection of discs of the form $D^2 \times p$, where $p$ is a point in the interior of $D^1$;
\item it is disjoint from the 3-handles.
\end{enumerate}
\end{definition}

Any surface properly embedded in $M$ that is in general position with respect to a triangulation $\calT$ gives rise to a surface that respects the dual handle structure $\mathcal{H}$ and vice versa.

\begin{definition}[Normal, handle structure]\label{Def:NormalInHandleStructure}
If a closed surface $S$ embedded in $M$ respects a handle structure $\calH$ and, in addition, it intersects each 0-handle and each 1-handle in a collection of properly embedded discs, it is \emph{standard}. If, furthermore, $S$ has the property that for each 0-handle $H_0$ of $\mathcal{H}$, each component of $S \cap H_0$ runs over any component of intersection between $H_0$ and the 2-handles in at most one arc, then $S$ is called \emph{normal with respect to $\mathcal{H}$}.
\end{definition}

Observe that a surface that is normal is also standard, and one that is standard also respects the handle structure. See \reffig{RespectsHandleStructure}.

\begin{figure}
  \includegraphics{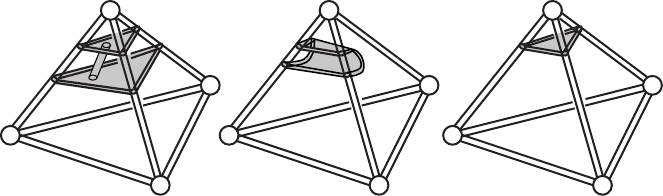}
  \caption{Left to right: A surface that respects a handle structure but is not standard; a surface that is standard but not normal; a surface that is normal.}
  \label{Fig:RespectsHandleStructure}
\end{figure}

We also define a notion of standard for 1-manifolds in surfaces with handle structures, as follows. 

\begin{definition}[Standard 1-manifold]
Let $S$ be a compact surface with a handle structure $\mathcal{H}$. Then a 1-manifold properly embedded in $S$ is \emph{standard} if 
\begin{enumerate}
\item it intersects each 0-handle in a collection of properly embedded arcs;
\item it intersects each 1-handle in a collection of arcs, each of which runs parallel to the core of the 1-handle and respects its product structure; and
\item it is disjoint from the 2-handles.
\end{enumerate}
\end{definition}

\begin{definition}[Normal 1-manifold]
Let $S$ be a compact surface with a handle structure $\mathcal{H}$. An arc properly embedded in a 0-handle $H_0$ is \emph{normal} if its endpoints
lie in distinct components of intersection between $H_0$ and the 1-handles of $\mathcal{H}$. A closed 1-manifold properly embedded in $S$ is \emph{normal} if it is standard and each arc of intersection with each 0-handle is normal.
\end{definition}

\begin{definition}[Triangles and squares in a 0-handle]\label{Def:TriangleSquareHandle}
  Let $\mathcal{H}$ be a handle structure of a 3-manifold $M$. Let $S$ be a closed normal surface properly embedded in $M$. Then a component $F$ of intersection between a 0-handle $H_0$ of $\mathcal{H}$ and $S$ is a \emph{triangle} (respectively, \emph{square}) if it is a disc and $\bdy F$ is a normal curve that runs over the 1-handles of $\bdy H_0$ exactly three (respectively, four) times.
\end{definition}

Recall that the 1-handles on the boundary of a 0-handle $H_0$ arise from 2-handles of $\calH$. When $\mathcal{H}$ is dual to a triangulation, its 2-handles can be viewed as thickened edges of this triangulation. Thus \refdef{TriangleSquareHandle} is completely analogous to \refdef{NormalSurface}.
The following result is easily checked.

\begin{lemma}[Normal surface meets 0-handles in triangles and squares]
Let $\mathcal{H}$ be a pre-tetrahedral handle structure of a 3-manifold $M$. Let $S$ be a closed normal surface properly embedded in $M$. Then any component of intersection between $S$ and a 0-handle of $\mathcal{H}$ is a triangle or square. \qed
\end{lemma}

We may define the complexity of a handle structure by analogy with that of a triangulation. 

\begin{definition}[Tetrahedral complexity, handle structure]\label{Def:Complexity}
Let $\calH$ be a pre-tetrahedral handle structure of a 3-manifold $M$, and let $H_0$ be a 0-handle of $\calH$. Let $\alpha$ denote the number of intersections of $H_0$ with the 3-handles of $\calH$. Define the \emph{complexity} of $H_0$ to be $(\alpha/8) + \beta$, where
\begin{itemize}
\item $\beta = 1/2$ if $H_0$ is tetrahedral,
\item $\beta = 1/4$ if $H_0$ is semi-tetrahedral, and
\item $\beta=0$ if $H_0$ is a product handle or parallelity handle.
\end{itemize}
Define the \emph{tetrahedral complexity} $\Delta(\mathcal{H})$ to be the sum of the complexities of its 0-handles.
\end{definition}

When $\mathcal{H}$ is dual to a triangulation $\calT$ of a closed 3-manifold $M$, then $\Delta(\mathcal{H}) = \Delta(\calT)$. This is because every 0-handle of $\mathcal{H}$ is tetrahedral and has four components of intersection with the 3-handles.

\begin{lemma}[Complexity unchanged under cutting along normal surface]\label{Lem:ComplexityDecomposition}
Let $\mathcal{H}$ be a pre-tetrahedral handle structure of a 3-manifold. Let $S$ be a closed normal surface properly embedded in $M$. Then $M \cut S$ inherits a pre-tetrahedral handle structure $\mathcal{H}'$ with $\Delta(\mathcal{H}') = \Delta(\mathcal{H})$.
\end{lemma}

\begin{proof}
Each tetrahedral 0-handle of $\mathcal{H}$ is decomposed into one of the following:
\begin{enumerate}
\item a single tetrahedral 0-handle plus possibly some parallelity and product handles; 
\item two semi-tetrahedral 0-handles plus possibly some parallelity and product handles.
\end{enumerate}
A semi-tetrahedral 0-handle is decomposed into exactly one semi-tetrahedral 0-handle plus possibly some parallelity and product handles. A product or parallelity 0-handle is decomposed into product and parallelity 0-handles. In every case, the components of intersection between the 0-handle and the 3-handles are shared out among the resulting 0-handles of $\mathcal{H}'$.
\end{proof}

Analogous to \refdef{AlmostNormalPiece}, we also define almost normal surfaces in pre-tetrahedral handle structures.

\begin{definition}[Almost normal surface, handle structure]
Let $\mathcal{H}$ be a pre-tetrahedral handle structure of a 3-manifold. 
\begin{enumerate}
\item
An \emph{octagon} is a disc properly embedded in a 0-handle $H_0$ of $\mathcal{H}$ with boundary that is a normal curve that runs over eight 1-handles and eight 0-handles of $\bdy H_0$.
\item
A \emph{tubed piece} is an annulus in a 0-handle $H_0$ obtained by tubing together two disjoint triangles or squares. The 0-handle is required to be tetrahedral or semi-tetrahedral. The tube runs parallel to an arc in $\partial H_0$ that is a co-core of a 1-handle of $\partial H_0$.
\end{enumerate}
An \emph{almost normal piece} is an octagon or tubed piece. 

A closed surface properly embedded in a 3-manifold $M$ with a pre-tetrahedral handle structure $\mathcal{H}$ is \emph{almost normal} if it respects the handle structure, it intersects each 0-handle in a collection of triangles and squares, except in precisely one 0-handle, where it is a collection of triangles and squares and exactly one almost normal piece.
\end{definition}

One motivation for this definition arises from the following lemma. The corresponding statement for triangulations is well known \cite[Proof of Claim~2, Section~2.2]{Stocking}.

\begin{lemma}[Non-normal has length at least eight]\label{Lem:StockingLemma}
Let $\calH$ be a pre-tetrahedral handle structure of a 3-manifold and let $H_0$ be a 0-handle of $\calH$. Let $D$ be a disc properly embedded in $H_0$ with boundary that is a normal curve $C$ in $H_0 \cap (\calH^1 \cup \calH^2)$. Then if $D$ is not a normal disc, then it runs over $\calH^2$ at least 8 times. Furthermore, each component of $\partial H_0 \cut C$ intersects some component of $H_0 \cap \calH^2$ at least twice.
\end{lemma}

\begin{proof}
We say that each component of $\partial H_0 \cut (\calH^1 \cup \calH^2)$ is a \emph{hole}. Since $\calH$ is pre-tetrahedral, there are at most $4$ holes. Note that $C$ is disjoint from the holes. The curve $C$ divides $\partial H_0$ into 2 discs, $D_1$ and $D_2$. The proof divides according to the number of holes in each disc.

Suppose first that one of these discs $D_i$ contains no holes. It therefore inherits a handle structure from $H_0 \cap (\calH^1 \cup \calH^2)$, as a union of 0-handles and 1-handles. Since $D_i$ is a disc, these 0-handles and 1-handles form a thickened tree. A leaf of the tree is a 0-handle of $D_i$ that is incident to a single 1-handle. But this implies that an arc of $C \cap H_0 \cap \calH^1$ has endpoints on the same component of $H^0 \cap \calH^2$, contradicting the assumption that $C$ is normal in $\partial H_0$.

Now suppose that one disc $D_i$ contains a single hole. Then, removing this component from $D_i$ gives an annulus, which again admits a handle structure with 0-handles and 1-handles. Again, this is a thickened graph. As the graph has no leaf, we deduce that the graph consists of a single cycle. Hence, $C$ is normally parallel in $\partial H_0 \cap (\calH^1 \cup \calH^2)$ to a boundary component of $\partial H_0 \cap (\calH^1 \cup \calH^2)$. This implies that $C$ has length three or four and hence $D$ is normal.

Thus, we are reduced to the case where $D_1$ and $D_2$ each contain $2$ holes. This implies that there are $4$ holes in total, and hence $H_0$ is tetrahedral. There are six components of $H_0 \cap \calH^2$, each one joining two distinct holes. Consider any pair of holes. If $C$ separates these holes, then $C$ must run over the component of $H_0 \cap \calH^2$ incident to the two holes an odd number of times. Because $C$ separates each hole from two others, in particular $C$ runs over four components of $H_0 \cap \calH^2$ at least once, hence at least four times in total. Now consider a pair of holes not separated by $C$, and consider the component of $H_0 \cap \calH^2$ incident to them. We claim that this component of $H_0 \cap \calH^2$ is not disjoint from $C$. If it were, then this component of $H_0 \cap \calH^2$ would lie in some $D_i$. Removing the two holes in $D_i$ from $D_i$ gives a pair of pants $P$. Again this has a handle structure, which is a thickened graph. Each vertex of this graph has degree at most $3$, since each component of $H_0 \cap \calH^1$ is incident to at most three components of $H_0 \cap \calH^2$. This graph is therefore topologically either a $\theta$-graph or an eyeglass. In the case where $P$ is a thickened $\theta$-graph, the total length of $C$ is equal to the sum of the lengths of the other two components of $\partial P$ minus $2$, where the latter subtraction is due to a component of $P \cap \calH^2$ being incident to both components of $\partial P - C$. Each component of $\partial P - C$ is the boundary of a hole and so has length $3$. Hence, in this case, $C$ has length $4$ and so is a square, and in particular $D$ is normal. When $P$ is a thickened eyeglass, there is some component of $P \cap \calH^2$ incident to a single component of $\partial P$. This component of $\partial P$ must be $C$, since no component of $H_0 \cap \calH^2$ runs over the same hole twice. But then $C$ is incident to every handle in the handle structure on $P$, and so there is no component of $H_0 \cap \calH^2 \cap P$ disjoint from $C$, which contradicts our assumption.

As $C$ runs over four components of $H_0 \cap \calH^2$ at least once each, and $C$ runs over the remaining two components of $H_0 \cap \calH^2$ at least twice each, we deduce that the total length of $C$ is at least eight. 

Now consider any component $D_i$ of $\partial H_0 \cut C$. In the case where $D$ is not normal, we have shown that $D_i$ must contain two holes. Consider the component of $H_0 \cap \calH^2$ joining these two holes. We have shown that $C$ runs over this component of $H_0 \cap \calH^2$ at least twice. Hence, $D_i$ intersects this component of $H_0 \cap \calH^2$ at least twice, as required.
\end{proof}

\begin{definition}[Handle structure, cutting along almost normal]\label{Def:HandleStructCutAlmostNormal}
Let $M$ be a 3-manifold with a pre-tetrahedral handle structure. Suppose $S$ is almost normal. Then in the case where $S$ has an octagonal piece, $M \cut S$ inherits a handle structure, where each $i$-handle is a component of intersection between $M \cut S$ and an $i$-handle of $M$. In the case where $S$ has a tubed piece, this does not quite work because one component of intersection between $M \cut S$ and a 0-handle of $M$ is a solid torus $V$. This solid torus is naturally the union of a 0-handle and a 1-handle as follows. 
There is a disc $D$ in the 0-handle $H_0$ containing the tubed piece such that $D \cap S$ is an arc in $\partial D$ running over the tubed piece, and the remainder of $\partial D$ is part of the co-core of a 1-handle of $\partial H_0$. Then a regular neighbourhood of $D$ is the 1-handle of $V$, and the remainder of $V$ is the 0-handle. The remaining handles of the handle structure are defined in the same way as previously, where each $i$-handle is a component of intersection between $M \cut S$ and an $i$-handle of $M$. We say that this handle structure is the one that $M \cut S$ \emph{inherits} from $M$. Observe that the handle structure is not necessarily pre-tetrahedral anymore.
\end{definition}

We now recall a few definitions and results from normal surface theory, and their analogues for handle structures. 

\begin{definition}[Weight]\label{Def:Weight}
If $S$ is a surface properly embedded in $M$, in general position with respect to a triangulation $\calT$, then its \emph{weight} is defined to be the number of intersections between $S$ and the edges of $\calT$.

If $S$ is a surface properly embedded in $M$ that respects a handle structure $\calH$, define its \emph{weight} to be the number of components of intersection between $S$ and the 2-handles of $\mathcal{H}$. 
\end{definition}

\begin{definition}[Topologically parallel surfaces]
  Two disjoint closed surfaces $S_0$ and $S_1$ properly embedded in $M$ are \emph{topologically parallel} if there is an embedding of $S \times [0,1]$ in $M$ such that $S_0 = S \times \{ 0 \}$ and $S_1 = S \times \{ 1 \}$.
\end{definition}

When we are dealing with surfaces that respect a triangulation or a handle structure of $M$, there is a stronger notion of parallel surfaces, as follows.

\begin{definition}[Normally parallel surfaces, tetrahedra]
For a triangulation $\calT$ of $M$, a \emph{normal isotopy} of $M$ is an isotopy $F \colon M \times [0,1] \rightarrow M$ such that,
for each $t \in [0,1]$, the map $F_t \from M \times \{ t \} \to M$ is a homeomorphism that preserves each simplex of $\calT$ and such that $F_0$ is the identity on $M$. We say that two surfaces $D_0$ and $D_1$ properly embedded in a tetrahedron of $\calT$ have the same \emph{type} if there is a normal isotopy taking one to the other. More specifically, there is a normal isotopy $F \from M \times [0,1] \to M$ such that $F_1(D_0) = D_1$.
We say that $D_0$ and $D_1$ are \emph{normally parallel} if there is a normal isotopy $F \from M \times [0,1] \to M$ taking $D_0$ to $D_1$ such that the restriction of $F$ to $D_0 \times [0,1]$ is an embedding into $M$. Similarly, disjoint surfaces $S_0$ and $S_1$ properly embedded in $M$ are \emph{normally parallel} if there is a normal isotopy $F \from M \times [0,1] \to M$ taking $S_0$ to $S_1$, such that the restriction of $F$ to $S_0 \times [0,1]$ is an embedding into $M$.
\end{definition}

In \reffig{SameDiscType}, left, two discs properly embedded within a tetrahedron are shown. They are of the same type but are not normally parallel.

\begin{figure}
  \includegraphics{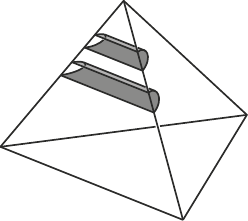}
  \hspace{.5in}
\includegraphics{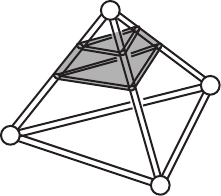}
  \caption{Left: Discs that are of the same type but that are not normally parallel. Right: Normally parallel disks in a handle structure. The region between is the union of parallelity regions in 0, 1, and 2-handles.}
  \label{Fig:SameDiscType}
\end{figure}

\begin{definition}[Normally parallel surfaces, handle structure]
Let $\mathcal{H}$ be a handle structure of a 3-manifold $M$. Let $S$ be a properly embedded surface that  respects $\mathcal{H}$. An \emph{elementary piece} of $S$ is a component of intersection between $S$ and some handle of $H$.

Two elementary pieces $D_0$ and $D_1$ are \emph{normally parallel} if there is an isotopy $F \from M \times [0,1] \to M$ that preserves all the handles of $\mathcal{H}$ and sends $D_0$ to $D_1$, and such that the restriction of $F$ to $D_0 \times [0,1]$ is an embedding into $M$.

Two disjoint surfaces $S_0$ and $S_1$ properly embedded in $M$ are \emph{normally parallel} if there is an isotopy $F \from M \times [0,1] \to M$ taking $S_0$ to $S_1$, preserving all the handles of $\calH$, such that the restriction of $F$ to $S_0 \times [0,1]$ is an embedding into $M$.
\end{definition}

Between normally parallel elementary pieces lie parallelity regions, defined as follows. 

\begin{definition}[Parallelity regions]\label{Def:ParallelityRegions}
Let $S$ be a surface properly embedded in $M$ that respects the handle structure $\mathcal{H}$.
For any handle $H$ of $\mathcal{H}$, a component of $H \cut S$ is a \emph{parallelity region} if it lies between two normally parallel elementary pieces $D_0$ and $D_1$ of $S \cap H$.  There is an identification between this region and $D_0 \times [0,1]$, where $D_0 = D_0 \times \{ 0 \}$
and $D_1 = D_0 \times \{ 1 \}$. The \emph{projection map} for this region is the map $D_0 \times [0,1] \rightarrow D_0$ onto the first factor. See \reffig{SameDiscType}, right. 
\end{definition}

\begin{theorem}[Stocking's theorem]\label{Thm:AlmostNormalBetweenNormal}
Let $M$ be a compact orientable 3-manifold equipped with a pre-tetrahedral handle structure $\calH$. 
Let $S_0$ and $S_1$ be disjoint normal closed connected surfaces embedded in $M$ that are not 2-spheres and that are topologically parallel but not normally parallel. Then there is an almost normal surface between them that is topologically parallel to each of them.
\end{theorem}

The proof of \refthm{AlmostNormalBetweenNormal} is essentially contained within the arguments of \cite{Stocking}.
There, Stocking considers a closed triangulated orientable 3-manifold, and shows that any strongly irreducible Heegaard surface for the manifold may be isotoped into almost normal form. 
The argument also applies to compact orientable 3-manifolds with boundary. In that situation, the right structure to use is more general than a triangulation. We call it a generalised triangulation, and define it as follows.

If a tetrahedron is cut along a union of disjoint triangles and squares, each resulting component is called a \emph{sub-tetrahedron}. It inherits a cell structure, with a single 3-cell, and where the 2-cells arise as components of intersection with the faces of the original tetrahedron and as copies of the triangles and squares. We call the latter type of 2-cells \emph{$\partial$-faces}. 

A \emph{generalised triangulation} for a compact 3-manifold $M$ is a cell structure for $M$ with the following properties. Each 3-cell is a copy of a sub-tetrahedron. The $\partial$-faces of these sub-tetrahedra form the 2-cells in $\partial M$. The other faces of these sub-tetrahedra are identified in pairs so that they form 2-cells with interior lying in the interior of $M$. Generalised triangulations are exactly the dual concept to a pre-tetrahedral handle structure.

Many of the familiar concepts for normal surfaces in triangulated 3-manifolds apply equally well to 3-manifolds with a generalised triangulation. One can speak of a triangle, square, octagon or tubed piece in a sub-tetrahedron, as long as one requires that these are disjoint from the $\partial$-faces. One can also therefore speak of a closed normal or almost normal surface. Furthermore, a closed surface in a generalised triangulation is normal if and only if it is normal in the dual pre-tetrahedral handle structure. The same is true of almost normal surfaces, except that in the case of a handle structure, we require that a tubed piece does not lie in a product or parallelity handle.

\begin{proof}[Proof of \refthm{AlmostNormalBetweenNormal}]
Stocking's argument for closed triangulated 3-manifolds applies just as well to compact orientable 3-manifolds with a generalised triangulation (see \cite[Section 6]{Stocking}).  In fact, in her argument, one cuts along closed normal surfaces and considers the complementary pieces. These inherit a generalised triangulation, which is what Stocking works with.

We are given a pre-tetrahedral handle structure $\calH$. Cutting this along $S_0$ and $S_1$, we obtain pre-tetrahedral handle structures on each of the pieces. In particular, the copy of $S \times [0,1]$ between $S_0$ and $S_1$ inherits such a handle structure $\mathcal{H}'$. Let $\calT$ be the generalised triangulation dual to $\mathcal{H}'$. Let $\calT^1$ be the union of the 1-cells of $\calT$ not lying in $S \times \{0 ,1 \}$. Now $S \times [0,1]$ admits a Heegaard splitting where $S \times \{ 1/2 \}$ is the Heegaard surface, dividing the manifold into two $I$-bundles. Thus, we can apply Stocking's methods to this splitting. 

Stocking's argument started by finding within $S \times [0,1]$ a maximal collection $\mathcal{S}$ of disjoint properly embedded normal 2-spheres, no two of which are normally parallel. The proof divides into the case where $\mathcal{S}$ is empty or non-empty. 

If $\mathcal{S}$ is non-empty, then one component $M'$ of $(S \times [0,1]) \cut \mathcal{S}$ is homeomorphic to $S \times [0,1]$ with a single open 3-ball removed (\cite[Lemma~2]{Thompson:3Sphere}).
No vertices of $\calT$ lie in the interior of $M'$, because the link of each such vertex is a normal 2-sphere that can be made disjoint from all other normal 2-spheres and that is therefore parallel to a sphere in $\calS$.
Hence the intersection between $\calT^1$ and $M'$ is a disjoint union of properly embedded arcs. 
Since one obtains a cell structure on $M'$ by starting with $\partial M' \cup (M' \cap \calT^1)$ and adding only 2-cells and 3-cells, we deduce that there must be some arc of $\calT^1 \cap M'$ that joins $S \times \{ 0 \}$ to the boundary of the ball. If we attach a tube that follows this arc, we obtain a tubed almost normal surface, as required.

So suppose that $\mathcal{S}$ is empty. One can view the product projection map $S \times [0,1] \rightarrow [0,1]$ as a height function. One can then place $\calT^1$ into thin position with respect to this height function. This is defined as follows. For each level $S \times \{ t \}$, where $t \in [0,1]$, one considers the number of intersections between $\calT^1$ and $S \times \{ t \}$, and this gives a function $[0,1] \rightarrow \mathbb{N}$. Since $S \times \{ 0 \}$ and $S \times \{1 \}$ are normal, $\{ 0 \}$ and $\{ 1 \}$ are local minima for this function. Since $S \times \{ 0 \}$ and $S \times \{1 \}$ are not normally parallel, the function is not constant. Hence, there is some local maximum realised by a surface $S \times \{ t \}$ that is in general position with respect to the 1-skeleton. It is shown (see \cite[Lemma~5]{Stocking}) that one may isotope and compress $S\times\{t\}$, keeping its intersection with $\calT^1$ unchanged, taking it to an octagonal almost normal surface. However, since $S \times \{ t \}$ is, in our case, incompressible, and $S \times [0,1]$ is irreducible, this implies that no compressions are required. Thus, this almost normal surface is topologically parallel to $S_0$ and $S_1$.

We now convert this almost normal surface in $\calT$ to one in $\mathcal{H}'$. It will then be almost normal in $\mathcal{H}$. The only extra condition that needs to be ensured is that if the almost normal surface contains a tubed piece, then this must miss the product and parallelity handles in $\mathcal{H}'$. Now, the union of the product and parallelity handles and any incident 1-handles and 2-handles is an $I$-bundle. The tube is vertical in this $I$-bundle. When the tube is compressed, the resulting normal surfaces are horizontal in the $I$-bundle. Note that this $I$-bundle is not all of $S \times [0,1]$ because then $\mathcal{H}'$ would consist entirely of product and parallelity handles and so $S_0$ and $S_1$ would be normally parallel. Thus, we can slide the tube, pulling it away from the $I$-bundle into an adjacent 0-handle that is not a product or parallelity handle. The resulting surface is then almost normal in $\mathcal{H}$. 
\end{proof}

A useful feature of pre-tetrahedral handle structures that is not always true of triangulated 3-manifolds is the following. Its proof is immediate.

\begin{lemma}[Boundary surface is normal]\label{Lem:NormalBoundary}
Let $\mathcal{H}$ be a pre-tetrahedral handle structure of a compact 3-manifold $M$. When $\partial M$ is pushed a little into the interior of $M$, it becomes a normal surface. \qed
\end{lemma}

As a consequence of the above lemma and of \refthm{AlmostNormalBetweenNormal}, we have the following theorem.

\begin{theorem}[Almost normal surface exists]\label{Thm:AlmostNormalExists}
Let $\mathcal{H}$ be a pre-tetrahedral handle structure of $S \times [0,1]$, where $S$ is a closed orientable surface, such that $\Delta(\mathcal{H}) > 0$. Then $\mathcal{H}$ contains an almost normal surface that is isotopic to $S \times \{ 1/2 \}$.
\end{theorem}

\begin{proof}
By \reflem{NormalBoundary}, when $S \times \{ 0 \}$ and $S \times \{ 1 \}$ are pushed a little into the interior of the manifold, they become normal surfaces. They are not normally parallel, because then every handle of $\mathcal{H}$ would be a parallelity handle and hence $\Delta(\mathcal{H})$ would be zero.  Hence, by \refthm{AlmostNormalBetweenNormal}, there is an almost normal surface between them that is topologically parallel to each.
\end{proof}

\section{Normalising almost normal surfaces}\label{Sec:Normalising}

\textbf{Road map:} We are still trying to build a sequence of surfaces between $S\times\{0\}$ and $S\times\{1\}$ in a manifold $M\cut S \cong S\times[0,1]$. So far, we have defined normal and almost normal surfaces. But we do not yet know enough to transfer spines between such surfaces in a way that will allow us to bound from above the number of edge expansions and contractions. In this section, we introduce nearly normal surfaces and moves between them. These will give us a first sequence of surfaces on which we can transfer spines.

Throughout this section $M$ is a compact orientable irreducible 3-manifold and $S$ is a closed incompressible surface properly embedded in $M$. The moves we describe are easiest to visualise when $M$ is equipped with a triangulation $\calT$. However, we will need to apply them in the case that $M$ admits a pre-tetrahedral handle structure $\calH$. Thus we will define the moves first in the triangulation setting, and then explain how to define them for a handle structure.

\begin{definition}[Edge compression disc, triangulation]\label{Def:EdgeComprDisc}
An \emph{edge compression disc} for $S$ is a disc $D$ lying in a tetrahedron $\Delta$ of $\calT$ such that $\partial D$ is the union of two arcs $\alpha \subset S$ and $\beta\subset \bdy\Delta$, where $\alpha=D\cap S$, $\beta=D\cap \bdy\Delta$ is a sub-arc of an edge of $\Delta$, and $\alpha \cap \beta = \partial \alpha = \partial \beta$. 
\end{definition}

\begin{definition}[Face compression disc, triangulation]\label{Def:FaceComprDisc}
A \emph{face compression disc} for $S$ is a disc $D$ lying in a 2-simplex $F$ such that $\partial D$ is the union of two arcs $\alpha\subset S\cap F$ and $\beta\subset \bdy F$, where $\alpha=D\cap S$, $\beta=D\cap\bdy F$ is a sub-arc of an edge of $F$, and $\alpha\cap\beta = \bdy \alpha = \bdy\beta$. 
\end{definition}

\reffig{EdgeCompressionDisc} left shows an edge or face compression disc.
Observe that by pushing a face compression disc slightly into an adjacent tetrahedron, it becomes an additional instance of an edge compression disc.

\begin{figure}
\begingroup%
  \makeatletter%
  \providecommand\color[2][]{%
    \errmessage{(Inkscape) Color is used for the text in Inkscape, but the package 'color.sty' is not loaded}%
    \renewcommand\color[2][]{}%
  }%
  \providecommand\transparent[1]{%
    \errmessage{(Inkscape) Transparency is used (non-zero) for the text in Inkscape, but the package 'transparent.sty' is not loaded}%
    \renewcommand\transparent[1]{}%
  }%
  \providecommand\rotatebox[2]{#2}%
  \newcommand*\fsize{\dimexpr\f@size pt\relax}%
  \newcommand*\lineheight[1]{\fontsize{\fsize}{#1\fsize}\selectfont}%
  \ifx\svgwidth\undefined%
    \setlength{\unitlength}{359.40577698bp}%
    \ifx\svgscale\undefined%
      \relax%
    \else%
      \setlength{\unitlength}{\unitlength * \real{\svgscale}}%
    \fi%
  \else%
    \setlength{\unitlength}{\svgwidth}%
  \fi%
  \global\let\svgwidth\undefined%
  \global\let\svgscale\undefined%
  \makeatother%
  \begin{picture}(1,0.36091069)%
    \lineheight{1}%
    \setlength\tabcolsep{0pt}%
    \put(0,0){\includegraphics[width=\unitlength,page=1]{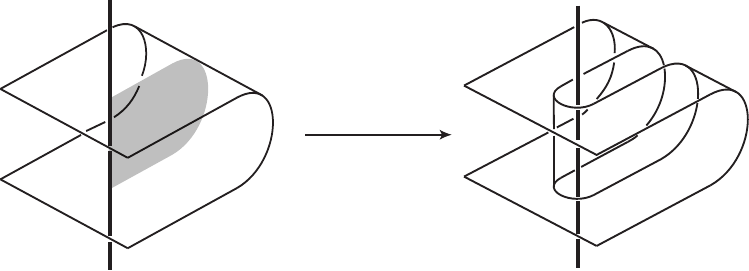}}%
    \put(0.15347692,0.34032268){\color[rgb]{0,0,0}\makebox(0,0)[lt]{\lineheight{0}\smash{\begin{tabular}[t]{l}edge of $\calT$\end{tabular}}}}%
    \put(0.03368948,0.28227899){\color[rgb]{0,0,0}\makebox(0,0)[lt]{\lineheight{0}\smash{\begin{tabular}[t]{l}$S$\end{tabular}}}}%
    \put(0.11670818,0.12271661){\color[rgb]{0,0,0}\makebox(0,0)[lt]{\lineheight{0}\smash{\begin{tabular}[t]{l}$\beta$\end{tabular}}}}%
    \put(0.26231853,0.15894576){\color[rgb]{0,0,0}\makebox(0,0)[lt]{\lineheight{0}\smash{\begin{tabular}[t]{l}$\alpha$\end{tabular}}}}%
    \put(0.45791512,0.19239097){\color[rgb]{0,0,0}\makebox(0,0)[lt]{\lineheight{0}\smash{\begin{tabular}[t]{l}isotopy\end{tabular}}}}%
    \put(0.20851437,0.22292491){\color[rgb]{0,0,0}\makebox(0,0)[lt]{\lineheight{0}\smash{\begin{tabular}[t]{l}$D$\end{tabular}}}}%
    \put(0,0){\includegraphics[width=\unitlength,page=2]{EdgeCompressionDisc.pdf}}%
    \put(0.45386634,0.14784574){\color[rgb]{0,0,0}\makebox(0,0)[lt]{\lineheight{0}\smash{\begin{tabular}[t]{l}along $D$\end{tabular}}}}%
  \end{picture}%
\endgroup%

  \caption{Isotopy along an edge compression disc.}
  \label{Fig:EdgeCompressionDisc}
\end{figure}

\begin{definition}[Isotopy ball, triangulation]\label{Def:IsotopyAlongEdgeCompressionDisc} 
Given an edge or face compression disc $D$ for a properly embedded surface $S$, 
its associated \emph{isotopy ball} $B$ is a small regular neighbourhood of $D$ in $M \cut S$. Thus, $\partial B \cap S$ is a regular neighbourhood of $D \cap S$ in $S$. The remainder of $\partial B$ is obtained by taking two parallel copies of $D$ and attaching a band that goes around the edge of $\calT$ incident to $D$.

An \emph{isotopy along $D$} is the isotopy that moves $S$ across this ball. When $S$ is transversely oriented, we require that this transverse orientation points towards $D$ and hence that the isotopy moves $S$ in that direction.
See \reffig{EdgeCompressionDisc}. 
\end{definition}

Definitions~\ref{Def:EdgeComprDisc}, \ref{Def:FaceComprDisc}, and~\ref{Def:IsotopyAlongEdgeCompressionDisc} arise when applying normalisation moves to a surface in a triangulation $\calT$. While they are not as natural in a handle structure, they can be easily moved to this setting as in the following definitions, which assume that $S$ respects a handle structure $\calH$.

\begin{definition}[Edge compression disc, handle structure]\label{Def:EdgeComprDiscHS}
Let $S$ be a closed surface $S$ embedded in $M$ that respects a pre-tetrahedral handle structure $\calH$ on $M$. 
An \emph{edge compression disc} for $S$ in $\calH$ is a disc $D$ lying in a 0-handle $H_0$ of $\calH$ such that $\bdy D$ is the union of two arcs $\alpha\subset S$ and $\beta\subset\bdy H_0$, where $\alpha=D\cap S$, $\beta=D\cap \bdy H_0$ is an arc on the boundary of a 2-handle $H_2\subset\calH$ that is vertical in its product structure, and $\alpha\cap\beta = \bdy\alpha = \bdy\beta$.

Note that because $S$ respects the handle structure, its intersection with the 2-handle $H_2 = D^2 \times D^1$ contains the discs
$D^2\times\bdy\alpha$. Thus when an edge compression disc exists, we actually have the product $D^2\times\beta$ lying in the adjacent 2-handle $H_2=D^2\times D^1$, with $(D^2 \times \beta) \cap S = D^2\times\bdy \beta= D^2\times\bdy\alpha$. 
\end{definition}

\begin{definition}[Face compression disc, handle structure]\label{Def:FaceComprDiscHS}
A \emph{face compression disc} for $S$ in a pre-tetrahedral handle structure $\calH$ is a disc $D$ lying in the intersection of a 1-handle $H_1$ with a 0-handle $H_0$ of $\calH$, such that $\bdy D$ is the union of two arcs $\alpha\subset S$ and $\beta\subset\bdy (H_1\cap H_0)$, where $\alpha=D\cap S$, and $\beta=D\cap \bdy (H_1\cap H_0)$ is an arc on the boundary of a 2-handle $H_2\subset\calH$, and $\alpha\cap\beta = \bdy\alpha = \bdy\beta$.

As in the case of an edge compression disc, because $S$ respects the handle structure $\calH$, its intersection with the 1-handle $H_1=D^1\times D^2$ 
contains $D^1\times\alpha$, and its intersection with the 2-handle $H_2=D^2\times D^1$ contains $D^2\times\bdy\alpha$. Thus when a face compression disc exists, in fact there is a product $D^1\times D\subset H_1$ with $D^1\times\alpha\subset S\cap H_1$, and $D^2\times\beta \subset H_2$ with $D^2\times\bdy\beta\subset S\cap H_2$. 
\end{definition}

\begin{definition}[Isotopy along edge or face compression disc]\label{Def:IsotopyAlongEdgeComprDiscHandle} 
Let $D$ be an edge or face compression disc for a properly embedded surface $S$ that respects a pre-tetrahedral handle structure $\calH$. Thus $\bdy D$ is the union of arcs $\alpha$ and $\beta$, as in Definition~\ref{Def:EdgeComprDiscHS} or~\ref{Def:FaceComprDiscHS}. Let $H_1 = D^1 \times D^2$ be the incident 1-handle in the case of a face compression disc.
 
When $D$ is an edge compression disc, its associated \emph{isotopy ball} $B$ is a small regular neighbourhood in $M\cut S$ of the union of $D$ and the incident 2-handle of $M \cut S$.
 When $D$ is a face compression disc, its associated \emph{isotopy ball} $B$ is a small regular neighbourhood in $M\cut S$ of the union of $D$, the incident 2-handle $H_2$ of $M \cut S$ and $D^1\times D \subset D^1 \times D^2 = H_1$. Then as in the triangulation case, $\bdy B\cap S$ is a regular neighbourhood of $D\cap S$ in $S$ that contains $D^1\times\alpha$ in $S\cap H_1$ and contains $D^2 \times \partial \beta$ in $S \cap H_2$, and the remainder of $\bdy B$ is obtained by taking two parallel copies of $D$ (that are disjoint from $H_1$ in the face compression case) and attaching a band that goes around the 2-handle $H_2$ incident to $D$. 
  
An \emph{isotopy along $D$} is the isotopy that moves $S$ across this ball. When $S$ is transversely oriented, we require that this transverse orientation points towards $D$ and hence that the isotopy moves $S$ in that direction.
\end{definition}

Another isotopy that is applied to normalise surfaces in triangulations is a compression isotopy, defined below. 

\begin{definition}[Compression isotopy, triangulation]
\label{Def:CompressionIsotopy} 
Let $S$ be a transversely oriented surface properly embedded in $M$, in general position with respect to a triangulation $\calT$. Let $D$ be a disc embedded in the interior of a face $F$ of $\calT$, such that $D\cap S = \partial D$ and so that $S$ points into $D$. Suppose that $\partial D$ bounds a disc $D'$ in $S$, and that $D \cup D'$ bounds a ball $B$ with interior disjoint from $S$. Then the \emph{isotopy induced by $D$} is the isotopy that moves $S$ across $B$ and then a little further so that the curve $\partial D$ is removed from $S \cap F$. We call this a \emph{compression isotopy} in $\calT$. See \reffig{CompressionIsotopy}.
\end{definition}

\begin{figure}
  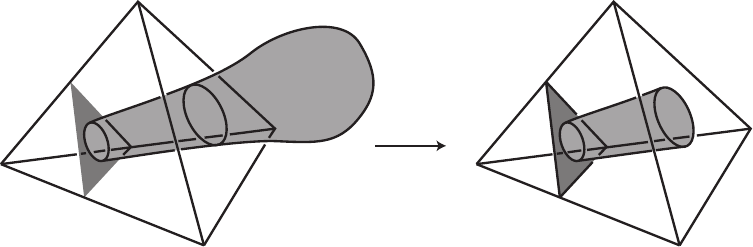
  \caption{A compression isotopy.}
  \label{Fig:CompressionIsotopy}
\end{figure}

As before, we adapt \refdef{CompressionIsotopy} to the handle structure case.

\begin{definition}[Compression isotopy, handle structure]\label{Def:CompressionIsotopyHS}
  Let $S$ be a transversely oriented surface properly embedded in $M$, that respects a pre-tetrahedral handle structure $\calH$ on $M$. Suppose a disc $D$ lies in the interior of the intersection $F=\bdy H_0 \cap H_1$, for a 0-handle $H_0$ and a 1-handle $H_1$, with $D\cap S = \bdy D$ and $S$ pointing into $D$. Suppose $\bdy D$ bounds a disc $D'$ in $S$, and that $D\cup D'$ bounds a ball $B$ with interior disjoint from $S$. Then the \emph{isotopy induced by $D$} is the isotopy that moves $S$ across $B$ and then (if necessary) a little further so that the annulus $D^1\times\bdy D$ is removed from $S\cap H_1$.  We call this a \emph{compression isotopy} in $\calH$.
\end{definition}

When we move between almost normal and normal surfaces, isotoping along edge and face compressing discs, we obtain surfaces that are not necessarily normal. However, they still have a nice form, described by the following definition.

\begin{definition}[Nearly normal, triangulation]\label{Def:NearlyNormal}
Let $F$ be a connected transversely oriented surface properly embedded in a tetrahedron $\Delta$ in general position with respect to the 1-skeleton of $\Delta$. Then $F$ divides $\Delta$ into two components.
Let $B$ be the component into which $F$ points. We say that $F$ is a \emph{nearly normal piece}
if all the following conditions hold:
\begin{enumerate}
\item $B$ is a 3-ball;
\item $\Delta \cut B$ forms a product region between $F$ and a subsurface of $\partial \Delta$: such a product region is obtained by pushing a subsurface of $\bdy \Delta$ into $\Delta$, keeping its boundary fixed;
\item the intersection between $B$ and any face of $\Delta$ is either empty or a disc;
\item $B$ intersects each edge of $\Delta$ in at most one component.
\end{enumerate}
We say that a surface $S$ properly embedded in $M$ is \emph{nearly normal} if it has a transverse orientation
which makes each component of $S \cap \Delta$ a nearly normal piece for each tetrahedron $\Delta$ of $\calT$.
\end{definition}

\begin{figure}
  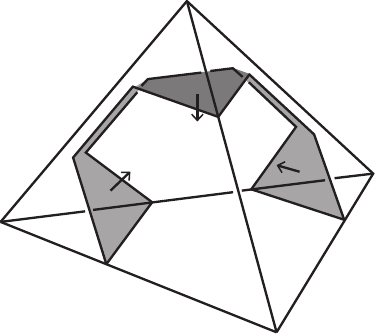
  \caption{A nearly normal disc. This might have been created from three normal triangles by performing isotopies along face compression discs in nearby faces.}
  \label{Fig:NearlyNormal}
\end{figure}

Triangles and squares form nearly normal pieces. However, neither an octagon nor a tubed piece is nearly normal.

The analogous definition for handle structures is the following.

\begin{definition}[Nearly normal, handle structure] \label{Def:NearlyNormalHS}
Let $\calH$ be a pre-tetrahedral handle structure for a compact 3-manifold $M$, and let $S$ be a transversely oriented surface properly embedded in $M$ that respects $\calH$. Let $H_0$ be a 0-handle of $\calH$ and let $F$ be a component of $S \cap H_0$. Then $F$ divides $H_0$ into two components; let $B$ denote the component into which $F$ points. We say that $F$ is a \emph{nearly normal piece} if the following hold:
\begin{enumerate}
\item $F$ is disjoint from $\bdy M\cap \bdy H_0$,
\item $B$ is a 3-ball,
\item $H_0\cut B$ forms a product region between $F$ and a subsurface of $\bdy H_0$,
\item for any component of intersection between $H_0$ and a 1-handle $H_1$, its intersection with $B$ is either empty or a disc,
\item for any component $E$ of intersection between $H_0$ and a 2-handle, $B \cap E$ is either empty or connected.
\end{enumerate}
The surface $S$ is \emph{nearly normal} if each component of $S\cap H_0$ is a nearly normal piece for each 0-handle $H_0$ of $\calH$. 
\end{definition}

\begin{lemma}[Bound on nearly normal pieces]\label{Lem:FinitenessForNearlyNormal}
A tetrahedral or semi-tetrahedral 0-handle contains only finitely many nearly normal pieces up to normal isotopy. Moreover, each nearly
normal piece intersects the union of the 1-handles in at most $12$ arcs or curves.
\end{lemma}

\begin{proof}
A nearly normal piece in a 0-handle $H_0$ is parallel to a subsurface of $\bdy H_0$. This surface $F$ is determined by its boundary curves $\bdy F$, up to two choices. For any 1-handle $H_1$, the intersection between a component of $H_0 \cap H_1$ and $\partial F$ is at most three arcs or is a simple closed curve, because otherwise the 3-ball $B$ in the definition of a nearly normal surface would intersect this component of $H_0 \cap H_1$ in more than one component or in something other than a disc. Hence, there are only finitely many possibilities for $\partial F$ up to normal isotopy. Since $H_0$ intersects the 1-handles in at most $4$ discs, we deduce that the number of components of intersection between $\partial F$ and these 1-handles is at most $12$.
\end{proof}

We will see that the isotopies we have introduced either take an almost normal surface to a nearly normal one, or take a nearly normal surface to a nearly normal one. We prove this here for the case of the compression isotopy. We will save the proof for isotopies along face and edge compression discs for a slightly more general setting in the next section. 

\begin{lemma}[Compression isotopy preserves nearly normal]\label{Lem:CompressionIsotopyPreservesNearlyNormal}
Let $S$ be a nearly normal surface, and let $S'$ be the result of applying a compression isotopy along a disc $D$ that lies in the interior of the intersection of a 1-handle $H_1$ with a 0-handle $H_0$. Then $S'$ is nearly normal.
\end{lemma}
  
\begin{proof}
Nearly normal pieces are removed in their entirety if they were a subset of the disc in $S$ bounded by $\partial D$, but removing pieces does not affect whether a surface is nearly normal. One of the two nearly normal pieces that are incident to $(D^1\times D) \cap H_1$ is removed, the other is modifed. Let $F$ be the nearly normal piece of $S$ that is changed, by adding a parallel copy of $D$. This piece divides the 0-handle $H_0$ that contains it into two components, one of which is the 3-ball $B$. The topology of $B$ remains a ball after this modification. The other component $H_0 \cut B$ remains a product. The intersection between $B$ and any 1-handle or 2-handle remains unchanged, except for the component of $H_0 \cap H_1$ containing $D$, where the intersection becomes empty. These are the only changes that are made to the surface, and so the surface $S'$ is nearly normal.
\end{proof}

Recall from \refdef{Weight} that the weight of a surface is its number of components of intersection with the 2-handles.

\begin{lemma}[Effect on weight, compression isotopy]\label{Lem:DiscIsotopyAndWeight}
Suppose $S'$ is obtained from $S$ by a compression isotopy. Then the weight of $S'$ is at most that of $S$.
\end{lemma}

\begin{proof}
For a compression isotopy, a disc $D$ parallel to the compression disc replaces a disc $D'$ on the surface $S$, and otherwise $S'$ agrees with $S$. Thus no new intersections with 2-handles are introduced. If $D'$ meets 2-handles, the components of intersection with these 2-handles are removed. In any case, the weight does not increase.
\end{proof}

There is one more move on almost normal surfaces that we need to introduce, as follows.

\begin{definition}[Tube compression]
Suppose $S$ is an almost normal incompressible surface with a tubed piece. Suppose $S$ is transversely oriented in the direction that points into the tube. Due to the incompressibility of $S$ and the irreducibility of $M$, compressing this tube yields a 2-sphere plus an isotopic copy of $S$, which we denote by $S'$. The isotopy that we perform moves $S$ across to $S'$. We call this a \emph{tube compression}.
\end{definition}

\begin{lemma}[Tube compression gives normal]\label{Lem:TubeCompressionNearlyNormal}
A tube compression takes an almost normal surface to a normal surface, hence a nearly normal surface.
\end{lemma}

\begin{proof}
Compressing a tubed piece gives two normal discs, one of which belongs to the surface after tube compression. Since the remaining pieces are unchanged, the result is normal. 
\end{proof}

We generally assume that if $S$ contains a tubed piece, then its orientation points out of the tube, because otherwise we immediately perform a tube compression and no longer have such a piece. 

\begin{proposition}[Unsimplifiable implies normal]\label{Prop:UnsimplifiableImpliesNormal}
  Let $S$ be a closed incompressible surface properly embedded in a compact orientable irreducible 3-manifold $M$ equipped with a pre-tetrahedral handle structure $\calH$. Suppose that $S$ is a transversely oriented nearly normal surface that admits no face compression disc in the specified transverse orientation and admits no compression isotopies. Suppose also that $S$ has no component that is a 2-sphere lying in a 0-handle. Then $S$ is normal.
\end{proposition}

\begin{proof}
We check the conditions of \refdef{NormalInHandleStructure}. First, by definition of nearly normal, $S$ respects the handle structure $\calH$.
  
Consider a 0-handle $H_0$. We first show that for any 1-handle $H_1$ incident to $H_0$, $S$ intersects $H_1\cap H_0$ in normal arcs.
Suppose not. Then at least one component of intersection with $H_1 \cap H_0$ is a simple closed curve or an arc with endpoints on the same 2-handle meeting $H_1$. Consider an innermost simple closed curve in $H_1 \cap H_0$. This bounds a disc $D$. The component of $S \cap H_0$ containing $\partial D$ is a nearly normal piece. This divides $H_0$ into two components, one of which is the 3-ball $B$ in the definition of nearly normal. The intersection between $B$ and the face is a disc, and hence it must be $D$. So, $S$ points into $D$ and so we can perform a compression isotopy, which is contrary to assumption.

So now consider an arc of intersection between $S$ and $H_1 \cap H_0$ with endpoints on the same 2-handle $H_2$
that is outermost, in the sense that it separates off a face compression disc $D$ with interior disjoint from $S$. This arc lies in a nearly normal piece, which again divides $H_0$ into two components, one of which is the 3-ball $B$. Since $B$ has connected intersection with each component of $H_2 \cap H_0$, we deduce that $D$ lies in $B$. Hence, we have a face compression disc in the specified transverse direction, which again is a contradiction.

Now consider any piece $F$ of $S \cap H_0$ for some 0-handle $H_0$. We claim that $F$ is a disc. Since $F$ is nearly normal, it is parallel to a subsurface of $\partial H_0$ and hence it is planar. We are assuming that $F$ is not a 2-sphere.
So if $F$ is not a disc, then it has at least two boundary components. These are normal curves, and hence we may find a component $\sigma$ of intersection between $H_0$ and a 1-handle
that intersects two different boundary components of $F$. We may find an embedded arc $\alpha$ in $\sigma$ with interior disjoint from $F$ and with endpoints lying on two different components of $\partial F$. Let $B$ be the 3-ball in the definition of nearly normal. Then $\alpha$ lies in $B$ because otherwise $B \cap \sigma$ is disconnected. We may find an arc $\beta$ properly embedded in $F$ joining the two endpoints of $\alpha$. Then $\alpha \cup \beta$ is a simple closed curve in $\partial B$ that intersects the two boundary components of $F$ in one point each. This is a contradiction, because any simple closed curve in a sphere is separating.

We now show that $S$ is normal. We have shown that each component $F$ of $S\cap H_0$ is a disc and that its boundary curve is a union of normal arcs. If $F$ runs over any of the 2-handles meeting $H_0$ in more than one arc, then \reflem{StockingLemma} implies that its boundary has length at least 8,
and for each component of $H_0 \cut F$, we can find a 2-handle meeting $H_0$ that intersects this component in at least two discs. This contradicts the definition of being nearly normal.
\end{proof}

\subsection{Canonical Handle Structures}

We have defined isotopy moves of surfaces that take almost normal surfaces through nearly normal surfaces to normal surfaces. Each of these surfaces intersects the handle structure of the larger 3-manifold in well-understood ways. Recall that eventually, we wish to count edge contractions and expansions on spines of surfaces as we pass from $S\times\{0\}$ to $S\times\{1\}$. Some of the normal, almost normal, and nearly normal surfaces we have encountered under the above isotopy moves will be the surfaces we analyse. In order to transfer spines, however, we need to determine how a cell structure on the surface itself changes. We start the process in this subsection.

\begin{definition}[Canonical handle structure, embedded surface]\label{Def:CanonicalHandleStructure}
Let $\mathcal{H}$ be a handle structure for a 3-manifold $M$, and let $S$ be an incompressible surface properly embedded in $M$ that respects $\mathcal{H}$ and that has no 2-sphere components. Let $F$ be the intersection between the 0-handles of $\calH$ and $S$. If any component of $F$ is not a disc, then its boundary curves all bound discs in $S$, since $S$ is incompressible. These discs may be nested. Under these circumstances, enlarge $F$ by including these discs, forming a surface $F^+$. The \emph{canonical handle structure} on $S$ has 0-handles equal to the components of $F^+$. The 1-handles are the components of intersection between the 1-handles of $\calH$ and $S \cut F^+$. The 2-handles are the components of intersection between the 2-handles of $\calH$ and $S \cut F^+$.
\end{definition}

Examples of canonical handle structures before and after an isotopy along an edge compression disc are shown in \reffig{CanonHS-EdgeIsotopy}. 
Note that, by \reflem{FinitenessForNearlyNormal}, each 0-handle of the canonical handle structure on a nearly normal surface has at most 12 components of intersection with the 1-handles and at most 12 components of intersection with the 2-handles.

\begin{figure}
\begingroup%
  \makeatletter%
  \providecommand\color[2][]{%
    \errmessage{(Inkscape) Color is used for the text in Inkscape, but the package 'color.sty' is not loaded}%
    \renewcommand\color[2][]{}%
  }%
  \providecommand\transparent[1]{%
    \errmessage{(Inkscape) Transparency is used (non-zero) for the text in Inkscape, but the package 'transparent.sty' is not loaded}%
    \renewcommand\transparent[1]{}%
  }%
  \providecommand\rotatebox[2]{#2}%
  \newcommand*\fsize{\dimexpr\f@size pt\relax}%
  \newcommand*\lineheight[1]{\fontsize{\fsize}{#1\fsize}\selectfont}%
  \ifx\svgwidth\undefined%
    \setlength{\unitlength}{224.74089432bp}%
    \ifx\svgscale\undefined%
      \relax%
    \else%
      \setlength{\unitlength}{\unitlength * \real{\svgscale}}%
    \fi%
  \else%
    \setlength{\unitlength}{\svgwidth}%
  \fi%
  \global\let\svgwidth\undefined%
  \global\let\svgscale\undefined%
  \makeatother%
  \begin{picture}(1,0.58532302)%
    \lineheight{1}%
    \setlength\tabcolsep{0pt}%
    \put(0,0){\includegraphics[width=\unitlength,page=1]{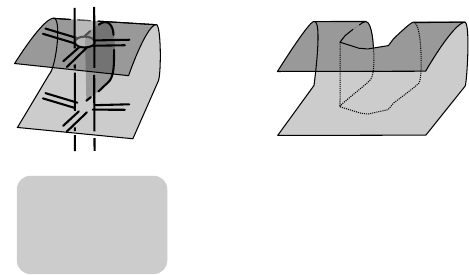}}%
    \put(-0.00369349,0.45842557){\color[rgb]{0,0,0}\makebox(0,0)[lt]{\lineheight{1.25000012}\smash{\begin{tabular}[t]{l}$S$\end{tabular}}}}%
    \put(0,0){\includegraphics[width=\unitlength,page=2]{CanonHS_Surface_Isotopy_Gray.pdf}}%
    \put(0.23718036,0.38415436){\color[rgb]{0,0,0}\makebox(0,0)[lt]{\lineheight{1.25000012}\smash{\begin{tabular}[t]{l}$D$\end{tabular}}}}%
    \put(0.22563365,0.55151671){\color[rgb]{0,0,0}\makebox(0,0)[lt]{\lineheight{1.25000012}\smash{\begin{tabular}[t]{l}$\alpha$\end{tabular}}}}%
    \put(0,0){\includegraphics[width=\unitlength,page=3]{CanonHS_Surface_Isotopy_Gray.pdf}}%
    \put(0.55032046,0.45528658){\color[rgb]{0,0,0}\makebox(0,0)[lt]{\lineheight{1.25000012}\smash{\begin{tabular}[t]{l}$S'$\end{tabular}}}}%
    \put(0,0){\includegraphics[width=\unitlength,page=4]{CanonHS_Surface_Isotopy_Gray.pdf}}%
  \end{picture}%
\endgroup%

  \caption{On the left is the surface $S$ admitting an edge compression disc. At the top, it is shown in $M$, below is shown its inherited canonical handle structure. On the right is the surface $S'$ after isotopy along the edge compression disc, with its canonical handle structure.}
  \label{Fig:CanonHS-EdgeIsotopy}
\end{figure}

\begin{lemma}[Compression isotopy and canonical handle structure]\label{Lem:CompressionIsotopyEffect}
Let $S$ be an incompressible nearly normal surface with no 2-sphere components,
and let $S'$ be obtained from it by a compression isotopy. Then the compression isotopy takes the canonical handle structure on $S$ to the canonical handle structure on $S'$.
\end{lemma}

\begin{proof}
Let $D$ and $D'$ be the discs in \refdef{CompressionIsotopyHS}. Let $H_0$ be the 0-handle of $\mathcal{H}$ containing $D$. The disc $D$ forms a compression disc for a component of $H_0 \cap S$, and hence $D'$ lies in the interior of a 0-handle of the canonical handle structure on $S$. The isotopy moves $D'$ across to $D$ and then pushes it a little further off the 1-handle containing $D$. Thus after this isotopy, the image of $D'$ lies within a 0-handle of the canonical handle structure on $S'$. The remainder of the canonical handle structures on $S$ and $S'$ are equal. 
\end{proof}

\begin{lemma}[Edge or face compression and canonical handle structure]\label{Lem:EffectEdgeFaceComp}
Suppose $S$ is an incompressible almost normal or nearly normal surface with no 2-sphere components, and $S'$ is a nearly normal surface obtained from $S$ by isotopy along an edge or face compression disc $D$, with $\bdy D = \alpha \cup \beta$ as in Definitions~\ref{Def:EdgeComprDiscHS} and~\ref{Def:FaceComprDiscHS}. Suppose that $\alpha$ does not lie entirely in the interior of a 0-handle of the canonical handle structure of $S$,
in other words, that $\alpha$ does not lie in the interior of a component of $F_+$ as in \refdef{CanonicalHandleStructure}. Suppose also that we are not performing an edge compression to a tubed piece.
Then the canonical handle structure of $S'$ is obtained from that of $S$ as follows: 
\begin{enumerate}
\item Remove the 2-handles at the end of $\alpha$.
\item If $\alpha$ lies in the boundary of a 0-handle of $S$ (i.e.\ $D$ is a face compression disc), remove the 1-handle containing $\alpha$.
\item Join each 1-handle at one end of $\alpha$ to a 1-handle at the other end, pairing those that lie in the same 1-handle of $\calH$. 
\item The spaces between these 1-handles become part of 0-handles in the new handle structure.
\item However, if this results in any regions that are not discs, these are replaced by discs that each form a single 0-handle.
\end{enumerate}
\end{lemma}

\begin{proof}
The proof follows from the definition of isotopy along an edge or face compression disc, \refdef{IsotopyAlongEdgeComprDiscHandle}. See Figures~\ref{Fig:CanonHS-EdgeIsotopy} and~\ref{Fig:CanonHSDiscIsotopy}.
\begin{figure}
\begingroup%
  \makeatletter%
  \providecommand\color[2][]{%
    \errmessage{(Inkscape) Color is used for the text in Inkscape, but the package 'color.sty' is not loaded}%
    \renewcommand\color[2][]{}%
  }%
  \providecommand\transparent[1]{%
    \errmessage{(Inkscape) Transparency is used (non-zero) for the text in Inkscape, but the package 'transparent.sty' is not loaded}%
    \renewcommand\transparent[1]{}%
  }%
  \providecommand\rotatebox[2]{#2}%
  \newcommand*\fsize{\dimexpr\f@size pt\relax}%
  \newcommand*\lineheight[1]{\fontsize{\fsize}{#1\fsize}\selectfont}%
  \ifx\svgwidth\undefined%
    \setlength{\unitlength}{224.74089432bp}%
    \ifx\svgscale\undefined%
      \relax%
    \else%
      \setlength{\unitlength}{\unitlength * \real{\svgscale}}%
    \fi%
  \else%
    \setlength{\unitlength}{\svgwidth}%
  \fi%
  \global\let\svgwidth\undefined%
  \global\let\svgscale\undefined%
  \makeatother%
  \begin{picture}(1,0.58532302)%
    \lineheight{1}%
    \setlength\tabcolsep{0pt}%
    \put(0,0){\includegraphics[width=\unitlength,page=1]{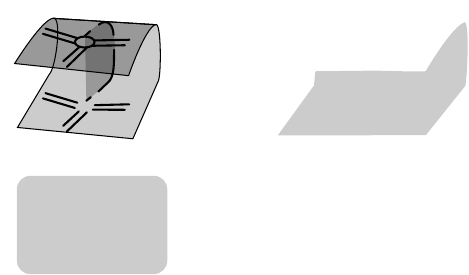}}%
    \put(-0.00369349,0.45842557){\color[rgb]{0,0,0}\makebox(0,0)[lt]{\lineheight{1.25000012}\smash{\begin{tabular}[t]{l}$S$\end{tabular}}}}%
    \put(0,0){\includegraphics[width=\unitlength,page=2]{CanonHS_FaceCompIsotopy_Gray.pdf}}%
    \put(0.22563365,0.55151671){\color[rgb]{0,0,0}\makebox(0,0)[lt]{\lineheight{1.25000012}\smash{\begin{tabular}[t]{l}$\alpha$\end{tabular}}}}%
    \put(0,0){\includegraphics[width=\unitlength,page=3]{CanonHS_FaceCompIsotopy_Gray.pdf}}%
    \put(0.55032046,0.45528658){\color[rgb]{0,0,0}\makebox(0,0)[lt]{\lineheight{1.25000012}\smash{\begin{tabular}[t]{l}$S'$\end{tabular}}}}%
    \put(0,0){\includegraphics[width=\unitlength,page=4]{CanonHS_FaceCompIsotopy_Gray.pdf}}%
    \put(0.18057151,0.12678059){\color[rgb]{0,0,0}\makebox(0,0)[lt]{\lineheight{1.25000012}\smash{\begin{tabular}[t]{l}$\alpha$\end{tabular}}}}%
    \put(0,0){\includegraphics[width=\unitlength,page=5]{CanonHS_FaceCompIsotopy_Gray.pdf}}%
  \end{picture}%
\endgroup%

  \caption{How the canonical handle structure on $S$ is affected by an isotopy along a face compression disc.}
  \label{Fig:CanonHSDiscIsotopy}
\end{figure}
\end{proof}

We now consider the cases excluded in the above lemma.

\begin{lemma}[Edge or face compression in 0-handle]\label{Lem:EffectEdgeCompIn0Handle}
Suppose $S$ is an incompressible almost normal or nearly normal surface with no 2-sphere components, and $S'$ is a nearly normal surface obtained from $S$ by isotopy along an edge or face compression disc $D$, with $\bdy D = \alpha \cup \beta$ as in Definitions~\ref{Def:EdgeComprDiscHS} and~\ref{Def:FaceComprDiscHS}.
If $\alpha$ lies entirely in the interior of a 0-handle of the canonical handle structure of $S$, then the isotopy taking $S$ to $S'$ transfers the canonical handle structure of $S$ to the canonical handle structure of $S'$.
\end{lemma}

\begin{proof}
The assumption of the lemma implies that $\alpha$ lies in a disc in $S$ with boundary lying in a 0-handle of $\mathcal{H}$. The isotopy from $S$ to $S'$ takes this disc to a corresponding disc in $S'$. Hence, the 0-handle of the canonical handle structure containing $\alpha$ is sent to a 0-handle of $S'$, and all the other handles remain unchanged.
\end{proof}

\begin{lemma}[Edge compression in tubed piece]\label{Lem:EffectEdgeCompTube}
Suppose $S$ is an incompressible almost normal surface with a tubed piece and no 2-sphere components, and $S'$ is a nearly normal surface obtained from $S$ by isotopy along an edge compression disc $D$, with $\bdy D = \alpha \cup \beta$ as in Definitions~\ref{Def:EdgeComprDiscHS} and~\ref{Def:FaceComprDiscHS}.
Then the canonical handle structure on $S'$ is obtained from that of $S$ as follows. Let $E$ be a disc in $S$ whose boundary $\bdy E$ is a simple closed curve forming a core curve of the tube. Components of intersection between $E$ and the 1-handles and 2-handles of $\calH$ have been removed to form the canonical handle structure on $S$. These are reinstated as 1-handles and 2-handles. Then steps (1) to (5) of \reflem{EffectEdgeFaceComp} are applied.
\end{lemma}

\begin{proof}
\refdef{CanonicalHandleStructure} ensures that $E\subset S$ lies in a 0-handle of the canonical handle structure of $S$, 
with all 1-handles and 2-handles of $E$ removed. Applying the isotopy along the edge compression disc may adjust this structure. However, after reinstating the 1- and 2-handles of $E$, the adjustment will then be identical to that of \reflem{EffectEdgeFaceComp}. 
\end{proof}

\begin{lemma}[Tube compression and canonical handle structure]\label{Lem:EffectTubeComp}
Suppose $S$ is an incompressible almost normal surface with a tubed piece and no 2-sphere components.
Let $S'$ be obtained from $S$ by a tube compression. Then the isotopy taking $S$ to $S'$ takes the canonical handle structure on $S$ to the canonical handle structure on $S'$.
\end{lemma}
\smallskip

\begin{proof} Let $C$ be a simple closed curve in $S$ that is the boundary of the compression disc for the tube. Then $C$ lies in an annular component of $S \cap \calH^0$. This annulus is a subset of a 0-handle in the canonical handle structure. The surface $S'$ is obtained from $S$ by removing the disc in $S$ bounded by $C$ and replacing it by the compression disc. Hence, this compression disc also ends up as part of a 0-handle in the canonical handle structure of $S'$. Away from these discs, the canonical handle structures of $S$ and $S'$ agree.
\end{proof}

\section{Parallelity bundles and generalised isotopy moves}\label{Sec:ParallelityBundles}

\textbf{Road map:} We are working towards a bound on the number of edge expansions and contractions required to transfer a spine on $S\times\{0\}$ to $S\times\{1\}$. 
We will be transferring the spines along a sequence of surfaces that consist of normal, almost normal, and nearly normal surfaces as defined in the previous section, and we will bound the number of edge contractions and expansions at each step. A consequence of \refprop{UnsimplifiableImpliesNormal} is that compression isotopies and isotopies along edge and face compression discs take an almost normal surface to a normal one. 
So this will give us a complete sequence of surfaces to work with. Unfortunately, using these moves alone gives far too many surfaces. There is no good way to bound the number of such surfaces in a sequence; without a bound on the number of surfaces, having a bound on edge expansions and contractions at each step will not lead to the bound we need. Therefore, we need to introduce more drastic moves. This section introduces the moves and the necessary terminology.

\subsection{Parallelity bundles}
Consider an isotopy across a face compression disc, as in \refdef{IsotopyAlongEdgeComprDiscHandle}. This pushes a surface $S$ past a 2-handle.
If on the other side of that 2-handle, $S$ cuts the handles of $M$ into parallelity handles, then this isotopy move will give rise to new face compression discs, and the move must be repeated to slide $S$. There is no good bound on the number of parallelity handles that $M\cut S$ may contain, and it could be the case that each one leads to a required isotopy across a face compression disc. For example this is shown one dimension down in \reffig{ParallelityExample2D}.
Instead of performing these isotopies one by one, we want to perform them all in a single step that depends on the parallelity handles adjacent to the face compression disc.

\begin{figure}
  \includegraphics{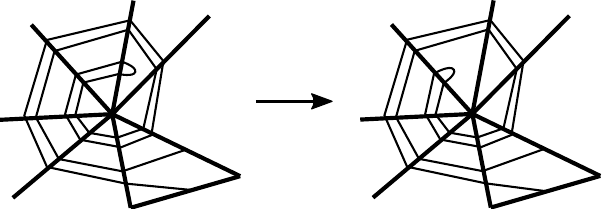}
  \caption{ A two-dimensional example of the inefficiency of isotoping across a face compression disc. There is no bound on the number of times such a move must be performed, since there is no bound on the number of parallel normal segments. }
  \label{Fig:ParallelityExample2D}
\end{figure}

In this subsection, we define parallelity bundles and their generalisations, which will allow us to define the more general isotopy moves that we need. These bundles first appeared in \cite{Lackenby:CrossingNo}. The idea is as follows. 
Normal surfaces are made up of triangles and squares in tetrahedra. We cannot bound the number of normal triangles and squares in the intersection of a normal surface $S$ with a given tetrahedron. But if $S$ intersects a single tetrahedron many times, then it must do so in more and more parallel triangles and squares. These cut the tetrahedron, and hence the 3-manifold $M \cut S$, into $I$-bundles. 

Throughout this section, let $\mathcal{H}$ be a handle structure of a 3-manifold $M$. Let $S$ be a surface properly embedded in $M$ that
respects the handle structure.

\begin{definition}[Parallelity bundle for $S$]
The \emph{parallelity bundle} $\mathcal{B}$ for $S$ is the submanifold of $M \cut S$ that is the union of the parallelity regions, as in \refdef{ParallelityRegions}. 
\end{definition}

When $\mathcal{H}$ is dual to a triangulation $\calT$ of $M$, then 
it is possible to visualise the parallelity bundle for a surface $S$ in $\calT$ without specifically having to pass to the dual handle structure.
If we have two adjacent normally parallel pieces of $S$ in a tetrahedron,
the space between them forms a parallelity region in $M \cut S$. Similarly, if the intersection between $S$ and a face of $\calT$ contains two adjacent parallel arcs, the space between them, thickened a little, forms a parallelity
region. When $S$ intersects an edge of $\calT$, it divides the edge into arcs, and all but the outermost two arcs produce parallelity regions in the dual 2-handle of $\mathcal{H}$. See \reffig{ParallelityBundle}.

\begin{figure}
  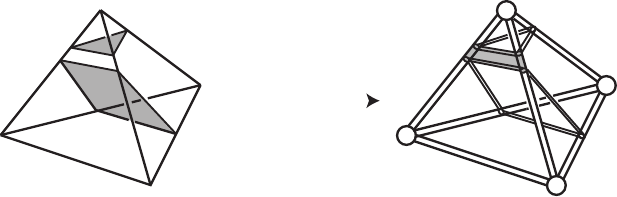
  \caption{A portion of the parallelity bundle}
  \label{Fig:ParallelityBundle}
\end{figure}

We also define parallelity handles in more general handle structures
$\mathcal{H}'$ for a 3-manifold $M'$ that is not necessarily of the form $M \cut S$, and where a specified subsurface $F$ of $\partial M'$ is given.

\begin{definition}[Parallelity bundle, handle structure]\label{Def:ParallelityBundle}
Let $\mathcal{H}'$ be a handle structure of a compact orientable 3-manifold $M'$, and let $F$ be a subsurface of $\partial M'$ such that $\partial F$ is standard.
A handle $H$ of $\mathcal{H}'$ is a \emph{parallelity handle} if it admits a product structure $D^2 \times I$ such that
\begin{enumerate}
\item $D^2 \times \partial I = H \cap F = H \cap \partial M'$;
\item each component of intersection between $H$ and another handle is $\beta \times I$
for a subset $\beta$ of $\partial D^2$.
\end{enumerate}

The \emph{parallelity bundle} $\mathcal{B}$ for $\mathcal{H}'$ is the union of the parallelity handles.
\end{definition}

It is shown in \cite[Lemma~5.3]{Lackenby:CrossingNo} that the product structures on the parallelity handles can be chosen so they make the parallelity bundle into an $I$-bundle over a surface $E$.

\begin{definition}[Horizontal and vertical boundary]
  The \emph{horizontal boundary} $\partial_h \mathcal{B}$ is the $(\partial I)$-bundle over $E$. It is a subsurface of $F$. The \emph{vertical boundary} $\partial_v \mathcal{B}$ is the $I$-bundle over $\partial E$. It is an orientable $I$-bundle properly embedded in $M'$, and hence it is a collection of annuli.
\end{definition}

We will want to expand the parallelity bundle to include not just the parallelity handles, but also other simple $I$-bundles whose fibring matches that of the parallelity bundle. This is done by the following definition.

\begin{definition}[Generalised parallelity bundle]\label{Def:GeneralisedParallelityBundle}
Let $M'$ be a compact orientable 3-manifold with a handle structure $\mathcal{H}'$. Let $F$ be a subsurface of $\partial M'$ such that $\partial F$ is standard.
A \emph{generalised parallelity bundle} $\mathcal{B}$ is a 3-dimensional submanifold of $M'$ such that
\begin{enumerate}
\item $\mathcal{B}$ is an $I$-bundle over a compact surface;
\item the horizontal boundary of $\mathcal{B}$ is the intersection between $\mathcal{B}$ and $F$;
\item $\mathcal{B}$ is a union of handles of $\mathcal{H}'$;
\item any handle of $\mathcal{B}$ that intersects $\partial_v \mathcal{B}$ is a parallelity handle, where the $I$-bundle structure on the parallelity handle agrees with the $I$-bundle structure of $\mathcal{B}$;
\item whenever a handle of $\mathcal{H}'$ lies in $\mathcal{B}$ then so do all incident handles of $\mathcal{H}'$ with higher index;
\item the intersection between $\partial_h \mathcal{B}$ and the non-parallelity handles lies in a union of disjoint discs in the interior of $\partial_h \calB$.
\end{enumerate}

We also say that $\calB$ is a \emph{generalised parallelity bundle for $(M',F)$}, when we wish to emphasize the manifold and surface. When we do not specify $F$, we take $F$ to be all of $\partial M'$.

A generalised parallelity bundle is \emph{maximal} if it is not a proper subset of another generalised parallelity bundle.
\end{definition}

Of course, the parallelity bundle for $\mathcal{H}'$ is a generalised parallelity bundle. But the more general concept is useful. For example, suppose that a vertical boundary component of the parallelity bundle $\mathcal{B}$ separates off a component of $M' \cut \mathcal{B}$ that is an $I$-bundle over disc, with the $(\partial I)$-bundle lying in $F$, and that this
$I$-bundle structure extends that on $\mathcal{B}$. Then it is natural to
enlarge $\mathcal{B}$ by including the $I$-bundle over the disc. The result
is a generalised parallelity bundle.

Note that condition (6) in the above definition is new to this paper; it does not appear in \cite[Definition 5.2]{Lackenby:CrossingNo}.

\subsection{Isotopy moves across $I$-bundles over discs}
We now generalise \refdef{IsotopyAlongEdgeComprDiscHandle}. 
Let $\mathcal{H}$ be a pre-tetrahedral handle structure of $M = S \times [0,1]$. Let $S$ be a normal or almost normal fibre. Let $\mathcal{B}$ be a maximal generalised parallelity bundle for $M \cut S$.

We suppose that we have performed some isotopies to $S$, all in the same transverse direction, taking it to a nearly normal surface $S'$. We also allow the possibility that no isotopies have been performed, and hence that $S' = S$, which is normal or almost normal. 

\begin{definition}[Generalised isotopy move across $I$-bundle over disc]\label{Def:GenIsotopyMoveAcrossB}
Suppose $S'$ is nearly normal and $D$ is a face compression disc for $S'$, or suppose $S'$ is almost normal and $D$ is an edge compression disc for $S'$. Suppose that the interior of $D$ is disjoint from $\mathcal{B}$, but $\bdy D$ meets $\calB$.
Its intersection with a 2-handle is an arc (denoted $\beta$ in \refdef{FaceComprDiscHS}) in the vertical boundary of $\mathcal{B}$. Let $W$ be the isotopy ball for $D$. Let $B$ be the component of $\mathcal{B}$ incident to $D$. Suppose that $B$ is an $I$-bundle over a disc and that its horizontal boundary lies in $S'$. Then the \emph{generalised isotopy move across $B$} moves $S'$ across the ball $W \cup B$. Thus, it removes the horizontal boundary of $B$ from $S'$, together with $S' \cap W$, and it replaces it with $\partial_v B \cut W$ together with two parallel copies of $D$, isotoped slightly so that the result respects the handle structure.
See \reffig{GenIsotopyMoveDisc}.

We also call this move a \emph{generalised isotopy along $D$}, where $D$ is an edge compression disc or face compression disc. We also occasionally refer to it for short as a \emph{generalised edge or face compression}. 
\end{definition}

\begin{figure}
  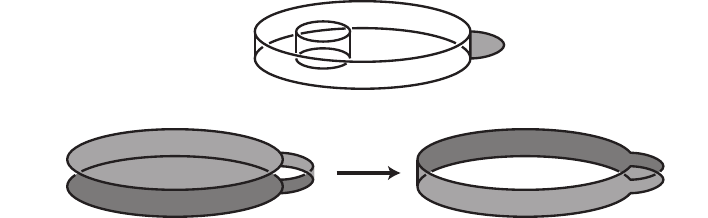
  \caption{A generalised isotopy move across an $I$-bundle over a disc}
  \label{Fig:GenIsotopyMoveDisc}
\end{figure}

Note that the 2-handle meeting the arc $\beta$ is a parallelity handle. If this handle is not adjacent to any other parallelity handles, then the usual isotopy along an edge or face compression disc, as in \refdef{IsotopyAlongEdgeComprDiscHandle}, is an example of a generalised isotopy along $D$, where the component $B$ consists just of the parallelity 2-handle meeting $\beta$.

The effect on the canonical handle structure on the surface is nearly identical to that in the usual isotopy along $D$, as follows.

\begin{lemma}[Effect of generalised isotopy on canonical handle structure]\label{Lem:EffectDiscIsotopy}
Suppose $S'$ is an incompressible nearly normal or almost normal surface without 2-sphere components,
and $S''$ is a nearly normal surface obtained from $S'$ by a generalised isotopy along an edge or face compression disc $D$, with $\bdy D = \alpha \cup \beta$ as in Definitions~\ref{Def:EdgeComprDiscHS} and~\ref{Def:FaceComprDiscHS}.
Then the canonical handle structure of $S''$ is obtained from that of $S'$ as follows:
\begin{itemize}
\item If $\alpha$ lies entirely in the interior of a 0-handle of the canonical handle structure of $S'$, then the handle structure of $S''$ is unchanged.
\item If the generalised isotopy is not along an edge compression disc running over a tubed piece, then the canonical handle structure of $S''$ is obtained completely analogously to steps (1) to (5) of \reflem{EffectEdgeFaceComp}:
  \begin{enumerate}
  \item Remove from $S'$ all handles in the horizontal boundary $\bdy_hB$ of $B$ in $S'$.
  \item If $D$ is a face compression disc, remove the 1-handle containing $\alpha$.
  \item Each 1-handle adjacent to one component of $\bdy_hB$ is paired with a 1-handle adjacent to the other component of $\bdy_hB$, where the two handles lie in the same handle of $\calH$, both adjacent to the same parallelity handle on $\bdy B$.
  \item The spaces between these 1-handles become part of 0-handles in the new handle structure.
  \item However, if this results in any regions that are not discs, these are replaced by discs that each form a single 0-handle.
  \end{enumerate}
\item If the generalised isotopy is along an edge compression disc running over a tubed piece, then as in \reflem{EffectEdgeCompIn0Handle}, reinstate 1-handles and 2-handles in a disc on $S'$ with boundary a core curve of the tube, and then apply steps (1) through (5) above. 
\end{itemize}
\end{lemma}

\begin{proof}
This follows by definition. See \reffig{CanonHSGenIsotopyAcrossB}.
\end{proof}

\begin{figure}
  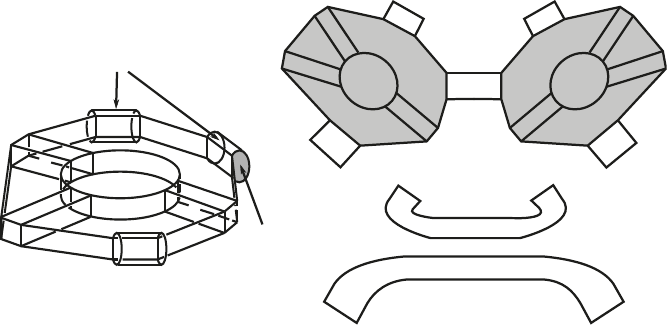
  \caption{Left: A component of the generalised parallelity bundle that is an $I$-bundle over a disc. Right: The effect on the canonical handle structure of a generalised isotopy move across this component. The numbers show the index of the handles. The shaded region is the horizontal boundary of the bundle.}
  \label{Fig:CanonHSGenIsotopyAcrossB}
\end{figure}

We now show that these moves result in a nearly normal surface.

\begin{lemma}[Generalised face compression preserves nearly normal]\label{Lem:FaceCompressionDiscIsotopyPreservesNearlyNormal}
Let $S'$ be a nearly normal surface, and let $D$ be a face compression disc for $S'$. Let $S''$ be the result of a generalised isotopy of $S'$ along $D$. Then $S''$ is nearly normal.
\end{lemma}

\begin{proof}
Let $H_1$ denote the 1-handle containing $D$, and $H_2$ the 2-handle meeting an arc $\beta$ of $\bdy D$. The surface $S''$ is obtained from $S'$ by modifying it in the two 0-handles incident to $H_1$, and in all handles in a component $W$ of a generalised parallelity bundle that is an $I$-bundle over a disc, and in handles adjacent to $\bdy_v W$.

Within $H_2$ and $W$, all intersections are removed in all handles. Within $H_1$, the component of $H_1 \cap S'$ incident to $D$ is removed. So there is nothing to consider in these handles.

Let us first consider a 0-handle $H_0$ that intersects $D^1\times D\subset H_1$.
Assume for now that $H_0$ meets $H_1$ and $W$ only once, so $((D^1\times D)\cup (D^2\times\beta))$ intersects $H_0$ only in a disc in a single component of $H_1\cap H_0$ and an adjacent component of $H_2\cap H_0$. Let $F$ be the nearly normal piece of $S'$ incident to $D$ in $H_0$. Then, by the definition of nearly normal, $F$ divides $H_0$ into two components, one of which is a 3-ball $B$ containing $D$. When we perform the isotopy, $B$ is reduced by removing a thin regular neighbourhood of $(D^1\times D) \cap \bdy H_0$, and of $(D^2\times \beta) \cap \bdy H_0$, but it remains a 3-ball $B'$ say. The intersection between $B'$ and the 1-handle $H_1$ meeting $D$ is now empty. 
There will be another 1-handle adjacent to $H_2$ on $\bdy H_0$. Again we assume for now that there is only one such 1-handle. The intersection of $B'$ with this 1-handle has shrunk, by removing a small regular neighbourhood of $D^2\times \beta\subset H_2$, but it remains a disc. 

The region $H_0\cut B'$ is obtained from $H_0\cut B$ by attaching a thin regular neighbourhood of $(D^1\times D) \cap \bdy H_0$ and $(D^2\times \beta) \cap \bdy H_0$, and hence it remains a product region. A region of $B$ meeting $D^2\times\beta \subset H_2$ is removed in this operation, but the remaining components of intersection with 2-handles are unchanged. Thus, the resulting piece of $S''$ is still nearly normal. 

Let us now consider a 0-handle $H_0$ that is incident to the $I$-bundle over a disc $W$, but does not meet $D^1 \times D$.
Consider any nearly normal piece $F$ of $S'$ in $H_0$ that is modified. It is altered by banding it onto another nearly normal piece $F'$ in $H_0$, or possibly banding it to itself. Then $F$ (respectively $F'$) divides $H_0$ into two regions, one of which is a ball $B$ (respectively, $B'$) and the other of which is a product region. The transverse orientations on $F$ and $F'$ point into $B$ and $B'$, by the definition of a nearly normal piece. So the band is disjoint from the product regions, and hence when the band is attached to form a piece $F''$, one component of $H_0\cut F''$ is again a product region between $F''$ and a subsurface of $\bdy H_0$. 
Hence, the other component is homeomorphic to $H_0$ and therefore a ball $B''$. The intersection between $B''$ and any component $\sigma$ of intersection between $H_0$ and a 1-handle must be empty or connected, for the following reason. The intersection between $B$ and $\sigma$ is empty or a disc, by hypothesis, as is $B' \cap \sigma$. Thus $B' \cap B \cap \sigma$ is empty or a disc. The intersection $B'' \cap \sigma$ either is equal to $B' \cap B \cap \sigma$ or is obtained from $B' \cap B \cap \sigma$ by removing a thin strip running along $\sigma\cap \bdy_vW$, and hence remains empty or a disc. The intersection between $B$ and any other 2-handle $H_2'$ meeting $\bdy H_0$ is empty or a product $[0,1]\times a$, where $a$ is an arc. The same is true for $B' \cap H_2'$. Hence, $B' \cap B \cap H_2'$ is empty or a product $[0,1]\times (a\cap a')$, and therefore so is $B'' \cap H_2'$. 

In the above argument, we assumed that the intersection of  $((D^1\times D) \cup (D^2\times\beta))$ with $H_0$ was either equal to a disc in a single 1-handle $H_1\cap H_0$ and an adjacent 2-handle $H_2\cap H_0$, or to a subset of a single 2-handle, $(D^2\times\beta)\cap \bdy H_0$. However, the boundary of the 0-handle $H_0$ might actually meet the same 1-handle $H_1$ twice, or $W$ multiple times. Thus $(D^1\times D) \cup (D^2\times \beta)$ may actually have several components of intersection on $\bdy H_0$ of this form. But the conclusion of the argument remains unchanged, because we can view the modification from $S' \cap H_0$ to $S'' \cap H_0$ as being achieved in several steps, each of the form discussed above.
\end{proof}

\begin{lemma}[Generalised edge compression gives nearly normal]\label{Lem:EdgeComprIsotopyYieldsNearlyNormal}
Let $S$ be a transversely oriented almost normal surface. If $S$ contains a tubed piece, suppose that its orientation points out of the tube. Let $S'$ be obtained from $S$ by a generalised isotopy along an edge compression disc. Then $S'$ is nearly normal. 
\end{lemma}

\begin{proof}
The edge compression disc $D$ must be incident to the almost normal piece of $S$. The effect on this almost normal piece yields a nearly normal piece.
Other pieces of $S$ are also affected, if they are incident to the $I$-bundle over a disc that is adjacent to $D$. A band is added to these pieces. As argued above in the proof of \reflem{FaceCompressionDiscIsotopyPreservesNearlyNormal}, attaching a band in this way does not alter the fact that the pieces are nearly normal.
\end{proof}

\begin{lemma}[Generalised isotopy along a disc reduces weight]\label{Lem:GenIsotopyAcrossBWeight}
Suppose $S''$ is obtained from $S'$ by applying a generalised isotopy along an edge or face compression disc. Then the weight of $S''$ is strictly less than that of $S'$.
\end{lemma}

\begin{proof}
Let $B$ denote the $I$-bundle over a disc in \refdef{GenIsotopyMoveAcrossB}.
It is a component of a generalised parallelity bundle $\calB$, and by definition whenever a handle lies in $\calB$, so do all incident handles of higher index. Since 2-handles are the highest index parallelity handle, $B$ must meet a 2-handle. Then $S'$, which runs over the horizontal boundary of $B$, must meet the same 2-handle. When we apply the isotopy move, the new surface $S''$ becomes disjoint from all 2-handles in $B$, and does not meet new 2-handles. Thus the weight strictly decreases. 
\end{proof}

\subsection{Parallelisation isotopy}

Again let $\calH$ be a pre-tetrahedral handle structure of $M = S \times [0,1]$. Let $S$ be an almost normal fibre.

\begin{definition}[Parallelisation isotopy] \label{Def:ParallelisationIsotopy}
Let $S'$ be a normal fibre that is topologically parallel to $S$. Suppose that there is a copy of $D^2 \times [0,1]$ in $M \cut (S \cup S')$ such that $D^2 \times \{ 0 \}$ is a subset of $S$, $D^2 \times \{ 1 \}$ is a subset of $S'$, and $\partial D^2 \times [0,1]$ is a vertical boundary component of the parallelity bundle of $M \cut (S \cup S')$. Then a \emph{parallelisation isotopy} moves $S$ across $D^2 \times [0,1]$, taking it to $S'$.
That is, replace $S$ with $(S\setminus (D^2\times\{0\})) \cup (D^2\times \{1\}) \cup (\bdy D^2 \times [0,1])$, perturbed to respect the handle structure. When $S$ has transverse orientation, we require that this points into $D^2 \times [0,1]$ so that the isotopy moves $S$ in this direction.
\end{definition}

We do not require that a parallelisation isotopy reduces the weight of $S$, and so there is no analogue of Lemma \ref{Lem:GenIsotopyAcrossBWeight} in this setting. By assumption, $S'$ is normal, and so the conclusion of Lemma \ref{Lem:EdgeComprIsotopyYieldsNearlyNormal} trivially holds.

The motivation for introducing parallelisation isotopies comes from the following.

\begin{lemma}[Almost normal piece in bundle]\label{Lem:AlmostNormalInBundle}
Let $\calH$ be a pre-tetrahedral handle structure for $M = S \times [0,1]$. Let $S$ be an almost normal fibre, transversely oriented in some way. Let $\calB$ be a maximal generalised parallelity bundle for the component of $M \cut S$ into which $S$ points. Suppose that the almost normal piece of $S$ intersects the interior of $\partial_h \calB$. Then $S$ admits a parallelisation isotopy.
\end{lemma}

\begin{proof} The almost normal surface $S$ has only one almost normal piece $P$. Hence, $P$ cannot be part of the parallelity bundle of $M \cut S$. However, we are assuming that it intersects the interior of the horizontal boundary of the maximal generalised parallelity bundle $\calB$. 
By (6) in the definition of a generalised parallelity bundle, Definition \ref{Def:GeneralisedParallelityBundle}, the surface that is obtained by removing the parallelity handles from $\partial_h \calB$ lies in a union of disjoint discs in the interior of $\partial_h \calB$. The almost normal piece $P$ must lie in one of these discs. The $I$-bundle structure on $\calB$ gives a copy of $D^2 \times [0,1]$ in $\calB$, where $D^2 \times \{ 0 \}$ is the disc containing $P$, $D^2 \times \{ 1 \}$ also lies in $\partial_h \calB$ and $\partial D^2 \times [0,1]$ is a vertical boundary component of the parallelity bundle. Thus, we can perform the paralleliation isotopy that moves $S$ across $D^2 \times [0,1]$.  
\end{proof}

\subsection{Annular simplification}
In addition to isotopy moves across $I$-bundles over a disc and parallelisation isotopies, we also need moves that can be applied to generalised parallelity bundles that are $I$-bundles over annuli. The simplification will be applied both to 3-manifolds and to surfaces; both definitions are below.

\begin{definition}[Annular simplification]\label{Def:AnnularSimplification}
Let $M'$ be a compact orientable irreducible 3-manifold with a handle structure $\mathcal{H}'$, and let $F$ be an incompressible surface of $\partial M'$ such that $\partial F$ is standard and $F$ is not a 2-sphere.

Suppose $M'$ contains the following:
\begin{itemize}
\item an annulus $A'$ that is a vertical boundary component of a generalised parallelity bundle $\calB$;
\item an annulus $A$ contained in $F$ such that $\partial A=\partial A'$;
\item a 3-manifold $P$ with $\partial P = A \cup A'$ such that $P$ either lies in a 3-ball or is a product region between $A$ and $A'$.
\end{itemize}
Suppose also that $P$ is a union of handles of $\calH'$, that whenever a handle of $\calH'$ lies in $P$, so do all incident handles with higher index, and that any parallelity handle of $\calH'$ that intersects $P$ lies in $P$. Finally, suppose that apart from the component of the generalised parallelity bundle incident to $A'$, all other components of $\mathcal{B}$ in $P$ are $I$-bundles over discs.

An \emph{annular simplification} of the 3-manifold $M'$ is the manifold obtained by removing the interiors of $P$ and $A$ from $M'$. When $P$ lies in a 3-ball, it is a \emph{trivial annular simplification}. When $A$ is an essential subsurface of $\partial M'$, it is an \emph{essential annular simplification}.
See \reffig{AnnularSimplification}.
Note that the resulting 3-manifold is homeomorphic to $M'$, even though $P$ may be homeomorphic to the exterior of a non-trivial knot.

Suppose that the component $F'$ of $F$ containing $A$ is transversely oriented, with orientation pointing into $M'$. An \emph{annular simplification} of the surface $F'$ is obtained by replacing $A$ in $F'$ by $A'$, and then isotoping slightly past the handles containing $A'$, in the direction of the orientation, obtaining a surface $S'$. 

We often apply this definition in the situation where $S \subset M$ is a surface that respects the handle structure on $M$ with no 2-sphere components, and $M'=M\cut S$ and $F'$ is a copy of $S$ in $\bdy M'$. Suppose $M'$ admits an annular simplification, with annulus $A\subset S$. Let $S'$ be the surface obtained from $F'$ by annular simplification, i.e.\ replacing $A$ by an annulus $A'$. Viewing $S'$ as a surface in $M$, we say $S'$ is obtained from $S$ by an annular simplification.
\end{definition}

\begin{figure}
  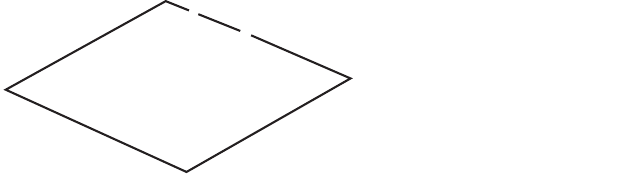
  \caption{An annular simplification removes $P$ and replaces $A$ with $A'$. Left: a trivial annular simplification. Right: a cross-section of an essential annular simplification. Note that the parallel lines in $\calB$ represent $I$-fibres, here and onwards.}
  \label{Fig:AnnularSimplification}
\end{figure}

\begin{lemma}[Annular simplification and canonical handle structure] \label{Lem:AnnularSimpChangeHS}
Let $M$ be a compact, orientable, irreducible 3-manifold with a handle structure. Let $S$ be an incompressible nearly normal surface without 2-sphere components.
Let $S'$ be obtained from $S$ by an annular simplification, replacing an annulus $A$ in $S$ with an annulus $A'$ in $S'$. Then the effect on the canonical handle structure of $S$ is as follows:
\begin{enumerate} 
\item Remove all handles in $A$.
\item Adjacent to each component of $\partial A$, in $S \cut A$, there are 0-handles and 1-handles in an alternating fashion. Each such 0-handle or 1-handle adjacent to one component of $\partial A$ corresponds to a 0-handle or 1-handle adjacent to the other component of $\partial A$, where the correspondence arises from the $I$-bundle structure on $A'$. Combine each pair of 0-handles into a single 0-handle of $S'$, and do the same for each pair of 1-handles.
\item If this results in any regions that are not discs, these are replaced by discs that each form a single 0-handle $F^+$, as in \refdef{CanonicalHandleStructure}.
\end{enumerate}
\end{lemma}

\begin{proof}
This follows from the definition. 
\end{proof}

Analogous to other isotopy moves, an annular simplification preserves nearly normal surfaces and decreases weight.

\begin{lemma}[Annular simplification preserves nearly normal]\label{Lem:AnnSimpPreservesNearlyNormal}
Let $M$ be a compact orientable irreducible 3-manifold with a handle structure. Let $S$ be an incompressible nearly normal surface without 2-sphere components, and let $S'$ be obtained from $S$ by applying an annular simplification. Then $S'$ is nearly normal.
\end{lemma}

\begin{proof}
We use the terminology of \refdef{AnnularSimplification}: $P$ is a submanifold of $M \cut S$, lying between an annulus $A$ in $S$ and an annulus $A'$ properly embedded in $M \cut S$. The modification to $S$ is the removal of $A$ and the addition of $A'$. This annulus $A'$ is a vertical boundary component of the parallelity bundle of $M \cut S$, and the component of the parallelity bundle that contains $A'$ lies in $P$. The effect of the removal of $A$ and $P$ is to remove some nearly normal pieces of $S$. We must consider what happens when we add $A'$.

Let $H_0$ be a 0-handle of $\calH$ that has non-empty intersection with $A'$. The intersection between $A'$ and $H_0$ is a union of fibres in the $I$-bundle structure on $A'$. These therefore form bands that are added to nearly normal pieces of $S \cap H_0$. In the definition of an annular simplification, the transverse orientation on the new surface $S'$ is such that it points out of $P$ along $A'$. Thus, just as in the case of a generalised isotopy along a face compression disc, \reflem{FaceCompressionDiscIsotopyPreservesNearlyNormal}, the new surface $S'$ remains nearly normal.
\end{proof}

\begin{lemma}[Annular simplification reduces weight]\label{Lem:AnnSimpReducesWeight}
Suppose $S'$ is obtained from $S$ by an annular simplification. Then the weight of $S'$ is strictly less than that of $S$.
\end{lemma}

\begin{proof}
Let $A'$ be an annulus as in the definition of an annular simplification, which replaces the annulus $A$ in $S$. If $A'$ meets any 2-handles, then since $\bdy A = \bdy A' \subset S$, the surface $S$ meets the same 2-handles. When we replace $A$ with $A'$, and then isotope past the handles containing $A'$, we will remove those intersections with 2-handles without adding any additional intersections with 2-handles. Thus if we can show that $A'$ runs through some 2-handle, we will have shown that the weight decreases strictly. But $A'$ lies in the vertical boundary of a generalised parallelity bundle $\calB$, and by definition whenever a handle lies in $\calB$, then so do all incident handles of higher index. Hence, as one travels around $A'$, the handles of $\mathcal{B}$ that are incident to $A'$ are alternately 1-handles and 2-handles. In particular, $A'$ meets at least one 2-handle.
\end{proof}

The left side of \reffig{AnnularSimplification} shows a special type of parallelity bundle, which we now define.

\begin{definition}[Boundary-trivial]\label{Def:BoundaryTrivial}
Let $M'$ and $F$ be as in \refdef{GeneralisedParallelityBundle}, and let $\mathcal{B}$ be a generalised parallelity bundle for $(M',F)$. A component $P$ of $\mathcal{B}$ is \emph{boundary-trivial} if the following all hold.
\begin{enumerate}
\item $P$ lies within a 3-ball $B$ such that $B \cap \partial M'$ is a single disc lying in $F$,
\item the disc $\partial B \cut \partial M'$ is disjoint from any component of $\mathcal{B}$ that is not an $I$-bundle over a disc,
\item and $P$ is not an $I$-bundle over a disc (with respect to the $I$-bundle structure of $\calB$).
\end{enumerate}
\end{definition}

Note in the above definition, $\bdy_hP$ lies in $F$, so $\bdy_vP$ must be annuli properly embedded in $B$ with boundary on $B\cap F$.

\begin{lemma}[Boundary-trivial bundles are annular]\label{Lem:BoundaryTrivialAnnular}
Let $M'$ and $F$ be as in \refdef{AnnularSimplification}. Let $\mathcal{B}$ be a maximal generalised parallelity bundle for $(M', F)$.  Then every boundary-trivial component of $\calB$ is an $I$-bundle over an annulus.
\end{lemma}

\begin{proof}
Let $P$ be a component of $\calB$ that is boundary-trivial. By definition, $P$ lies within a 3-ball $B$ such that $B\cap \bdy M'$ is a single disc lying in $F$. The horizontal boundary of $P$ lies in the disc $B \cap \partial M'$ and so it is a planar surface. By definition, $P$ is not an $I$-bundle over a disc. Suppose that it is not an $I$-bundle over an annulus.

Note first that $P$ cannot be a twisted $I$-bundle, because it would then contain a M\"obius band with boundary in $\partial B$, which is impossible. Hence, $P$ is a product $I$-bundle over a planar surface with at least three boundary components. Let $F_1$ and $F_2$ be the two horizontal boundary components. Each $F_i$ has one boundary component $C_i$ that bounds a disc in $B \cap \partial M'$ containing $F_i$, but the remaining components of $\partial F_i$ bound discs in $B \cap \partial M'$ with interior disjoint from $F_i$. Since the base of the $I$-bundle $P$ has at least three boundary components, there is a component $A'$ of $\partial_v P$ disjoint from $C_1$ and $C_2$. Assign a transverse orientation to $A'$ that points out of $P$. Let $E_1$ and $E_2$ be the two components of $\partial A'$, where $E_i = A' \cap F_i$. Each curve $E_i$ bounds a disc $D_i$ in $B \cap \partial M'$ with interior disjoint from $F_i$. We claim that these discs are not nested. Suppose that, on the contrary, they are, with $D_2$ lying in $D_1$, say. The transverse orientation on $E_i$ points out of $P$ and hence into $D_i$. Thus, we can form a closed curve in $B$ that intersects $A'$ once transversely, by starting in $D_1$ near $E_1$, then running through $D_1 \cut D_2$ to $E_2$, then through $E_2$, then along a curve parallel to $A'$. This contradicts the fact that $A'$ must separate the 3-ball $B$. Thus, $D_1$ and $D_2$ are disjoint, and therefore we can form the 2-sphere $D_1 \cup A' \cup D_2$. This bounds a 3-ball in $B$, and we can extend the $I$-bundle structure of $P$ over this 3-ball. This contradicts the assumption that $\calB$ is maximal.
\end{proof}

\begin{lemma}[Boundary-trivial admits annular simplification]\label{Lem:BoundTrivAdmitsAnnularSimp}
Let $M'$ and $F$ be as in \refdef{AnnularSimplification}.
Let $\mathcal{B}$ be a maximal generalised parallelity bundle for $(M', F)$.  
Suppose that a component of $\mathcal{B}$ is boundary-trivial. Then $M'$ admits an annular simplification.
In particular, a vertical boundary component $A'\subset \bdy_vP$ has boundary components $\bdy A'$ cobounding an annulus $A$ on $F$.
\end{lemma}

\begin{proof}
Let $P$ be a component of a maximal generalised parallelity bundle $\calB$ that is boundary-trivial. By definition, $P$ lies within a 3-ball $B$ such that $B\cap \bdy M'$ is a single disc lying in $F$, and $\bdy B\cut \bdy M'$ is disjoint from any component of $\calB$ that is not an $I$-bundle over a disc. Choose $P$ so that it is furthest from $\partial B \cut \partial M'$, in the sense that if $P'$ is any other component of $\calB$ in $B$ that is not an $I$-bundle over a disc, then $P$ does not separate $P'$ from $\partial B \cut \partial M'$.

By Lemma \ref{Lem:BoundaryTrivialAnnular}, $P$ is an $I$-bundle over an annulus. So, $\partial_v P$ is two annuli.
Let $A'$ be the component of the vertical boundary of $P$ that is closer to $\partial B \cut \partial M'$. 
Then $A'$ is properly embedded in $M'$ with both boundary components in the disc $B\cap \bdy M'$. If the components of $\bdy A'$ are parallel within $B\cap \bdy M'$, then they cobound an annulus $A$, and all the conditions on $A$ and $A'$ of \refdef{AnnularSimplification} are satisfied. Thus it admits an annular simplification.

If the components of $\bdy A$ are not parallel within $B\cap \bdy M'$, then they bound disjoint discs. But then the two discs must form the horizontal boundary of an $I$-bundle over a disc with vertical boundary $A'$. Thus because $\calB$ is maximal, this $I$-bundle makes up $P$, and $P$ is not boundary-trivial. 
\end{proof}

\subsection{Generalised isotopy moves}
At this point we have described four isotopies that involve drastic changes along parallelity bundles, namely generalised isotopies along edge or face compression discs, parallelisation isotopies and annular simplifications. 
In this section, we show that these isotopies are, in some sense, exactly the right isotopies to move an almost normal surface to a normal surface in the presence of large parallelity bundles. We do this by first classifying all components of the maximal generalised parallelity bundle in \refthm{IncompressibleHorizontalBoundary}. We then investigate properties of parallelity bundles that admit a simplification under one of our isotopy moves. We show that any such parallelity bundle either is boundary-trivial, or an $I$-bundle over a disc or annulus, in \refprop{DiscsAndAnnuli}, which means that if the surface admits a simplification through a parallelity bundle, it will occur under one of the isotopies we have defined.

The following is a version of \cite[Corollary 5.7]{Lackenby:CrossingNo}.

\begin{theorem}[Classification of generalised parallelity bundles]\label{Thm:IncompressibleHorizontalBoundary}
Let $M'$ be a compact orientable irreducible 3-manifold with a handle structure $\mathcal{H}'$, and let $F$ be an incompressible surface of $\partial M'$ such that $\partial F$ is standard and $F$ is not a 2-sphere.
Let $\mathcal{B}$ be a generalised parallelity bundle for $(M', F)$ that is maximal, in the sense that it is not a proper subset of another generalised parallelity bundle. Then $\mathcal{B}$ contains every parallelity handle of $\mathcal{H}'$. Moreover, every component of $\mathcal{B}$ is either
\begin{itemize}
\item an $I$-bundle over a disc, or
\item boundary-trivial, or
\item has incompressible vertical and horizontal boundary.
\end{itemize}
\end{theorem}

\begin{proof}
Note first that $\mathcal{B}$ must contain every parallelity handle of $\mathcal{H}'$, as otherwise we may take any component of the parallelity bundle of $(M', F)$ that does not lie wholly in $\mathcal{B}$ and add it to $\mathcal{B}$. It is easy to check that the result is still a generalised parallelity bundle, which contradicts the hypothesis that $\mathcal{B}$ is maximal.

We divide $\mathcal{B}$ into three subsets, each of which is a union of components of $\mathcal{B}$:
\begin{enumerate}
\item the union of the $I$-bundles over discs, denoted $\mathcal{B}_D$;
\item the union of the boundary-trivial components, denoted $\calB_\partial$;
\item the union of the remaining components, denoted $\calB_I$.
\end{enumerate}
We will show that $\calB_I$ has incompressible vertical boundary and incompressible horizontal boundary.

Our first task is show that we may find a collection of disjoint properly embedded discs that separate $\calB_\partial$ from $\calB_I$. These discs may intersect $\calB_D$. The boundary of the discs will lie in $F$.

Let $D$ be a collection of disjoint compression discs for $F - \partial_h (\calB_\partial \cup \calB_I)$ in the complement of $\calB_\partial \cup \calB_I$, which is maximal in the sense that any other compression disc that is disjoint from $D$ is parallel to a component of $D$. Since $F$ is incompressible and $M'$ is irreducible, each component of $D$ is parallel to a disc in $F$, via a 3-ball. Any two of these 3-balls must be disjoint or nested, because otherwise $F$ is a 2-sphere. Hence the union of these 3-balls is a collection of balls $B$, each of which intersects $F$ in a single disc.
Note that $\bdy B \cut F$ is disjoint from $\calB_\bdy \cup \calB_I$.

Any component of $\mathcal{B}$ lying within $B$ that is not an $I$-bundle over a disc must be boundary-trivial by definition. We now show that $B$ contains $\calB_\partial$. Suppose that, on the contrary, there is a component of $\calB_\partial$ that is disjoint from $B$. By definition, this lies in a ball $B'$ such that $\partial B' \cut \partial M'$ is a disc disjoint from $\calB_\partial \cup \calB_I$. By cutting and pasting, we may modify $B'$ so that $\bdy B' \cut \bdy M'$ is disjoint from $D$. By maximality of $D$, $\bdy B'\cut\bdy M'$ must be parallel into $D$, via a product region disjoint from $\calB_\partial \cup \calB_I$. But then $B'$ can be isotoped into $B$, without moving $\calB_\partial \cup \calB_I$, which contradicts our assumption that this component of $\calB_\bdy$ is disjoint from $B$. 

Suppose that the vertical boundary of $\mathcal{B}_I$ is compressible. 
Let $D'$ be a compression disc, with boundary lying in an annular component $A$ of $\partial_v \mathcal{B}_I$. The interior of $D'$ is disjoint from $\mathcal{B}_I$ because no component of $\mathcal{B}_I$ is $I$-bundle over a disc. By cutting and pasting, we may assume that $D'$ is disjoint from the discs $\partial B \cut F$. Then $A$ compresses to two discs $D_1$ and $D_2$. These are disjoint from $\mathcal{B}_\partial \cup \mathcal{B}_I$. They are parallel to discs $D'_1$ and $D'_2$ in $F$, and for $i = 1$ and $2$, $D_i \cup D_i'$ bounds a 3-ball $B_i$. The balls $B_1$ and $B_2$ are either disjoint or nested. We consider these two cases separately. The argument is illustrated in \reffig{ComprVertBdy}.

\begin{figure}
  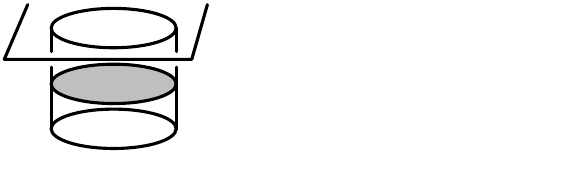
  \caption{Left: if $B_1$ and $B_2$ are disjoint. Right: if $B_2\subset B _1$. }
  \label{Fig:ComprVertBdy}
\end{figure}

If $B_1$ and $B_2$ are disjoint, then the union of $B_1$, $B_2$ and a regular neighbourhood of $D'$ forms a copy of $D^2 \times I$, with interior disjoint from $\mathcal{B}_I$ and with $(D^2 \times I) \cap F = D^2 \times \partial I$. Therefore, we may enlarge $\mathcal{B}$ by attaching $D^2 \times I$, forming  a larger generalised parallelity bundle than $\mathcal{B}$. This contradicts the assumption that $\mathcal{B}$ is maximal.

Suppose now $B_1$ and $B_2$ are nested. Say that $B_2$ lies in $B_1$. Then because $D'$ is disjoint from the interior of $\mathcal{B}_I$, we deduce that the component of $\mathcal{B}_I$ containing $A$ lies in $B_1$. This implies that this component is boundary-trivial, which is a contradiction.

Thus, we have shown that every component of $\mathcal{B}$ either is an $I$-bundle over a disc, is boundary-trivial, or has incompressible vertical boundary. 

We will show that the horizontal boundary of $\mathcal{B}_I$ is incompressible. We can assume that a compression disc $D'$ is disjoint from the discs $D = \partial B \cut F$. It must necessarily intersect the vertical boundary of $\mathcal{B}_I$.
Then consider the curves of intersection between $D'$ and $\bdy_v\mathcal{B}_I$. Any curve can either be removed by an isotopy, or it bounds a disc in $D'$ that forms a compression disc for $\partial_v \mathcal{B}_I$. But we have just shown that $\partial_v \mathcal{B}_I$ is incompressible. So $\bdy_h\mathcal{B}_I$ is incompressible. 
\end{proof}

One of our goals is to show that generalised isotopy moves are powerful enough to isotope an almost normal surface to a normal one. In order to show this, we need to understand the way that generalised parallelity bundles can lie in $S \times [0,1]$. We are therefore led to the following definition.

\begin{definition}[Coherent, incoherent $I$-bundle]\label{Def:Coherent}
Let $B$ be a connected $I$-bundle embedded in $S \times [0,1]$ with $\bdy_h B = B \cap (S \times \{ 0 ,1 \})$. Then $B$ is \emph{coherent} if it intersects both $S \times \{ 0 \}$ and $S \times \{ 1 \}$. Otherwise, it is \emph{incoherent}. See \reffig{CoherentIncoherentSchematic}. 
\end{definition}

\begin{figure}
  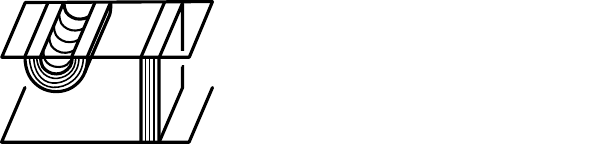
  \caption{Schematic pictures of coherent and incoherent $I$-bundles}
  \label{Fig:CoherentIncoherentSchematic}
\end{figure}

Note that isotopy moves along edge and face compression discs move a surface through an incoherent $I$-bundle. Similarly, an annular simplification occurs in an incoherent $I$-bundle. The following proposition shows that these are essentially the only moves needed to reduce a surface.

\begin{proposition}[Incoherent is boundary-trivial or $I$-bundle over disc or annulus]\label{Prop:DiscsAndAnnuli}
Let $\calH$ be a handle structure for $S \times [0,1]$ where $S$ is a closed orientable surface. Let $\calB$ be a maximal generalised parallelity bundle for $S \times [0,1]$. Let $B$ be an incoherent component of $\calB$. Then $B$ is either boundary-trivial or an $I$-bundle over a disc or an annulus.
\end{proposition}

\begin{proof}
Suppose that $B$ is neither boundary-trivial nor an $I$-bundle over a disc. By \refthm{IncompressibleHorizontalBoundary}, the horizontal and vertical boundaries of $B$ are incompressible. Any vertical boundary component $A$ of $B$ is an incompressible annulus properly embedded in $S \times [0,1]$ with $\partial A$ lying in $S \times \{ 0 \}$, say. Hence, it is parallel to an annulus in $S \times \{ 0 \}$. 
It cannot be the case that, for every vertical boundary component of $B$, the product region between it and the corresponding annulus in $S \times \{ 0 \}$ has interior disjoint from $B$. Hence, there is some vertical boundary component $A$ with the property that the product region between $A$ and the annulus in $S\times \{ 0 \}$ contains $B$. Therefore, $\partial_h B$ lies in this annulus in $S\times \{ 0 \}$. By the incompressibility of $\partial_h B$, we deduce that each component of $\partial_h B$ is an annulus. Therefore, $B$ is an $I$-bundle over an annulus or M\"obius band.
The proof is concluded by observing that a M\"obius band does not properly embed in a copy of $S \times [0,1]$, as follows. Suppose not. Then a M\"obius band would have its single boundary component embedded in $S\times\{0\}$ or $S\times\{1\}$. Double along that boundary component to obtain a copy of the Klein bottle embedded in $S\times[-1,1]$. But now $S\times[-1,1]$ embeds in the 3-sphere, so the Klein bottle embeds in the 3-sphere, and that is a contradiction.
\end{proof}

\begin{definition}[Essential annular $I$-bundle]
Let $B$ be a connected $I$-bundle embedded in $S \times [0,1]$ with $\partial_h B = B \cap (S \times \{ 0 ,1 \})$. Then $B$ is \emph{annular} if it is an $I$-bundle over an annulus. It is \emph{essential} if its vertical boundary is incompressible.
\end{definition}

Note that by \refprop{DiscsAndAnnuli}, any incoherent component of the maximal generalised parallelity bundle that is neither boundary-trivial nor an $I$-bundle over a disc must be both annular and essential. 

We now list all the isotopy moves that we will use in this paper.

\begin{definition}[Generalised isotopy moves]\label{Def:GeneralisedIsotopyMove}
Let $S$ be a normal or almost normal fibre in $M = S \times [0,1]$. Give it a transverse orientation. Let $\calB$ be a maximal generalised parallelity bundle for the component of $M \cut S$ into which $S$ points. Let $S'$ be either $S$ or a nearly normal surface in $M$, on the side of $S$ towards which $S$ points. 
A \emph{generalised isotopy move} is one of the following moves applied to $S'$:
\begin{enumerate}
\item a generalised isotopy along an edge compression disc, provided $S'$ is almost normal (\refdef{EdgeComprDiscHS} and \refdef{GenIsotopyMoveAcrossB});
\item a parallelisation isotopy, provided $S'$ is almost normal (\refdef{ParallelisationIsotopy});
\item a compression isotopy, provided $S'$ is nearly normal (\refdef{CompressionIsotopyHS});
\item an annular simplification, on the side into which $S'$ points
  (\refdef{AnnularSimplification});
\item a generalised isotopy along a face compression disc (\refdef{GenIsotopyMoveAcrossB});
\item a tube compression, if $S'$ is a tubed almost normal surface that is oriented into the tube.
\end{enumerate}
\end{definition}

\begin{remark}
\label{Rem:PairwiseDisjointSurfaces}
Observe that if $S$ is an almost normal or nearly normal surface with a transverse orientation, then performing any of the generalised isotopy moves above to $S$, and then isotoping slightly in the direction in which $S$ points, yields a surface that:
\begin{itemize}
\item is disjoint from $S$,
\item lies in the component of $M\cut S$ to which $S$ points, and
\item admits a transverse orientation, pointing away from $S$.
\end{itemize}
Thus repeatedly applying generalised isotopy moves yields a sequence of pairwise disjoint, transversely oriented surfaces. 
\end{remark}

\begin{remark}
\label{Rem:IsotopyMovesOnNormalSurface}
Many of the generalised isotopy moves are only applicable to almost normal surfaces. So if $S'$ is nearly normal, then the only possible moves that might be applied to $S'$ are compression isotopies, annular simplifications and generalised isotopy along a face compression disc. Thus, if $S'$ is normal, the only posssible move that might be applicable is an annular simplification.
\end{remark}

\begin{proposition}[Nearly normal preserved]\label{Prop:NearlyNormalPreserved}
Let $S$ be a transversely oriented almost normal surface.
Let $S'$ be obtained from $S$ by an isotopy along an edge compression disc, a generalised isotopy along an edge compression disc, a tube compression or a parallelisation isotopy, followed by a sequence of generalised isotopy moves. Then $S'$ is nearly normal.
\end{proposition}

\begin{proof}
By Lemmas~\ref{Lem:EdgeComprIsotopyYieldsNearlyNormal} and \ref{Lem:TubeCompressionNearlyNormal}, a first isotopy that is a generalised isotopy along an edge compression disc,
or a tube compression yields a nearly normal surface. By definition, a parallelisation isotopy takes $S$ to a normal surface, which is therefore nearly normal. By Lemmas~\ref{Lem:CompressionIsotopyPreservesNearlyNormal}, \ref{Lem:FaceCompressionDiscIsotopyPreservesNearlyNormal} and~\ref{Lem:AnnSimpPreservesNearlyNormal}, any compression isotopies, generalised isotopies along face compression discs, and annular simplifications leave the surface nearly normal. So the result is nearly normal.
\end{proof}

\subsection{Avoiding the interior of the generalised parallelity bundle}\label{Sec:AvoidingB}

We have defined generalised isotopy moves.
They are defined in terms of an initial normal or almost normal surface $S$ and a maximal generalised parallelity bundle $\calB$ for $M \cut S$. Note that as we perform the generalised isotopy moves, we obtain a sequence of surfaces, but we continue to work with the parallelity bundle $\calB$ for $M \cut S$ where $S$ is the initial surface. There is an important property that we would like these surfaces to satisfy. We would like to ensure that each surface in this sequence is disjoint from the interior of $\calB$. The following lemma nearly achieves this. It also places the moves into a convenient order.

\begin{lemma}[Almost normal to normal, mostly disjoint from $\calB$]\label{Lem:AlmostNormGenIsotopyDisjointB}
Let $S$ be an almost normal surface that is transversely oriented. Let $\calB$ be a maximal generalised parallelity bundle for the component of $M \cut S$ into which $S$ points. Then there is a sequence of generalised isotopy moves, all in the direction specified by the transverse orientation of $S$, taking $S$ to a normal surface, with the following properties:
\begin{enumerate}
\item If a generalised isotopy along an edge compression disc is performed, this is the first move of the sequence.
\item If a tube compression or a parallelisation isotopy is performed, this is the only move of the sequence.
\item If any compression isotopies are performed, these all take place at the end of the sequence.
\item For each surface in the sequence before the compression isotopies, the surface is disjoint from the interior of each component of $\calB$, except possibly in a regular neighbourhood of the edge compression disc, where it is allowed to go a little into a component of $\calB$ that is an $I$-bundle over an annulus.
\item From the second surface onwards, all surfaces in the sequence are nearly normal.
\end{enumerate}
\end{lemma}

\begin{proof}
Suppose first that the initial surface $S$ has a tubed piece, and the transverse orientation points into the tube. Then we perform a tube compression and end with a normal surface. 
Similarly, if $S$ admits a parallelisation isotopy, then we perform this and end with a normal surface. 

Suppose now that if $S$ contains a tubed piece, then the transverse orientation points out of the tube. Suppose also that $S$ does not admit a parallelisation isotopy. The initial almost normal surface $S$ has an edge compression disc $D$ on the side into which $S$ points. By \reflem{AlmostNormalInBundle}, the almost normal piece is disjoint from the interior of $\partial_h \calB$, and hence $D$ is disjoint from the interior of $\calB$. This disc $D$ is incident to a 2-handle, which is part of $\calB$. If the 2-handle is an entire component of $\calB$, then perform an edge compression along $D$. If this component of $\calB$ is a larger $I$-bundle over a disc, then perform a generalised isotopy along the edge compression disc $D$. If this component of $\calB$ is an $I$-bundle over an annulus, then we just perform the edge compression along $D$ and allow the resulting surface to go slightly into this component of $\calB$.

After this process, the result is a nearly normal surface. By \refprop{NearlyNormalPreserved}, any later surfaces in the sequence will be nearly normal. Let $S'$ be the surface that has been obtained so far. 

Suppose that there is an incoherent component $B$ of $\calB$ on the side of $S'$ into which it points, such that $B$ is not an $I$-bundle over a disc. 
Then $B$ is boundary-trivial or an $I$-bundle over an essential annulus, by \refprop{DiscsAndAnnuli}. If there is a boundary-trivial component, then $S'$ admits an annular simplification by \reflem{BoundTrivAdmitsAnnularSimp}. If there is an essential incoherent annular $I$-bundle, then we may perform an annular simplification along an extrememost one in $M \cut S'$.

So, we may assume that every incoherent component of $\calB$ on the side of $S'$ into which it points is an $I$-bundle over a disc. If $S'$ 
has a face compression disc, pick one, called $D$. Let $\beta$ be the arc as in \refdef{FaceComprDiscHS}. This lies in a 2-handle $D^2 \times D^1$ with $D^2 \times \partial D^1$ in $S$.
Hence, this 2-handle is a parallelity 2-handle, and is therefore part of a component $B$ of $\calB$. By assumption, $B$ is an $I$-bundle over a disc, and so we can perform a generalised isotopy move across $B$. 

Each of these moves reduces the weight of $S'$, by Lemmas~\ref{Lem:GenIsotopyAcrossBWeight} and~\ref{Lem:AnnSimpReducesWeight}. Thus, this process eventually terminates at a surface $S'$. This is nearly normal but does not have a face compression disc. If $S'$ is not normal, then by \refprop{UnsimplifiableImpliesNormal}, $S'$ admits a compression isotopy. Perform this compression isotopy. This does not create any new face compression disc.
Hence, we may repeat, performing compression isotopies at each stage. Each compression isotopy reduces the number of annular components of intersection between $S'$ and the 1-handles. Hence, this process eventually terminates with a surface that must be normal.
\end{proof}

In the sequence of surfaces taking $S$ to a normal surface, we want to ensure that the surfaces are disjoint from the interior of $\calB$. The reason is that this will enable us to provide upper bounds on the number of edge swaps taking a spine in one surface to a spine in the other. These bounds will mostly be in terms of the length of the vertical boundary of $\calB$. A formal definition of this length will be given in Definition \ref{Def:LengthVerticalBoundary} but it is roughly the number of handles of $\calH$ that $\partial_v \calB$ runs through. 

Unfortunately, Lemma~\ref{Lem:AlmostNormGenIsotopyDisjointB} is not quite sufficient for our purposes. One reason for this is that in (4), the surface is allowed to go a little into the interior of $\calB$. We now develop an extension to Lemma~\ref{Lem:AlmostNormGenIsotopyDisjointB} that will allow us to improve (4).
However, first we need a definition.

\begin{definition}[Clean annular simplification]
\label{Def:CleanAnnularSimplification}
Let $M$ be a compact orientable 3-manifold with a pre-tetrahedral handle structure. Let $S$ be a transversely oriented normal or almost normal surface properly embedded in $M$. Let $\calB$ be a maximal generalised parallelity bundle for the component of $M \cut S$ into which $S$ points. 

Suppose $S'$ is obtained from $S$ by a sequence of generalised isotopy moves satisfying \reflem{AlmostNormGenIsotopyDisjointB}, all in the given transverse direction.
Suppose that $S''$ is obtained from $S'$ by an annular simplification that moves $S'$ across a boundary-trivial or essential annular component $B$ of $\calB$. Suppose that this sequence of moves taking $S$ to $S''$ also satisfies  \reflem{AlmostNormGenIsotopyDisjointB}. That is, we think of $S'$ as occurring after a sequence of steps in the process of \reflem{AlmostNormGenIsotopyDisjointB}, and $S''$ as a further step.

The annular simplification from $S'$ to $S''$ isotopes an annulus $A'$ in $S'$ across to a vertical boundary component $A''$ of $B$. We claim the sequence of isotopy moves from $S$ to $S'$ does not move $S\cap \bdy_hB$, except possibly in the first step if that step is an edge compression isotoping slightly into $B$. This is because these isotopy moves are followed by a further annular simplification and so, by the ordering in Lemma~\ref{Lem:AlmostNormGenIsotopyDisjointB}, none are a tube compression, a parallelisation isotopy, or a compression isotopy. All other generalised isotopy moves remove the component of $\calB$ that they meet. Thus the isotopy taking $S$ to $S'$ does not move $S\cap \bdy_h B$.

Hence, there is a corresponding annulus $A$ in $S$ with boundary equal to $\partial A''$.
We say that this annular simplification is \emph{clean} if the isotopies restricted to $A$, taking $A$ to $A'$, are all annular simplifications. Therefore the notion of clean depends on the entire sequence of generalised isotopy moves taking $S$ to $S''$.
\end{definition}

For example, if an edge compression takes a surface a little into a component of $\calB$ that is an $I$-bundle over an annulus $B$, and then $B$ is adjusted immediately after by an annular simplification, the annular simplification is not clean.

On the other hand, the annular simplification moves an annulus across a 3-manifold $P'$ bounded by $A'$ and $A''$. If we adjust $P'$ by taking the union of $P'$ and a neighbourhood of the edge compression disc of the first isotopy, then it becomes a 3-manifold $P$ bounded by an annulus $A\subset S$ and $A''$. Performing a single annular simplification from $A$ to $A''$ eliminates the edge compression, and gives a clean annular simplification. We generalise this in the following lemma.

\begin{lemma}[Isotopy to normal, extra properties]
\label{Lem:GeneralisedIsotopyExtraProperties}
In the sequence of generalised isotopy moves in \reflem{AlmostNormGenIsotopyDisjointB}, we may in addition ensure that the following properties hold:
\begin{enumerate}
\item The annular simplifications are all clean.
\item For each surface in the sequence before the compression isotopies, the surface is disjoint from the interior of each component of $\calB$, unless the first isotopy is a parallelisation isotopy or a tube compression, in which case it is the only generalised isotopy move that is performed.
\end{enumerate}
\end{lemma}

\begin{proof}
We suppose that a parallelisation isotopy or a tube compression is not performed, as in this case, it is the only move in the sequence and the lemma is automatically true.

We now explain how to make all the annular simplifications clean. Suppose that in the sequence of \reflem{AlmostNormGenIsotopyDisjointB}, there is some annular simplification that is not clean. Consider the first such annular simplification. This involves an annular component $B$ of $\calB$. It moves an annulus $A'$ in $S'$ across to a vertical boundary component $A''$ of $B$. Let $S''$ be the resulting surface. By definition of an annular simplification, $A' \cup A''$ bounds a 3-manifold $P'$ that lies in a 3-ball or is a product region between $A'$ and $A''$. Now, $S$ is isotopic to $S'$, and before the annular simplification, this has not moved $S \cap \partial_h B$,
except possibly in the first step if that step is an edge compression isotoping slightly into $B$.
Hence, there is a corresponding annulus $A$ in $S$ and a corresponding 3-manifold $P$ bounded by $A'' \cup A$. As we are considering the first annular simplification that is not clean, all previous ones were clean. Hence, we can perform these, taking $S$ to some surface $\tilde S'$. For any generalised isotopies that were applied to $S$ outside of $P$, we apply the same isotopies to $\tilde S'$, giving a surface $\hat S'$.
We can then perform the annular simplification involving $B$ taking $\hat S'$ to $S''$. This involves exactly the same number of annular simplification as before, but it involves no other generalised isotopies within $P$. Thus, it is clean. Hence, we have created a shorter sequence of generalised isotopy moves satisfying the conclusions of \reflem{AlmostNormGenIsotopyDisjointB}. So in a shortest such sequence, all annular simplifications are clean.

This actually implies that before the compression isotopies, the surfaces are disjoint from the interior of each component of $\calB$. For the only way that this can be violated is at the first move, when an isotopy is performed along an edge compression disc, taking the surface a little into an annular component of $\calB$. Such a component must be incoherent. Therefore, at some stage, the surface must be isotoped across this annular $I$-bundle, necessarily by an annular simplification. But this annular simplification is then not clean, because the isotopy along the edge compression disc precedes it.
\end{proof}

\begin{remark}
A consequence of the above proof is that the sequence of generalised isotopy moves does not necessarily start with an isotopy along an edge compression disc. Instead, we may start with annular simplifications which end up by removing the almost normal piece of $S$. Alternatively, we may perform a parallelisation isotopy or a tube compression.
\end{remark}

The following result is similar to Lemma~\ref{Lem:AlmostNormGenIsotopyDisjointB} but where we start with a normal surface and apply an annular simplification.

\begin{lemma}[Normal to normal, disjoint $\calB$]\label{Lem:NormGenIsotopyDisjointB}
Let $S$ be a normal surface that is transversely oriented. Let $\calB$ be a maximal generalised parallelity bundle for the component of $M \cut S$ into which $S$ points. Suppose that $S$ admits an annular simplification in this direction. Then there is a sequence of generalised isotopy moves, all in the direction specified by the transverse orientation of $S$, taking $S$ to a normal surface, with the following properties:
\begin{enumerate}
\item The annular simplifications are all clean, and at least one annular simplification is performed.
\item If any compression isotopies are performed, these all take place at the end of the sequence.
\item For each surface in the sequence before the compression isotopies, the surface is disjoint from the interior of each component of $\calB$.
\item All surfaces in the sequence are nearly normal.
\item Each incoherent component of $\calB$ incident to $S$ that is not an $I$-bundle over a disc is isotoped across in this sequence.
\end{enumerate}
\end{lemma}

\begin{proof}
This is very similar to the proof of Lemma \ref{Lem:AlmostNormGenIsotopyDisjointB}. Consider all the incoherent components of $\calB$ that are not $I$-bundles over discs and that are incident to $S$. 
Since $S$ admits an annular simplification, there is at least one such component of $\calB$.
If there is a boundary-trivial component, then the surface admits an annular simplification by \reflem{BoundTrivAdmitsAnnularSimp}. If there is an essential incoherent annular $I$-bundle, then we may perform an annular simplification along an extrememost one. Since these annular simplifications occur before any other generalised isotopy moves, they are clean. We keep doing this until we end with a surface $S'$ such that every incoherent component of $\calB$ on the side of $S'$ into which it points is an $I$-bundle over a disc. If this surface has a face compression disc, we apply a generalised isotopy move across it. We keep doing this until the resulting surface has no face compression discs. We then apply compression isotopies as many times as possible. The result is the final normal surface.
\end{proof}

Lemmas \ref{Lem:GeneralisedIsotopyExtraProperties} and \ref{Lem:NormGenIsotopyDisjointB} start with a normal or almost normal surface $S$ and they work with a fixed maximal generalised parallelity bundle $\calB$ for $M \cut S$. Thus, the generalised isotopy moves are defined in terms of $\calB$ and they mostly keep this sequence of surfaces disjoint from the interior of $\calB$.

There are some circumstances where we might need to iterate this procedure. The following lemma is an application of Lemmas~\ref{Lem:GeneralisedIsotopyExtraProperties}
and~\ref{Lem:NormGenIsotopyDisjointB} multiple times until no further generalised isotopy moves are possible.

\begin{lemma}[Iterating generalised isotopy moves]\label{Lem:IterationGenIsotopy}
Let $S$ be an almost normal or normal surface that is transversely oriented. Then there is a sequence of disjoint transversely oriented surfaces $S= S_0, \dots, S_n$ with the following properties:
\begin{enumerate}
\item Each $S_i$ points towards $S_{i+1}$, and the final surface $S_n$ points away from the others.
\item For all $i > 0$, $S_i$ is normal.
\item The final surface $S_n$ admits no generalised isotopy moves.
\end{enumerate}
For each $i$, let $\mathcal{B}_i$ be a maximal generalised parallelity bundle for the manifold between $S_i$ and $S_{i+1}$. Then there is  a sequence of generalised isotopy moves taking $S_i$ to $S_{i+1}$ with the following properties:
\begin{enumerate}
\item If a generalised isotopy along an edge compression disc is performed, this is the first move of the sequence.
\item If a tube compression or a parallelisation isotopy is performed, this is the only move of the sequence.
\item If any compression isotopies are performed, these all take place at the end of the sequence.
\item For each surface in the sequence before the compression isotopies, the surface is disjoint from the interior of each component of $\calB_i$, unless the first isotopy is a parallelisation isotopy or a tube compression, in which case it is the only generalised isotopy move that is performed.
\item From the second surface onwards, all surfaces in the sequence are nearly normal.
\item The annular simplifications are all clean.
\item When $S_i$ is normal, for every incoherent component $B$ of $\calB_i$ incident to $S_i$ such that $B$ is not an $I$-bundle over a disc, we isotope across $B$ in this sequence.
\end{enumerate}
\end{lemma}

\begin{proof}
We start with $S = S_0$. If $S_0$ is normal and admits no annular simplifications, then we set $n =0 $ and stop. Otherwise, we may apply Lemmas \ref{Lem:AlmostNormGenIsotopyDisjointB} and \ref{Lem:GeneralisedIsotopyExtraProperties} or Lemma \ref{Lem:NormGenIsotopyDisjointB} to get a sequence of generalised isotopy moves taking $S_0$ to a normal surface $S_1$. If this admits an annular simplification, we apply Lemma \ref{Lem:NormGenIsotopyDisjointB} again. We repeat until we reach a normal surface $S_n$ that admits no annular simplifications. This must exist because at each stage when going from $S_i$ to $S_{i+1}$ (except possibly in the parallelisation isotopy which happens only once), the weight of the surface strictly decreases.
\end{proof}

\section{Surfaces and generalised isotopy moves}\label{Sec:GenIsotopy}

\textbf{Road map:} The generalised isotopy moves defined in \refsec{ParallelityBundles} will give us a sequence of surfaces in $S\times [0,1]$. Eventually, we want to show that there is a bounded number of moves required to transfer a spine on $S\times\{0\}$ to $S\times\{1\}$. These spines will step through a sequence of surfaces produced by generalised isotopy moves. 

Each of the generalised isotopy moves is defined in terms of a given fibre $S$ in $S \times [0,1]$ and a maximal generalised parallelity bundle $\calB$ for $M \cut S$.
In this section, we gather results that determine properties of this bundle $\calB$. These will be used in the proof of the main theorem.

\begin{definition}[Normally cylindrical/acylindrical]
A properly embedded normal or almost normal 2-sided surface $S$ in $M$ is \emph{normally cylindrical} if there exists an annulus $A$ embedded in $M$ with the following properties:
\begin{enumerate}
\item $A \cap S = \partial A$;
\item each curve of $\partial A$ is essential in $S$;
\item $A$ lies in the parallelity bundle of $M \cut S$ and is vertical in it;
\item near both components of $\partial A$, $A$ emanates from the same side of $S$.
\end{enumerate}
We also say that $S$ is normally cylindrical on the side containing $A$.
If no such annulus exists as above, then we say that $S$ is \emph{normally acylindrical}. A surface can also be normally acylindrical on one side.
\end{definition}

\begin{definition}[Coherently bundled]\label{Def:CoherentlyBundled}
Let $\mathcal{H}$ be a handle structure for $S \times [0,1]$. Then $\mathcal{H}$ is \emph{coherently bundled} if no vertical boundary component of its maximal generalised
parallelity bundle is an incompressible annulus with both boundary components in $S \times \{ 0 \}$ or both boundary components in $S \times \{ 1 \}$.
\end{definition}

The following lemmas are immediate consequences of the definitions and \refthm{IncompressibleHorizontalBoundary}.

\begin{lemma}[Coherently bundled alternative]
\label{Lem:CoherentlyBundledAlternative}
Let $\mathcal{H}$ be a handle structure for $S \times [0,1]$. Then $\mathcal{H}$ is coherently bundled if and only if
it admits a maximal generalised parallelity bundle with the property that every incoherent component either is an $I$-bundle over a disc or is boundary-trivial.\qed
\end{lemma} 

\begin{lemma}[Acylindrical fibre gives coherently bundled]\label{Lem:CutFibredAlongAcylindrical}
Let $\mathcal{H}$ be a handle structure for a closed orientable 3-manifold $M$ that fibres over the circle. Let $S$ be a normal fibre. Then the handle structure that $M \cut S$ inherits is coherently bundled if and only if $S$ is normally acylindrical.\qed
\end{lemma} 

\begin{lemma}[Acylindrical fibre in product gives coherently bundled]\label{Lem:CutProductAlongAcylindrical}
Let $\mathcal{H}$ be a coherently bundled handle structure for $M = S \times [0,1]$. Let $S'$ be a normal fibre. Then the handle structure that $M \cut S'$ inherits is coherently bundled if and only if $S'$ is normally acylindrical. \qed
\end{lemma}

\begin{lemma}[Acylindrical one side gives coherent one side]\label{Lem:AcylindricalOneSide}
Let $\mathcal{H}$ be a coherently bundled handle structure for $M = S \times [0,1]$. Let $S'$ be a normal fibre that is normally acylindrical on one side. Let $\mathcal{H}'$ be the handle structure for the component of $M \cut S'$ on that side. Then $\mathcal{H}'$ is coherently bundled. \qed
\end{lemma}

Suppose $\calH$ is a pre-tetrahedral handle structure of $M=S\times[0,1]$. Then the maximal generalised parallelity bundle for $\calH$, which we will denote by $\calB(\calH)$, has components that can be partitioned into three subsets, by \refthm{IncompressibleHorizontalBoundary}:
\begin{itemize}
\item $\mathcal{B}_D(\mathcal{H})$, the $I$-bundles over discs;
\item $\mathcal{B}_\partial(\mathcal{H})$, the boundary-trivial components;
\item $\mathcal{B}_I(\mathcal{H})$, the remaining components, which have incompressible horizontal boundary and incompressible vertical boundary.
\end{itemize}

\begin{lemma}[Incoherent components for $M\cut S$ in coherent bundle for $M$]\label{Lem:B_ADisjointB_IH}
Let $\calH$ be a pre-tetrahedral handle structure of $M=S\times[0,1]$, with components of its maximal generalised parallelity bundle partitioned into $\calB_D(\calH)$, $\calB_\bdy(\calH)$, and $\calB_I(\calH)$ as above. Let $\calB_I^{\mathrm{coh}}(\calH)$ denote the coherent components of $\calB_I(\calH)$. Suppose that $\Delta(\calH) > 0$.
\begin{itemize}
\item If $\mathcal{H}$ contains a normal fibre that is not normally parallel to a boundary component, then let $S$ be one such fibre with least weight.
\item If every normal fibre in $\mathcal{H}$ is normally parallel to a boundary component, then let $S$ be an almost normal fibre with least weight.
\end{itemize}
Let $\calB$ be a maximal generalised parallelity bundle for the handle structure obtained by cutting along $S$.
Let ${\calB_{\calA}}$ be the union of the incoherent essential annular components of $\calB$ that are incident to $S$. Then ${\calB_{\calA}}$ is disjoint from $\calB_I^{\mathrm{coh}}(\calH)$.
\end{lemma}

\smallskip
\begin{remark}
On notation, observe that $\calB_{*}(\calH)$ denotes a union of components of the maximal generalised parallelity bundle for all of $M$, whereas $\calB_{\calA}$ denotes components of the maximal generalised parallelity bundle for $M\cut S$. 
\end{remark}
\smallskip
  
\begin{proof}
We first show that no component of ${\calB_{\calA}}$ lies in the interior of $\calB_I^{\mathrm{coh}}(\calH)$.
For suppose instead that a component $B$ of ${\calB_{\calA}}$ lies completely in the interior of $\calB_I^{\mathrm{coh}}(\calH)$. The component of $\calB_I^{\mathrm{coh}}(\calH)$ containing $B$ consists of a component $P$ of the parallelity bundle for $\calH$, possibly with some $D^2 \times I$ pieces attached to it. Now $B$ cannot intersect $P$, since $S$ would have to be horizontal within $P$ and hence $B$ would intersect $\partial_v \calB_I^{\mathrm{coh}}(\calH)$. 
Hence, $B$ lies within the $D^2\times I$ pieces of this component of $\calB_I^{\mathrm{coh}}(\calH)$, which is a collection of 3-balls.
But this implies that the horizontal boundary of this component of ${\calB_{\calA}}$ is inessential in $S$, which contradicts the definition of ${\calB_{\calA}}$.

We now show that ${\calB_{\calA}}$ is disjoint from $\calB_I^{\mathrm{coh}}(\calH)$. To do this, it suffices to show that it is disjoint from $\bdy_v \calB_I^{\mathrm{coh}}(\calH)$, because, in the previous paragraph, we showed that no component of ${\calB_{\calA}}$ lies in the interior of $\calB_I^{\mathrm{coh}}(\calH)$. Suppose that ${\calB_{\calA}}$ intersects a component $A$ of $\bdy_v \calB_I^{\mathrm{coh}}(\calH)$.
{Thus $A$ is an annulus with boundary components on $S\times\{0\}$ and $S\times\{1\}$ by \refdef{Coherent}.}
Then $S \cap A$ would be at least two parallel core curves of $A$. In fact, because $\calB_I^{\mathrm{coh}}(\calH)$ is coherent and because $S$ separates $S \times \{ 0 \}$ from $S \times \{ 1 \}$, we deduce that any essential arc in $A$ must intersect $S$ an odd number of times. Hence $S \cap A$ is at least three core curves of $A$. Pick three such core curves that are adjacent in $A$, and let $A_-$ and $A_+$ be the annuli between them in $A$.
Since $S$ is separating, we may choose a transverse orientation on $S$ so that, near $\bdy A_-$, it points into $A_-$, and near $\bdy A_+$, it points out of $A_+$. Because $A$ is incompressible, $\bdy A_-$ bounds an annulus $\tilde A_-$ in $S$, and similarly $\bdy A_+$ bounds an annulus $\tilde A_+$ in $S$. See \reffig{BADisjointBIH}, left.

\begin{figure}
  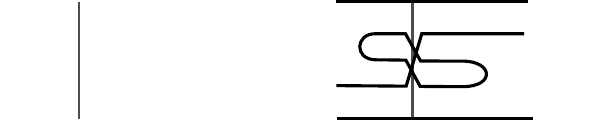
  \caption{Left: Schematic picture of annuli $A$ in $S\times[0,1]$, $A_-$ and $A_+$ in $A$, and $\tilde A_-$ and $\tilde A_+$ in $S$. Right: $S$ is the normal sum of $S'$ and the torus $\tilde A_- \cup \tilde A_+$.}
  \label{Fig:BADisjointBIH}
\end{figure}

Suppose first that $\tilde A_-$ and $\tilde A_+$ emanate from the curve $A_- \cap A_+$ in opposite directions in $S$. Then the interiors of $\tilde A_-$ and $\tilde A_+$ are disjoint, unless $S$ is a torus. So suppose for the moment that $S$ is not a torus. Then we may remove $\tilde A_- \cup \tilde A_+$ from $S$ and attach the boundary components of the resulting surface, to obtain a normal or almost normal fibre $S'$. We claim that $S'$ is not normally parallel to a boundary component of $M$. One way to see this is to note that $S$ is the normal sum of $S'$ and a torus.
This torus is obtained from $\tilde A_-$ and $\tilde A_+$ by perturbing using a small normal isotopy and then gluing their boundary components together in pairs, as in \reffig{BADisjointBIH}, right. (The normal sum is obtained by cutting along the parallel normal curves and reconnecting into the surface $S$; see, for example \cite{JacoOertel} for information on normal sums.)
If $S'$ is normally parallel to a boundary component, then this normal sum would just be a disjoint union, which is a contradiction. 
Now, $S'$ is a fibre and it has smaller weight than $S$. If $S$ was normal, then so is $S'$. If $S$ is almost normal, then $S'$ may be normal or almost normal. But in all cases, this contradicts our hypotheses. 

Suppose now that $\tilde A_-$ and $\tilde A_+$ emanate from the curve $A_- \cap A_+$ in opposite directions in $S$, but that the interiors of $\tilde A_-$ and $\tilde A_+$ overlap. Then $S$ is a torus and the three curves $\bdy \tilde A_- \cup \bdy \tilde A_+$ are parallel curves on $S$. Hence, we may consider $S \cut \tilde A_-$ instead of $\tilde A_-$ and consider $S \cut \tilde A_+$ instead of $\tilde A_+$. These annuli then have disjoint interiors, and we may then argue as above to deduce that $S$ is the normal sum of a fibre $S'$ and a torus, and thereby reach a contradiction.

Suppose now that $\tilde A_-$ and $\tilde A_+$ emanate from $A_- \cap A_+$ in the same direction in $S$. Because the other curves of $\bdy A_-$ and $\bdy A_+$ are disjoint, we deduce that $\tilde A_-$ and $\tilde A_+$ are nested. Say that $\tilde A_- \subset \tilde A_+$. Note that near its two boundary curves, $\tilde A_-$ emanates from the same side of $A$. For otherwise, $\tilde A_- \cup A_-$ is the union of two annuli embedded in $S\times[0,1]$ with matching boundary components. This union is a Klein bottle, because the transverse orientations on $S$ at the two components of $\partial A_-$ are oppositely oriented. But a Klein bottle does not embed in $S\times[0,1]$, which is a contradiction.
The same argument proves that $\tilde A_+$ emanates from the same side of $A$. 
But then the annulus $\tilde A_+ \cut \tilde A_-$ emanates from opposite sides of $A$ near its boundary.
We can reglue the boundary components of $\tilde A_+ \cut \tilde A_-$ to form an embedded torus $T$. The surface $S$ is the normal sum of $T$ 
and another normal or almost normal fibre $S'$. As argued above, this has smaller weight than $S$ and is not normally parallel to a boundary component of $M$. This is again a contradiction.
\end{proof}

\begin{lemma}[Least weight surface is acylindrical]\label{Lem:AcylindricalExists}
Let $M$ be a closed orientable 3-manifold that fibres over the circle. Then, for any pre-tetrahedral handle structure $\mathcal{H}$ of $M$, there is a normal fibre that is normally acylindrical. Indeed, any normal fibre that has least weight in its isotopy class is normally acylindrical and admits no annular simplifications.
\end{lemma}

\begin{proof} 
Let $S$ be a fibre that is standard with respect to $\mathcal{H}$ and that has least weight among all such fibres. Then we may isotope $S$ to a normal surface without increasing the weight: the fact that such an isotopy exists is well-known. See, for example \cite[Theorem~3.4.7]{Matveev:Book} or \cite[Proposition~4.4]{Lackenby:CrossingNo}. 

Suppose that $S$ admits an annular simplification. Then this reduces the weight of $S$ by \reflem{AnnSimpReducesWeight}, which is a contradiction. Suppose that $S$ is normally cylindrical. Then by Lemma \ref{Lem:CutFibredAlongAcylindrical}, $M \cut S$ is not coherently bundled. So, by \reflem{CoherentlyBundledAlternative}, it has an essential incoherent annular component of its maximal generalised parallelity bundle. So $S$ admits an annular simplification, which we have established not to be the case.
\end{proof}

The next lemma ensures that a normal surface in $M\cut S$ is also normal in $M$.

\begin{lemma}[Normal in $M\cut S$ is normal in $M$]\label{Lem:NormalTwoHandleStructures}
Let $\calH$ be a handle structure of a compact orientable 3-manifold $M$. Let $S$ be a closed normal surface in $M$. Let $\calH'$ be the handle structure that $M \cut S$ inherits. Let $S'$ be a closed surface in $M \cut S$ that is normal with respect to $\calH'$. Then $S'$ is normal with respect to $\calH$ in $M$.
\end{lemma}

\begin{proof}
Let $H_0$ be a 0-handle of $\mathcal{H}$ and let $D$ be a component of $S' \cap H_0$.  Suppose that $D$ is not normal. Then it runs over a component $E$ of intersection between $H_0$ and the 2-handles more than once. This component $E$ is divided up by $S$. However, since $S$ is normal in $\mathcal{H}$, each component of intersection between $S$ and $H_0$ runs over $E$ at most once. Hence, any 0-handle of $H_0 \cap \mathcal{H}'$ intersects $E$ in either the empty set or a single disc. Therefore, if $D$ runs over $E$ more than once, then $S'$ is not normal in $\mathcal{H}'$, which is a contradiction.
\end{proof}

Similarly, if a surface is normally acylindrical in $M\cut S$, then it is normally acylindrical in $M$, under the following assumptions.

\begin{lemma}[Normally acylindrical in $M\cut S$ and $M$]\label{Lem:AcylindricalTwoHandleStructures}
Let $\mathcal{H}$ be a handle structure of a compact orientable 3-manifold $M$. Let $S$ be a closed separating incompressible normal surface in $M$ that is normally acylindrical. Let $\mathcal{H}'$ be the handle structure that $M \cut S$ inherits. Let $S'$ be a closed incompressible connected normal surface in $\mathcal{H}'$ that is normally acylindrical with respect to $\mathcal{H}'$. Then $S'$ is a normal surface in $\mathcal{H}$ that is normally acylindrical with respect to $\mathcal{H}$.
\end{lemma}

\begin{proof}
By \reflem{NormalTwoHandleStructures}, $S'$ is normal in $\mathcal{H}$.
Suppose that $S'$ is normally cylindrical in $\mathcal{H}$, via an annulus $A$. If this annulus has empty intersection with $S$, then $S'$ is normally cylindrical with respect to $\mathcal{H}'$, which is contrary to assumption. Therefore, $A$ is divided up by curves of $S \cap A$.

Observe first that no component of $A\cap S$ is homotopically trivial in $A$, because $A$ is vertical in the parallelity bundle of $M \cut S'$, and $S$ is normal in $M$, hence $A\cap S$ consists of core curves of $A$.

If there is more than one circle of $A\cap S$, then two adjacent circles bound an annulus in $A$ with interior disjoint from $S$ and which emanate from the same side of $S$. Since $S$ is normally acylindrical, we deduce that these curves must be inessential in $S$. But then the curves $A \cap S'$ bound discs in the complement of $S'$, which contradicts the assumption that $S'$ is incompressible.

So suppose that there is just one curve of $A \cap S$. Then an essential arc in $A$ intersects this curve just once, and using the fact that $S'$ is connected, we may join the endpoints of this arc by a curve parallel to $S'$. This implies that $S$ is non-separating, which is a contradiction.
\end{proof}

\begin{proposition}[Acylindrical one side stays acylindrical one side]\label{Prop:CylindricalWrongSide}
Let $\calH$ be a pre-tetrahedral handle structure for $M = S \times [0,1]$. Let $S$ be a normal fibre that is normally cylindrical on at most one side. Transversely orient $S$. If $S$ is cylindrical on one side, suppose that it is transversely oriented in this direction. Let $S'$ be a normal fibre that is obtained from $S$ by performing generalised isotopy moves, all in this transverse direction. Then $S'$ is normally acylindrical on the side into which $S'$ does not point.
\end{proposition}

\begin{proof}
Each time that we perform an annular simplification, a generalised isotopy along an $I$-bundle over a disc, a
generalised isotopy along a face compression disc, or a compression isotopy, we remove points of intersection with 2-handles of $\calH$ without adding more such intersections. Thus, we can view each component of intersection between one of the surfaces and the 2-handles of $\calH$ as corresponding to a component of intersection with the previous surface and the 2-handles of $\calH$. In order to ensure that all these surfaces in this isotopy are disjoint, each time we perform an isotopy, we push the entire surface a little in its specified transverse orientation (as in Remark \ref{Rem:PairwiseDisjointSurfaces}).

Suppose that $S'$ is cylindrical on the side into which it does not point, via an annulus $A$. The annulus $A$ is vertical in the parallelity bundle of $M \cut S'$. Thus by isotoping $A$ to the boundary of some parallelity 0-handle through which it runs, we may ensure that $A$ contains the co-core of a 2-handle of $M \cut S'$. This co-core will be an essential arc $\alpha$ in $A$, with both endpoints of $\alpha$ on $S'$. See \reffig{AnnulusA}.

\begin{figure}
  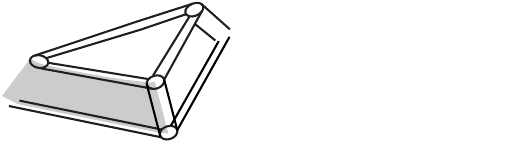
  \caption{The annulus $A$ may be taken to run through the co-core of a 2-handle, intersecting it in an arc $\alpha$ that is essential in $A$, with both endpoints on $S'$.}
  \label{Fig:AnnulusA}
\end{figure}

If $S$ were to intersect $A$, it would do so in a collection of parallel core curves of $A$, because $S$ is normal.
Since each component of intersection of $S'$ with 2-handles corresponds to a component of intersection of $S$ with 2-handles, the two intersections between $\alpha$ and $S'$ at the endpoints of $\alpha$ also give points of intersection with $S$. These must lie in the interior of $\alpha$, because the transverse orientation on $S'$ points away from $\alpha$.
Hence $S \cap A$ contains at least two core curves, one parallel to one component of $\partial A$ and the other parallel to the other component of $\partial A$. These core curves point towards the relevant component of $\partial A$. We deduce that there are two adjacent curves of $S \cap A$ that co-bound an annulus $A'$ in $A$, where $S$ points out of $A'$ near $\partial A'$. Hence, $S$ is normally cylindrical on the side into which it does not point, contrary to assumption.
\end{proof}

\begin{corollary}[Maximal isotopies, cylindrical one side gives normal]\label{Cor:IsotopeCylindricalOneSide}
Suppose $\calH$ is a pre-tetrahedral handle structure for $M = S \times [0,1]$. Suppose that $M$ contains no normal fibre that is normally acylindrical, other than those that are normally parallel to $S \times \{ 0 \}$ or $S \times \{ 1 \}$. Let $S$ be a transversely oriented normal fibre that is cylindrical on one side only, the side into which it points. Let $S'$ be a surface that is obtained from $S$ by performing generalised isotopy moves,   
all in the specified transverse direction, until no further generalised isotopy moves are possible. Then $S'$ is normal and boundary parallel.
\end{corollary}

\begin{proof}
By \refprop{UnsimplifiableImpliesNormal}, $S'$ is normal. By \refprop{CylindricalWrongSide}, $S'$ is acylindrical on the side into which it does not point. It cannot be cylindrical on the side into which it points, as otherwise an annular simplification could be performed upon it. Hence, it is acylindrical and therefore, by hypothesis, it is boundary parallel.
\end{proof}

\begin{proposition}[Maximal isotopies, everything cylindrical]\label{Prop:IsotopeEverythingCylindrical}
Suppose $\calH$ is a pre-tetrahedral handle structure for $M = S \times [0,1]$. Suppose that every normal fibre is cylindrical on the $S \times \{ 1 \}$ side at least, with the exception of those that are normally parallel to $S \times \{ 0 \}$ or $S \times \{ 1 \}$. Let $S$ be a normal or almost normal fibre that is not boundary parallel. Let $S'$ be a surface that is obtained from $S$ by a maximal sequence of generalised isotopy moves,  
all in the $S \times \{ 1 \}$ direction. Then $S'$ is normally parallel to $S \times \{ 1 \}$.
\end{proposition}

\begin{proof}
  The surface $S'$ is normal, by \refprop{UnsimplifiableImpliesNormal}. If it is not parallel to $S \times \{ 1 \}$ it admits an annular simplification in the $S \times \{ 1 \}$ direction, by hypothesis. Because $S'$ admits no such simplification, it must therefore be normally parallel to $S \times \{ 1 \}$.
\end{proof}

\begin{proposition}[Isotopies from a least weight fibre]\label{Prop:IsotopeFromLeastWeight}
Let $M=S\times[0,1]$, and suppose $\calH$ is a pre-tetrahedral handle structure for $M$. 
\begin{itemize}
\item If there is a normal fibre that is not normally parallel to $S \times \{ 0 \}$ or $S \times \{ 1 \}$, let $S$ be one with least weight. In this case, suppose that $S$ is normally cylindrical on the $S \times \{ 1 \}$ side.
\item If the only normal fibres are those that are normally parallel to $S \times \{ 0 \}$ or $S \times \{ 1 \}$, let $S$ be an almost normal fibre with least weight.
\end{itemize}
Then there is a sequence of generalised isotopies satisfying the conclusions of Lemma \ref{Lem:AlmostNormGenIsotopyDisjointB} and \ref{Lem:GeneralisedIsotopyExtraProperties} (in the case where $S$ is almost normal) or \ref{Lem:NormGenIsotopyDisjointB} (in the case where $S$ is normal) taking $S$ to $S \times \{ 1 \}$.
\end{proposition}

\begin{proof} Orient $S$ towards $S \times \{ 1 \}$. By Lemmas \ref{Lem:AlmostNormGenIsotopyDisjointB} and \ref{Lem:GeneralisedIsotopyExtraProperties}, or Lemma \ref{Lem:NormGenIsotopyDisjointB}, there is a sequence of generalised isotopy moves as in those lemmas, taking $S$ to a normal surface $S'$. When $S$ is almost normal, then by assumption, the only normal fibres are parallel to $S \times \{ 0 \}$ or $S \times \{1 \}$ and so $S'$ must be parallel to $S \times \{ 1 \}$. When $S$ is normal, then $S'$ has smaller weight than $S$ and hence again it must be parallel to $S \times \{ 1 \}$.
\end{proof}

\begin{definition}[Innermost]\label{Def:InnermostCylindricalSurface}
Let $\mathcal{H}$ be a handle structure of $M = S \times [0,1]$. Let $S$ be a normal fibre that is cylindrical on exactly one side. Then $S$ is \emph{innermost on the acylindrical side}
in $M$ if, for every normal fibre $S'$ that is disjoint from $S$, that lies on the normally acylindrical side of $S$ and that is normally cylindrical in $M$ on the side pointing towards $S$ and normally acylindrical on the other side, $S'$ is normally parallel to $S$. See \reffig{InnermostDef}.
\end{definition}

Observe that if there is a normal fibre that is cylindrical on exactly one side, then there is one that is innermost on the acylindrical side. This follows from 
Kneser's upper bound on the number of disjoint normal surfaces; for example see Matveev~\cite{Matveev:Book}. 

\begin{figure}
\begingroup%
  \makeatletter%
  \providecommand\color[2][]{%
    \errmessage{(Inkscape) Color is used for the text in Inkscape, but the package 'color.sty' is not loaded}%
    \renewcommand\color[2][]{}%
  }%
  \providecommand\transparent[1]{%
    \errmessage{(Inkscape) Transparency is used (non-zero) for the text in Inkscape, but the package 'transparent.sty' is not loaded}%
    \renewcommand\transparent[1]{}%
  }%
  \providecommand\rotatebox[2]{#2}%
  \newcommand*\fsize{\dimexpr\f@size pt\relax}%
  \newcommand*\lineheight[1]{\fontsize{\fsize}{#1\fsize}\selectfont}%
  \ifx\svgwidth\undefined%
    \setlength{\unitlength}{293.47472763bp}%
    \ifx\svgscale\undefined%
      \relax%
    \else%
      \setlength{\unitlength}{\unitlength * \real{\svgscale}}%
    \fi%
  \else%
    \setlength{\unitlength}{\svgwidth}%
  \fi%
  \global\let\svgwidth\undefined%
  \global\let\svgscale\undefined%
  \makeatother%
  \begin{picture}(1,0.31896939)%
    \lineheight{1}%
    \setlength\tabcolsep{0pt}%
    \put(0,0){\includegraphics[width=\unitlength,page=1]{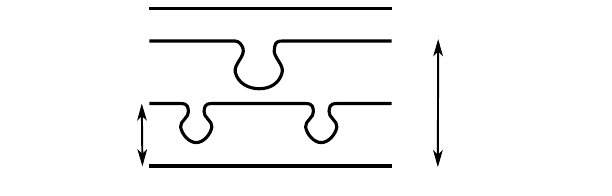}}%
    \put(0.65319853,0.23908072){\color[rgb]{0,0,0}\makebox(0,0)[lt]{\lineheight{1.25}\smash{\begin{tabular}[t]{l}$S$\end{tabular}}}}%
    \put(0.65319853,0.29274797){\color[rgb]{0,0,0}\makebox(0,0)[lt]{\lineheight{1.25}\smash{\begin{tabular}[t]{l}$S\times\{1\}$\end{tabular}}}}%
    \put(0.73287805,0.16752438){\color[rgb]{0,0,0}\makebox(0,0)[lt]{\lineheight{1.25}\smash{\begin{tabular}[t]{l}$S$ acylindrical\\this side\end{tabular}}}}%
    \put(0.65319853,0.13685717){\color[rgb]{0,0,0}\makebox(0,0)[lt]{\lineheight{1.25}\smash{\begin{tabular}[t]{l}$S'$\end{tabular}}}}%
    \put(0.62764265,0.00652203){\color[rgb]{0,0,0}\makebox(0,0)[lt]{\lineheight{1.25}\smash{\begin{tabular}[t]{l}$S\times\{0\}$\end{tabular}}}}%
    \put(-0.00103155,0.10619055){\color[rgb]{0,0,0}\makebox(0,0)[lt]{\lineheight{1.25}\smash{\begin{tabular}[t]{l}$S'$ acylindrical\\this side\end{tabular}}}}%
    \put(0,0){\includegraphics[width=\unitlength,page=2]{InnermostDefn.pdf}}%
  \end{picture}%
\endgroup%

  \caption{If $S$ as shown is innermost, then $S'$ must be normally parallel to $S$.}
  \label{Fig:InnermostDef}
\end{figure}

\begin{lemma}[Cutting along innermost]\label{Lem:CutAlongInnermost}
Let $\mathcal{H}$ be a handle structure of $M = S \times [0,1]$. Suppose that every normal fibre in $M$ is normally cylindrical, with the exception of those that are boundary parallel. Let $S$ be a normal fibre that is normally cylindrical on exactly one side. Suppose that $S$ is innermost. Then, in the component of $M \cut S$ on the normally acylindrical side of $S$, every normal fibre is either normally cylindrical in $M \cut S$ on the side away from $S$, or boundary parallel.
\end{lemma}

\begin{proof}
Suppose that, on the contrary, there is a normal fibre $S'$ in $M$ on the normally acylindrical side of $S$, such that $S'$ not boundary parallel, but $S'$ is normally acylindrical in $M\cut S$ on the side away from $S$. This fibre $S'$ is normal in $M$ by \reflem{NormalTwoHandleStructures}. It is normally cylindrical in $M$ by hypothesis, but because $S$ is innermost, $S'$ cannot be normally cylindrical in $M$ only on the side that points toward $S$. Note that if $S'$ is normally cylindrical in $M\cut S$ only on the side that points toward $S$, then it will still be normally cylindrical on this side in $M$, since the incoherent parallelity bundle for $S'$ in $M\cut S$ will be an incoherent parallelity bundle for $S'$ in $M$. Thus $S'$ must be normally cylindrical in $M$ on the side away from $S$. But then the incoherent parallelity bundle for $S'$ on this side is disjoint from $S$, and thus remains an incoherent parallelity bundle for $M\cut S$, and so $S'$ is normally cylindrical in $M\cut S$ on the side away from $S$. This contradicts our assumption. 
\end{proof}

\section{Spines on surfaces}\label{Sec:Spines}

\textbf{Road map:} We have built nearly normal surfaces and described generalised isotopy moves that interpolate between almost normal, nearly normal, and normal surfaces. By tracing through the effect of these moves on a surface $S\subset S\times[0,1]$, we will be able to transfer a spine from one nearly normal or almost normal surface to another. This section describes how the spine transfer is done.

Finally, recall that to prove the main theorem, we need to bound the number of edge contractions and expansions required when we transfer spines in $S\times[0,1]$. This section also give some initial upper bounds on the number of edge contractions and expansions under generalised isotopy moves.

\subsection{Edge swaps in the spine graph}\label{Sec:EdgeSwaps}
Given a sequence of nearly normal surfaces obtained from generalised isotopy moves, we know how to transfer the canonical handle structure across the surfaces. However, we actually want to be transferring spines across surfaces, and bounding the number of edge contractions and expansions required in order to do so. Given canonical handle structures, we now describe how to transfer spines.

Our main tool will be a modification to a spine called an edge swap.

\begin{definition}[Edge swap]\label{Def:EdgeSwap}
Let $\Gamma$ be a spine for a closed surface $S$. Let $e_1$ be an arc properly embedded in the disc $S \cut \Gamma$. Let $e_2$ be an edge of the graph $\Gamma \cup e_1$ that has distinct components of $S \cut (\Gamma \cup e_1)$ on either
side of it. Then the result of removing $e_2$ from $\Gamma$ and adding $e_1$ is a new spine $\Gamma'$ for $S$. We say that $\Gamma$ and $\Gamma'$ are related by an \emph{edge swap}.
\end{definition}

The following is an important example. 

\begin{lemma}[Edge swap and Dehn twist]\label{Lem:EdgeSwapDehnTwist}
Let $\Gamma$ be a spine for a closed orientable surface. Let $C$ be a simple closed curve intersecting $\Gamma$ transversely in a single point in the interior of an edge of $\Gamma$. Let $\Gamma'$ be obtained from $\Gamma$ by Dehn twisting about $C$. Then $\Gamma'$ is obtained from $\Gamma$ by a single edge swap.
\end{lemma}

\begin{proof}
Let $e_2$ be the edge of $\Gamma$ that intersects $C$. Let $e_1$ be an arc in $S \cut \Gamma$ joining the endpoints of $e_2$ and emanating from opposite sides of $e_2$. There are two possible choices for $e_1$. The spine $\Gamma'$ is obtained from $\Gamma$ by removing $e_2$ and adding one choice for $e_1$.
\end{proof}

The next lemma shows that a single edge swap is realised by a bounded number of edge contractions and expansions, so we will be able to bound distances in the spine graph by bounding edge swaps. 

\begin{lemma}[Edge swap and contraction/expansion]\label{Lem:EdgeSwapBound}
Let $S$ be a closed orientable surface. Let $\Gamma$ be a spine for $S$. Then an edge swap can be realised by a sequence of at most $24g(S)$ edge expansions and contractions.
\end{lemma}

\begin{proof}
Let $e_1$ and $e_2$ be as in \refdef{EdgeSwap}. Let $\Gamma'$ be obtained from $\Gamma$ by adding $e_1$ and removing $e_2$.
We first reduce to the case where $e_2$ is an edge of $\Gamma$. Suppose that it is not. Then $e_2$ is a sub-arc of an edge $e_3$ of $\Gamma$. 
Let $e_4$ be the arc $e_1 \cup {\rm cl}(e_3 - e_2)$. A small isotopy makes it disjoint from $e_3$. Then $\Gamma'$ is obtained from $\Gamma$ by removing $e_3$ and adding $e_4$. See \reffig{EdgeSwap}.

\begin{figure}
  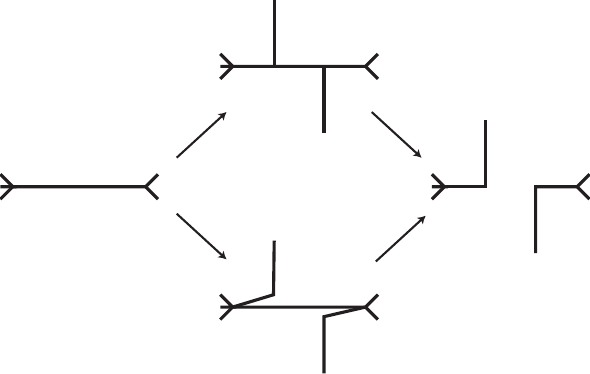
  \caption{Two ways of performing an edge swap}
  \label{Fig:EdgeSwap}
\end{figure}

Thus, we may assume that $e_2$ is an edge of $\Gamma$. Let $\Gamma_-$ be the result of removing $e_2$ from $\Gamma$. Then $S \cut \Gamma_-$ is an annulus $A$. The two edges $e_1$ and $e_2$ are both properly embedded in $A$.
They are both essential in $A$, since if $A$ is cut along either $e_1$ or $e_2$, the result in both cases is a disc. Hence, there is an isotopy of $A$ taking $e_2$ to $e_1$. We will show that this isotopy
can be achieved using a sequence of edge expansions and contractions, starting
with $\Gamma$ and ending with $\Gamma'$.

We perform the isotopy in three stages: first we move one endpoint of $e_2$ to an endpoint of $e_1$, then we move the other endpoint of $e_2$ to the other endpoint of $e_1$, in such a way that the interiors of $e_1$ and $e_2$ are then isotopic, then we move the interior of $e_2$ to the interior of $e_1$. The latter step is achieved by an isotopy of the spine and so requires no expansions or contractions. So we need only consider the first two stages.

If an endpoint of $e_2$ has valence more than 3, then we first perform an edge expansion so that the valence at this end of $e_2$ becomes exactly 3. Then we start to perform the isotopy. When we move the endpoint of $e_2$, it will, at various moments in time, move past a vertex of $\Gamma_-$ or move past the other endpoint of $e_2$. Each such move can be achieved by an edge contraction
followed by an edge expansion. So, the total number of contractions and expansions to move one endpoint is at most twice the number of vertices in one component of the boundary of $A$. 
To move the other requires at most twice the number of vertices in the other component of $\bdy A$. Thus we consider the number of vertices of both components of $\bdy A$. This may be more than the number of vertices of $\Gamma_-$ because when identifications are made on the boundary of $A$ to form $S$, distinct vertices of $A$ may become identified. But the number of vertices in $\partial A$ is equal to the number of edges in $\partial A$ and this is twice the number of edges in $\Gamma_-$.  \reflem{BoundOnVerticesAndEdges} implies there are at most $6 g(S)-4$ edges in $\Gamma_-$. Hence there are less than $24 g(S)$ edge expansions and contractions required.
\end{proof}

\begin{lemma}[Edge of a spine has essential dual curve]\label{Lem:RemoveEdgeSpine}
  Let $\Gamma$ be a spine for a closed orientable surface $S$. Let $e$ be an edge of $\Gamma$, and let $\Gamma_-$ be the graph obtained from $\Gamma$ by removing $e$. Then the core curve of the annulus $S \cut \Gamma_-$ is essential in $S$.
\end{lemma}

\begin{proof}
  Suppose that the core curve $C$ is inessential in $S$. Then it bounds a disc $D$. We may assume that $C$ intersects $e$ transversely at a single point. The intersection between $\Gamma$ and $D$ is a graph $G$ in $D$ that intersects $\partial D$ at one point. Since $D \cut G$ contains a single face, the one incident to $\partial D$, we deduce that $G$ is a tree. Hence it contains a vertex with valence $1$, other than the one on $\partial D$. This is a vertex with valence $1$ in $\Gamma$, contradicting the hypothesis that $\Gamma$ is a spine for $S$, \refdef{Spine}. 
\end{proof}

\begin{definition}[Associated cell structure from surface handle structure]\label{Def:AssociatedCellStructure}
Given a handle structure $\mathcal{H}$ for a closed surface $S$, its \emph{associated cell structure} is defined as follows. Each 2-cell is a handle of $\mathcal{H}$. Thus, its 1-cells are the components of intersection between pairs of handles. Its 0-cells are components of intersection between triples of handles. We say that a graph embedded in $S$ is \emph{cellular} if it is a subcomplex of this cell structure.
\end{definition}

\begin{definition}[Cellular spine]\label{Def:Cellular}
  Let $\mathcal{C}$ be a cell structure for $S$. A spine embedded in $S$ is \emph{cellular} if it is a subcomplex of $\mathcal{C}$. 
\end{definition}
    
\begin{lemma}[Moving spine off a disc]\label{Lem:SlidingOffADisc}
Let $S$ be a closed orientable surface with a cell structure $\mathcal{C}$, and with a cellular spine $\Gamma$. Let $D$ be a cellular subset of $S$ that is an embedded disc, and let $\ell$ be the length of $\partial D$, where each
1-cell of $\mathcal{C}$ is declared to have length $1$. Then there is a sequence
of at most $6g(S) + 2\ell$ edge swaps taking $\Gamma$ to a cellular spine $\Gamma'$ that is disjoint from the interior of $D$. Moreover, $\Gamma' - D$ is contained in $\Gamma - D$.
\end{lemma}

\begin{proof}
We will perform the modifications to $\Gamma$ in two stages. In the first stage,
we will remove edges that lie entirely in the interior of $D$. The number of such edges is at most the number of edges of $\Gamma$, which is at most $6g(S)$ by \reflem{BoundOnVerticesAndEdges}. Let $e$ be an edge of $\Gamma$ that lies entirely in the interior of $D$, and let $x$ be a point in the interior of $e$. We will perform an edge swap that removes $e$, but we need to find a suitable arc $e'$ to replace it. Removing $e$ from $\Gamma$ gives a graph $\Gamma_-$, the exterior of which is an annulus $A$. The arc $e$ is essential in $A$. Hence, there is an embedded core curve $C$ for $A$ that intersects $e$ once at exactly the point $x$. The boundary $\partial D$ intersects the interior of $A$ in a collection of properly embedded arcs. By choosing $C$ suitably, we may assume that $C$ intersects each of these arcs at most once. It has to intersect at least one of these arcs, by \reflem{RemoveEdgeSpine}. So, let $e'$ be one such arc. Then the edge swap removes $e$ and inserts $e'$.

After this procedure, $\Gamma$ intersects the interior of $D$ in a union of embedded arcs, each of which has at least one endpoint in $\partial D$. The closures of these arcs may intersect on $\partial D$. But these arcs may also intersect each other at their endpoints in the interior of $D$. So each component of $\Gamma \cap \mathrm{int}(D)$ is either an arc or a star-shaped graph. These arcs and graphs divide $D$ into discs, and each of these discs contains at least one 1-cell lying in $\partial D$, for otherwise cutting along $\Gamma$ would separate $S$ into at least two components, one entirely contained in the interior of $D$, contradicting the definition of a spine.
The number of these discs is more than half the number of arcs of $\Gamma \cap \mathrm{int}(D)$; this can be shown by induction on the number of components of intersection, using the fact that there are no vertices of valence one in a spine. Hence, the number of arcs is less than twice the length $\ell$ of $\partial D$.
Each modification that we make will reduce the number of arcs of ${\rm int}(D) \cap \Gamma$, and so the number of edge swaps needed in this stage will be at most $2\ell$. Let $e$ be an arc of ${\rm int}(D) \cap \Gamma$, and let $x$ be a point in $\mathrm{int}(e) \cap {\rm int}(D)$. Again, there is a closed embedded curve $C$ intersecting $\Gamma$ exactly at the point $x$. We may assume that $C$ intersects each arc of $\partial D \cut \Gamma$ at most once. Let $e_1$ be an arc of $\partial D \cut \Gamma$ that meets $C$ exactly once. Let $e_2$ be the edge of $\Gamma \cup e_1$ containing $x$. Then $\Gamma \cup e_1$ divides $S$ into two discs, which lie on either side of $e_2$. Hence, we may perform an edge swap that adds $e_1$ and removes $e_2$.
\end{proof}

\begin{corollary}[Moving spine off multiple discs]\label{Cor:SlidingOffDiscs}
Let $S$ be a closed orientable surface with a cell structure $\mathcal{C}$, and with a cellular spine $\Gamma$. Let $D_1, \dots, D_k$ be cellular subsets of $S$, each of which is an embedded disc, and with disjoint interiors. Let $\ell$ be the sum of the lengths of $\partial D_1, \dots, \partial D_k$. Then there is a sequence of at most $6k g(S) + 2\ell$ edge swaps taking $\Gamma$ to a cellular spine $\Gamma'$ that is disjoint from the interior of $D_1, \dots, D_k$. Moreover, $\Gamma' - (D_1 \cup \dots \cup D_k)$ is contained in $\Gamma - (D_1 \cup \dots \cup D_k)$.
\end{corollary}

\begin{proof} Use \reflem{SlidingOffADisc} to move the spine off the interior of $D_1$, then the interior of $D_2$, and so on. At the $i$th stage, we do not add to the spine away from $D_i$. Hence, when we remove it from the interior of $D_i$, the new spine does not go into the interior of an earlier $D_j$. Hence, after this process, the spine is disjoint from the interior of  $D_1, \dots, D_k$. Moreover, $\Gamma' - (D_1 \cup \dots \cup D_k)$ is contained in $\Gamma - (D_1 \cup \dots \cup D_k)$.
\end{proof}

\subsection{Moving spines on nearly normal surfaces}
In this subsection, surfaces will have a handle structure, and we will work with the associated cell structure of \refdef{AssociatedCellStructure}. We give concrete bounds on the number of edge swaps required to transfer a spine in the case of isotopy along edge and face compression discs and generalised isotopy across an $I$-bundle over a disc, i.e.\ generalised edge or face compressions as in \refdef{GenIsotopyMoveAcrossB}. 

\begin{remark}
Note that since the canonical handle structure is unchanged under a compression isotopy, the spine on one transfers to the other immediately.
\end{remark}

\begin{lemma}[Edge swap bound, edge and face compression discs]\label{Lem:MovingSpineIsotopyAcrossD}
Suppose $S$ is almost normal or nearly normal, and $S'$ is the nearly normal surface obtained from $S$ by an isotopy along an edge or face compression disc $D$. Let $v$ be the valence of the 2-handle meeting the endpoints of $D \cap S$, i.e.\ the number of components of intersection between this 2-handle and the 1-handles.
Give $S$ and $S'$ their canonical handle structures.
\begin{itemize}
\item If $S$ is equipped with a cellular spine $\Gamma$, then we may build a cellular spine on $S'$ from $\Gamma$ by a sequence of at most $\max \{ 0, 6(2v-3)g(S) +88v -152 \}$ edge swaps in the face compression case, and by at most $6(2v-1)g(S) + 88v-64$ edge swaps in the edge compression case. 
\item If instead $S'$ is equipped with a cellular spine $\Gamma'$, then we may build a cellular spine on $S$ from $\Gamma'$ by a sequence of at most $6\,g(S) + 8v$ edge swaps in the face compression case, and at most $6g(S)+8v+40$ edge swaps in the edge compression case.
\end{itemize}
\end{lemma}

\begin{proof}
Consider the arc $\alpha = D\cap S$ of Definitions~\ref{Def:EdgeComprDiscHS} and~\ref{Def:FaceComprDiscHS}.  We may suppose that $\alpha$ does not lie entirely in the interior of a 0-handle of $S$, else \reflem{EffectEdgeCompIn0Handle} implies $\Gamma$ or $\Gamma'$ transfers unchanged, and the bound is trivial. In the face compression case, this implies that $v \geq 2$.

Suppose first that the isotopy is not across an edge compression disc incident to a tubed piece of $S$. Then \reflem{EffectEdgeFaceComp} applies.
Let $H_2$ and $H_2'$ be the 2-handles of $S$ at the endpoints of $\alpha$. If $\alpha$ lies on a 1-handle, i.e.\ in the face compression disc case, let $H_1$ denote that 1-handle. If not, $\alpha$ runs through the interior of a 0-handle; let $H_0$ denote this 0-handle.

The effect of the isotopy along $D$ is given in \reflem{EffectEdgeFaceComp}; see also Figures~\ref{Fig:CanonHS-EdgeIsotopy} and~\ref{Fig:CanonHSDiscIsotopy}. The 2-handles $H_2$ and $H_2'$ are removed, along with $H_1$ in the face compression case, and handles incident to $H_2$ and $H_2'$ are combined. Finally, complementary regions that are not discs are replaced by discs. In the edge compression case, this produces (at most) $v$ 1-handles in $S'$, which we denote by $J_1^1, \dots, J_1^v$, and (at most) $(v-1)$ 0-handles between them, which we denote by $J_0^1, \dots, J_0^{v-1}$. In the face compression case, we only join $(v-1)$ 1-handles, since $H_1$ was removed, and so we produce (at most) $(v-1)$ 1-handles, denoted $J_1^1, \dots, J_1^{v-1}$ and (at most) $(v-2)$ 0-handles, denoted $J_0^1, \dots, J_0^{v-2}$. 

Let $\mathcal{C}$ and $\mathcal{C}'$ be the cell structures associated to the handle structures on $S$ and $S'$. We form a common refinement $\mathcal{C}''$, as follows. Starting from $\mathcal{C}$, overlay the cells of $\mathcal{C}'$ to form $\mathcal{C}''$. In the face compression case, this divides the 1-handle $H_1$ into $(2v-3)$ 2-cells by introducing 1-cells that run along it, and divides the 2-handles $H_2$ and $H'_2$ similarly, each becoming $(2v-3)$ 2-cells. See \reffig{DiscIsotopyCellStructure}. The edge compression case is similar, with the only difference being that additional 1-cells at the top and bottom are added. 

\begin{figure}
  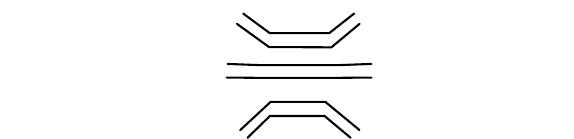
  \caption{$\mathcal{C}''$ is a common refinement of $\mathcal{C}$ and $\mathcal{C'}$.}
  \label{Fig:DiscIsotopyCellStructure}
\end{figure}

Suppose that we start with a spine $\Gamma$ that is cellular in $\mathcal{C}$. This is therefore cellular in $\mathcal{C}''$. We will move the spine through $\mathcal{C}''$ to a spine that is cellular in $\mathcal{C}'$. To do so, we move the spine to the boundary of any 1-handle $J_1^j$ and 0-handle $J_0^k$. These determine cellular discs in $\mathcal{C}''$. The length of $\bdy J_1^j$ in $\mathcal{C}'$ is four.
Because we divide the 1-handle to form the refinement, the length of $\bdy J_1^j$ in $\mathcal{C}''$ is 12. (We are assuming here that $J_1^j \not= J_1^k$ for $j \not= k$. In other words, we are assuming that no 1-handle $J_1^j$ winds several times through the disc region shown in Figure~\ref{Fig:DiscIsotopyCellStructure}. If some $J_1^j$ does enter this disc region several times, then our estimates for the total length of the discs $J_1^1, \dots, J_1^v, J_0^1, \dots, J_0^{v-1}$ will still hold.) The length of $\bdy J_0^k$ in $\mathcal{C}'$ is at most 24: 12 intersections with 1-handles by \reflem{FinitenessForNearlyNormal} and 12 intersections with 2-handles between 1-handles. Therefore the length of $\bdy J_0^k$ in $\mathcal{C}''$ is at most 32. Thus the total length of all boundaries is at most:
\begin{align*}
  12(v-1) + 32(v-2) = 44v - 76 & \qquad \mbox{ face compression case} \\
  12v + 32(v-1) = 44v - 32 & \qquad \mbox{ edge compression case}
\end{align*}

Then \refcor{SlidingOffDiscs} implies that $\Gamma$ can be transferred to $\Gamma'$ by a sequence of edge swaps, with total number bounded by:
\begin{align*}
  6(2v-3)g(S) + (88v-152) & \qquad \mbox{ face compression case} \\
  6(2v-1)g(S) + (88v-64) & \qquad \mbox{ edge compression case}
\end{align*}

Suppose now that we have a spine $\Gamma'$ that is cellular in $\mathcal{C}'$. It is then cellular with respect to $\mathcal{C}''$. We move the spine off the interior of the discs $H_2$ and $H_2'$, and in the face compression case off the  interior of the disc $H_1$. In the edge compression case, we move it off the  interior of the disc $H_0$.

In the case of a face compression disc, in $\mathcal{C}''$, the union of the discs $H_1, H_2, H_2'$ has boundary of length $2((v-1)+(v-2))+6 = 4v$.

In the case of an edge compression disc, we will consider the disc that is the union of discs $H_2, H_2'$, and $H_0$. The portion of the boundary of this disc that meets $\bdy H_2$ has length $v+(v-1)$, and similarly for the portion of the boundary that meets $\bdy H_2'$. As for $H_0$, if it is a nearly normal piece in $S$, then its boundary has length at most $24$ in $\mathcal{C}$ by \reflem{FinitenessForNearlyNormal}. Note that two of these $24$ curves come from intersections with $H_2$ and $H_2'$. Thus the length of $\bdy H_0$ away from $H_2$ and $H_2'$ in $\mathcal{C}''$ is at most $22$ in this case, and so the total length of the boundary in $\mathcal{C}''$ is at most $2(2v-1)+22 = 4v+20$.

If $H_0$ is not a nearly normal piece, then it comes from an octagon or tubed piece in $S$. An octagon is a disc whose boundary has length $16$ in $\mathcal{C}$. Again two of these curves come from intersections with $H_2$ and $H_2'$, so the length of the portion of the boundary of $H_0$ that is disjoint from $\bdy H_2 \cup \bdy H_2'$ is $14$. It follows that the total length of the boundary of $H_0\cup H_2\cup H_2'$ in $\mathcal{C}''$ is at most $2(2v-1)+14$ in this case.

In all cases other than the tubed case, \reflem{SlidingOffADisc} implies that there is a sequence of edge swaps in $\mathcal{C}''$ taking $\Gamma'$ to a spine $\Gamma$, which is cellular in $\mathcal{C}$, and that sequence has length at most:
\begin{align*}
  6\,g(S) + 8v &  \qquad \mbox{ face compression case} \\
  6\,g(S) + 8v+40& \qquad \mbox{ edge compression case}.
\end{align*}

Finally suppose $H_0$ came from a tubed piece. Lemma~\ref{Lem:EffectEdgeCompTube} describes the way that the canonical handle structure on $S'$ is obtained from that of $S$.  As above, $\mathcal{C}$ and $\mathcal{C}'$ are the cell structures associated to the handle structures on $S$ and $S'$. We define their common refinement $\mathcal{C}''$ again by superimposing them. Let $E$ be a disc in $S$ whose boundary $\partial E$ forms a core curve of the tube. Then $E$ contains either $H_2$ or $H_2'$, say $H_2'$. Then $H_2'$ and the other components of intersection between $E$ and the 1-handles and 2-handles of $\calH$ will have been deleted in $S$ since they lie in a disc of $F^+$ in \refdef{CanonicalHandleStructure}. However, to form the handle structure of $S'$, we reinstate these components of intersection between $E$ and the 1-handles and 2-handles, and then apply steps (1) to (5) of \reflem{EffectEdgeFaceComp}. Hence, $\calC''$ will look very similar to that in \reffig{DiscIsotopyCellStructure}, but with some of the 1-cells in the boundary of $H_2'$ removed.
Thus, the 2-handle $H_2$ in $S$ at one endpoint of $\alpha$ is divided into (at most) $(2v-1)$ 2-cells. At most $v$ 1-handles of $S'$ and at most $(v-1)$ 0-handles of $S'$ run through $H_2$. Denote these by $J_1^1, \dots, J_1^v$ and $J_0^1, \dots, J_0^{v-1}$. 

Suppose first that we have a spine that is cellular in $\mathcal{C}'$. Then it is also cellular in $\mathcal{C}''$. 
Let $W$ be the union of $H_2$ and the 
disc from $F^+$ described above. Its boundary has length in $\calC''$ at most $2v+6$. So we can move the spine in $\mathcal{C}''$ off the interior of $W$ using at most $6\,g(S) + 4v+12$ edge swaps, by Lemma~\ref{Lem:SlidingOffADisc}. It is then cellular in $\mathcal{C}$.

A spine that is cellular in $\mathcal{C}$ is cellular in $\mathcal{C}''$. We can then move it off the interior of $J_1^1, \dots, J_1^v$ and $J_0^1, \dots, J_0^{v-1}$ using Corollary~\ref{Cor:SlidingOffDiscs}. The length of $\partial J_1^j$ in $\mathcal{C}''$ is $8$. Then length of $\partial J_0^j$ in $\mathcal{C}''$ is at most $28$.
So, the total length of their boundaries is at most
\[ 8v + 28(v-1) = 36v - 28. \]
The number of edge swaps required to take the spine off these handles is therefore at most
\[ 6(2v-1)g(S) + (72v-56) .\qedhere \]
\end{proof}

We want to extend \reflem{MovingSpineIsotopyAcrossD} to generalised isotopies across edge or face compression discs. To do so, we need a notion of the length of the vertical boundary of an $I$-bundle. 

\begin{definition}[Length of vertical boundary] \label{Def:LengthVerticalBoundary}
Let $\mathcal{B}$ be the parallelity bundle for a compact orientable 3-manifold $M$. Then its vertical boundary $\partial_v{\mathcal{B}}$ is a union of annuli, which inherit a cell structure, as follows. Each 2-cell is a component of intersection between a parallelity handle and a non-parallelity handle. Thus, each such 2-cell is a union of fibres in $\mathcal{B}$. Define the \emph{length} of $\partial_v\mathcal{B}$ to be the number of such 2-cells. Similarly, if $\calB$ is a maximal generalised parallelity bundle, then $\partial_v \calB$ is a union of vertical boundary components of the parallelity bundle, and so its length is also defined.
\end{definition}

\begin{lemma}[Edge swap bound, isotopy across $I$-bundle over disc]\label{Lem:MovingSpineGenIsotopyAcrossB}
Let $S'$ be obtained from $S$ by a generalised isotopy move across an $I$-bundle over a disc as in \refdef{GenIsotopyMoveAcrossB}, where the disc has boundary of length $\ell$. Let $\mathcal{C}$ and $\mathcal{C}'$ be the cell structures associated with the canonical handle structures of $S$ and $S'$. Let $\Gamma$ be a spine for $S$ that is cellular in $\mathcal{C}$. Then there is a spine for $S'$ that is cellular in $\mathcal{C}'$ and that is obtained from $\Gamma$ by performing at most $6(2\ell+1) g(S)+92 \ell-64$ edge swaps. Similarly, if $\Gamma'$ is a cellular spine for $S'$, then there is a cellular spine for $S$ obtained by performing at most $6g(S)+8 \ell+40$ edge swaps. 
\end{lemma}

\begin{proof}
We may assume that $\alpha$ does not lie wholly in the interior of a 0-handle of $S$, as in this case, the isotopy from $S$ to $S'$ takes canonical handle structure to canonical handle structure. Let us also assume, for the moment, that if $S$ is almost normal, then the edge compression disc does not run over a tubed piece. 

Let $\mathcal{C}_-$ be the cell structure obtained from $\mathcal{C}$ by replacing each component of $\partial_h B$ by a 2-cell. Let $\mathcal{C}''$ be the cell structure that arises by superimposing $\mathcal{C}_-$ and $\mathcal{C}'$, as in the proof of \reflem{MovingSpineIsotopyAcrossD}. 

Consider first the spine $\Gamma$ that is cellular in $\mathcal{C}$. Then using \refcor{SlidingOffDiscs}, we can move $\Gamma$ off the interior of $\partial_h B$ using at most $12 g(S) + 4\ell$ edge swaps. It is then cellular in $\mathcal{C}_-$ and hence cellular in $\mathcal{C}''$. As in the proof of \reflem{MovingSpineIsotopyAcrossD}, there is then a sequence of at most 
$6(2\ell-1) g(S) + 88\ell - 64$
edge swaps taking to a spine that is cellular in $\mathcal{C}'$.

Consider now the spine $\Gamma'$ that is cellular in $\mathcal{C}'$. As in the proof of \reflem{MovingSpineIsotopyAcrossD}, there is a sequence of at most 
$6 g(S) + 8\ell+40$ edge swaps taking to a spine that is cellular in $\mathcal{C}_-$. It is then cellular in $\mathcal{C}$.

Finally, suppose that $S$ is almost normal and that the edge compression disc runs over a tubed piece.
By \reflem{EffectDiscIsotopy}, in this case we reinstate 1-handles and 2-handles of $\calH$, replace each component of $\partial_h B$ by a 2-cell, and apply steps (1) to (5) of \reflem{EffectDiscIsotopy} to obtain the canonical handle structure of $S'$.

Suppose that we are given a spine $\Gamma'$ that is cellular in $\mathcal{C}'$. Then the argument of \reflem{MovingSpineIsotopyAcrossD} applies. Specifically, each boundary component of the tube bounds a disc in $S$; let $W$ be the union of $\partial_h B$ and the larger of these two discs.  Its boundary in $\calC''$ has length at most $2\ell+6$. Then we make the spine disjoint from the interior of $W$ using at most $6g(S) + 4\ell + 12$ edge swaps. It is then cellular in $\mathcal{C}$.

To convert a spine $\Gamma$ in $S$ to a cellular spine in $S'$ requires at most $6(2 \ell - 1) g(S) + (72\ell - 56)$ edge swaps, as in the proof of \reflem{MovingSpineIsotopyAcrossD}.
\end{proof}

We can also deal with parallelisation isotopies at this stage.

\begin{lemma}[Edge swap bound, parallelisation isotopy]\label{Lem:EdgeSwapBoundParallelisation}
Let $S$ be an almost normal surface, and let $S'$ be a normal surface that is obtained from $S$ by a parallelisation isotopy. This moves $S$ across $D^2 \times [0,1]$, where $\partial D^2 \times [0,1]$ is a vertical boundary component of the parallelity bundle for $M \cut (S \cup S')$. Give $S$ and $S'$ their canonical handle structures. Consider a spine that is cellular in one of these surfaces. Then it may be converted to a spine that is cellular in the other surface using at most $6g(S) + 2 \ell$ edge swaps, where $\ell$ is the length of $\partial D^2 \times [0,1]$.
\end{lemma}

\begin{proof}
The canonical handle structures on $S$ and $S'$ agree away from $D^2 \times \{ 0,1 \}$. Thus, we make the spine disjoint from the interior of one of these discs using \reflem{SlidingOffADisc}. It is then cellular in the handle structure on the other surface.
\end{proof}

To apply Lemmas~\ref{Lem:MovingSpineGenIsotopyAcrossB} and~\ref{Lem:EdgeSwapBoundParallelisation} we need a bound on the length of vertical boundary components of parallelity bundles. This is obtained by the following lemmas. 

\begin{lemma}[Length bound] \label{Lem:BoundOnLengthVerticalBoundary}
Let $\mathcal{H}$ be a pre-tetrahedral handle structure for a compact orientable 3-manifold $M$. Let $\mathcal{B}$ be a maximal generalised parallelity bundle, or a union of components of the 
parallelity bundle.
Then the length of $\partial_v \mathcal{B}$ is at most $56 \Delta(\mathcal{H})$.
\end{lemma}

\begin{proof}
In both a maximal generalised parallelity bundle, and simply a parallelity bundle, each 2-cell in $\partial_v \calB$ is a component of intersection between a parallelity handle and a non-parallelity handle. The non-parallelity handle is a 0-handle or a 1-handle. As one travels around $\partial_v \calB$, one meets these 0-handles and 1-handles alternately. The number of 2-cells in $\partial_v \calB$ lying in the non-parallelity 1-handles is at most the number lying in the non-parallelity 0-handles. So it suffices to bound the latter quantity. A tetrahedral 0-handle has at most six components of intersection with the parallelity handles. By Definition \ref{Def:Complexity}, its complexity is at least 1/2, so it meets a number of parallelity handles equal to at most 12 times its complexity. A semi-tetrahedral 0-handle has at most seven components of intersection with the parallelity handles (five 2-handles and two 1-handles). Again by Definition \ref{Def:Complexity}, its complexity is at least 1/4, so it meets a number of parallelity handles equal to at most 28 times its complexity. A product non-parallelity 0-handle is disjoint from the parallelity handles. So, the number of 2-cells in $\partial_v \mathcal{B}$ is at most $56 \Delta(\calH)$.
\end{proof}

\begin{lemma}[Length bound, almost normal case] \label{Lem:LengthBoundVerticalBoundaryAlmostNormal}
Let $\calH$ be a pre-tetrahedral handle structure for a compact orientable 3-manifold $M$. Let $S$ be an almost normal surface in $M$, and let $\calH'$ be the handle structure on $M \cut S$ described in Definition \ref{Def:HandleStructCutAlmostNormal}. Let $\mathcal{B}$ be a maximal generalised parallelity bundle for $M \cut S$, or a union of components of the parallelity bundle.
Then the length of $\partial_v \mathcal{B}$ is at most $88 \Delta(\mathcal{H})$.
\end{lemma}

\begin{proof}
As above, it suffices to bound the number of components of intersection between the non-parallelity 0-handles of $\calH'$ and the parallelity handles. When a 0-handle $H_0$ of $\calH$ is divided along normal discs into 0-handles of $\calH'$, the sum of the complexities of these 0-handles is equal to the complexity of $H_0$. The number of intersections between each such 0-handle and the parallelity handles is at most $28$ times its complexity, as argued in the proof of \reflem{BoundOnLengthVerticalBoundary}. 

Thus, it suffices to consider when a 0-handle $H_0$ of $\calH$ contains an almost normal piece, plus possibly some triangles and squares. If the almost normal piece is tubed, the number of intersections of the resulting 0-handles with the parallelity handles is the same as in the situation when the tube is compressed. So again the number of intersections is at most $28$ times the complexity of $H_0$.

So, we now consider the situation with $H_0$ contains an octagon plus possibly some triangles and squares. Then $H_0$ must be tetrahedral. One can check that the resulting 0-handles of $\calH'$ have at most $14$ intersections with the parallelity 2-handles and at most $8$ intersections with the parallelity 1-handles. Again, since the complexity of a tetrahedral 0-handle is at least 1/2, it follows that the number of intersections between these 0-handles and the parallelity handles is at most $44$ times the complexity of $H_0$.

We must then double this to get an upper bound on the length of $\partial_v \mathcal{B}$, because we must also take account of the intersections between non-parallelity 1-handles of $\calH'$ and the parallelity handles. But, as argued in the proof of \reflem{BoundOnLengthVerticalBoundary}, the number of these is at most the number of intersections between the non-parallelity 0-handles of $\calH'$ and the parallelity handles.
\end{proof}

\section{Edge swap bounds for annular simplification}\label{Sec:AnnularSimplification}

\textbf{Road map:} Using generalised isotopy moves, we have a sequence of surfaces interpolating between almost normal, nearly normal, and normal surfaces. For compression isotopies, tube compressions, parallelisation isotopies and generalised isotopy along an edge or face compression disc, we know how to transfer spines along those surfaces, and we can bound the number of edge swaps to do so. What we are missing is a similar bound for annular simplification. This case is somewhat more difficult than the others, and so we put it in its own section. 

The main goal of this section is to bound the number of edge swaps required to transfer a cellular spine from one surface to another under an annular simplification. Our bound on edge swaps will be in terms of the width of an annulus, defined as follows.

\begin{definition}[Width]\label{Def:Width}
Let $M'$ be a compact orientable irreducible 3-manifold with a handle structure $\mathcal{H}'$. Let $F$ be an incompressible subsurface of $\partial M'$ such that $\partial F$ is standard and $F$ is not a 2-sphere. Let $\mathcal{B}$ be a maximal generalised parallelity bundle for $(M', F)$, and let $B$ be an essential annular component of $\mathcal{B}$. Then the \emph{width} of $B$ is defined as follows.

Pick a component $A$ of $\partial_h B$ and let $P$ be its intersection with the parallelity bundle. By (6) in the definition of a generalised parallelity bundle, \refdef{GeneralisedParallelityBundle}, $A \cut P$ lies in a union of disjoint discs in the interior of $A$. 
Hence, one component $A_-$ of $P$ is obtained from $A$ by removing some discs from its interior. Define the width of $B$ to be length of the shortest cellular path in $A_-$ joining the two components of $\partial A$, where the length of a path is the number of edges traversed.

Note that width is independent of the choice of component $A$ of $\partial_h B$, since the parallelity bundle sets up a handle preserving homeomorphism from $A_-$ to the corresponding surface in the other component of $\partial_h B$.
\end{definition}

\subsection{Edge swaps in annuli}\label{Sec:EdgeSwapsAnnuli}

We consider again cellular spines on surfaces, as in Subsection~\ref{Sec:EdgeSwaps}. In that subsection, we discussed edge swaps within discs. We now need to build similar tools within annuli.

\begin{lemma}[Arranging a spine and an annulus]\label{Lem:SimplifySpineAnnulus}
Let $S$ be a closed orientable surface with a cell structure $\mathcal{C}$, and $\Gamma$ a spine for $S$ that is cellular with respect to $\mathcal{C}$. 
Let $A$ be an essential annulus embedded in $S$ that is a union of cells, such that $\bdy A$ has length $\ell$. Finally, let $\alpha$ be an essential properly embedded arc in $A$ that is cellular and has length $d$. Then there is a sequence of at most $6g(S)+2\ell+4d$ edge swaps taking $\Gamma$ to a cellular spine $\Gamma'$ such that $\Gamma' \cap {\rm int}(A) = {\rm int}(\alpha)$.
\end{lemma}

\begin{proof}
Let $D$ be the disc $A \cut \alpha$. We will modify $\Gamma$ so that it misses the interior of $D$. The proof is very similar to that of \reflem{SlidingOffADisc}, but there we required an embedded disc, and here $D$ is not embedded.

Let $e$ be an edge of $\Gamma$ with interior that lies entirely in $D$,
and let $x$ be a point in the interior of $e$. Removing $e$ from $\Gamma$ gives a graph $\Gamma_-$, the exterior of which is an annulus $A'$. The arc $e$ is essential in $A'$. Hence, there is an embedded core curve $C$ for $A'$ that intersects $e$ once at exactly the point $x$. The boundary $\partial A$ intersects the interior of $A'$ in a (possibly empty) collection of properly embedded arcs. By choosing $C$ suitably, we may assume that $C$ intersects each of these arcs at most once. If $C$ does intersect one of these arcs, $e'$ say, then we perform the edge swap that removes $e$ and inserts $e'$. On the other hand, if $C$ is disjoint from $\partial A$, then it lies entirely in $A$. It must then intersect $\alpha$ by \reflem{RemoveEdgeSpine}, and we can assume that it intersects each arc of $\alpha \cut \Gamma_-$ at most once. Let $e'$ be one such arc. If necessary, add sub-arcs of $\partial A$ to $e'$ so that it becomes properly embedded in $A'$. Perform an edge swap that removes $e$ and replaces it by this enlarged $e'$. By Lemma \ref{Lem:BoundOnVerticesAndEdges}, at most $6g(S)$ edges lie in the interior of $D$, so at most $6g(S)$ edge swaps take $\Gamma$ to a spine without edges in $\mathrm{int}(D)$.

After this procedure, $\Gamma$ intersects the interior of $D$ in a collection of embedded arcs and star-shaped graphs. These divide $D$ into discs, and each of these discs contains at least one 1-cell lying in $\partial D$. Hence, the number of arcs is less than twice the length $\ell + 2d$ of $\partial D$. Let $e$ be an arc of $\mathrm{int}(D) \cap \Gamma$, and let $x$ be a point in $\mathrm{int}(e) \cap \mathrm{int}(D)$. Again, there is a closed embedded curve $C$ intersecting $\Gamma$ exactly at the point $x$.
We may assume that $C$ intersects each arc of $\partial A \cut \Gamma$ at most once. If it does intersect some arc of $\partial A \cut \Gamma$, then perform the edge swap that adds this arc plus possibly a sub-arc of $\alpha \cut \Gamma$, forming a graph $\Gamma_+$, and then removes the edge of $\Gamma_+$ that contains $x$. On the other hand, if $C$ is disjoint from $\partial A$, then it intersects some arc of $\alpha \cut \Gamma$, and we may perform an edge swap that adds this arc, plus possibly some sub-arcs of $\partial A$ and removes the arc containing $x$.
We do such an edge swap for each arc intersecting the interior of $D$, and there are at most $2\ell + 4d$ of these.

After at most $6g(S) + 2\ell + 4d$ edge swaps, we end with a spine $\Gamma'$ such that
$\Gamma' \cap {\rm int}(A) \subseteq {\rm int}(\alpha)$. In fact, 
$\Gamma' \cap {\rm int}(A) = {\rm int}(\alpha)$ since otherwise a core curve of $A$ lies in the disc $S \cut \Gamma'$, contradicting the assumption that $A$ is essential in $S$.
\end{proof}

\begin{lemma}[Replacing an annulus]\label{Lem:RemoveAnnulus}
Let $S$ be a closed orientable surface with a cell structure $\calC$, and let $\Gamma$ be a spine that is cellular with respect to $\calC$. Let $A$ be an essential annulus embedded in $S$ that is a union of cells, with $\ell$ the length of $\partial A$. Suppose $\calC'$ is obtained from $\calC$ by removing the interior of $A$ and inserting some other cell structure. Then the number of edge swaps required to take $\Gamma$ to a spine that is cellular with respect to $\calC'$ is at most
\[ 6g(S) + 2\ell + 4d + |\alpha \cap \alpha'| + 2, \]
where $\alpha$ is any essential properly embedded arc in $A$ that is cellular with respect to $\calC$, $d$ is the length of $\alpha$, and $\alpha'$ is an embedded arc in $A$ with the same endpoints as $\alpha$ that is cellular with respect to $\calC'$.
\end{lemma}

\begin{proof}
First, by \reflem{SimplifySpineAnnulus}, we may apply at most $6g(S)+2\ell+4d$ edge swaps to take $\Gamma$ to a spine $\Gamma_1$ that is cellular in $\mathcal{C}$ and such that $\Gamma_1 \cap {\rm int}(A) = {\rm int}(\alpha)$. Now, we are assuming that $\alpha'$ is an embedded arc in $A$, but it is not necessarily properly embedded.
So it might intersect $\Gamma_1$ and $\bdy A$ in its interior. But it contains a sub-arc $\alpha''$ that is embedded and
essential in $A$ such that $\alpha'' \cap \bdy A \cap \Gamma_1 = \partial \alpha''$. Also, let $\alpha'''$ be the arc obtained from $\alpha'$ by pushing its interior a little into the interior of $A$. By \reflem{EdgeSwapDehnTwist}, we may perform $|\alpha \cap \alpha'|$ edge swaps to $\Gamma_1$, that leaves $\Gamma_1 - A$ unchanged and that takes $\alpha$ to $\alpha'''$. Then, at most two further edge swaps adjust the endpoints of $\alpha'''$ so that they are equal to those of $\alpha''$. After this process, the spine is cellular with respect to $\mathcal{C}'$.
\end{proof}

\subsection{Essential annular simplifications}
Recall from \refdef{AnnularSimplification} that annular simplifications are either trivial or essential. Our next goal is to find bounds on edge swaps when an essential annular simplification is performed. To do so, we first set up some notation.

Let $M$ be a compact orientable 3-manifold with a pre-tetrahedral handle structure $\calH$. Let $S$ be a closed normal or almost normal surface with a given transverse orientation, and let $\calB$ be a maximal generalised parallelity bundle for the component of $M\cut S$ into which $S$ points.

Inductively, assume we have performed a (possibly empty) sequence of annular simplifications in the direction $S$ points, obtaining a surface $S'$ disjoint from the interior of $\calB$. 
Suppose that $S'$ admits a further essential annular simplification. Then the setup is as follows, illustrated in \reffig{AnnSimpAcrossSolidTorus}.

\begin{figure}
  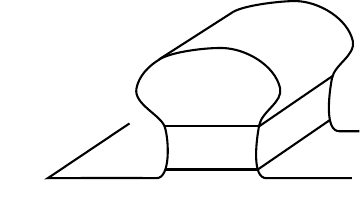
  \caption{Schematic picture for when $S'$ admits an essential annular simplification.}
  \label{Fig:AnnSimpAcrossSolidTorus}
\end{figure}

Recall from \refdef{AnnularSimplification} that we will replace an annulus $A \subset S'$ with an annulus $A'$ that lies on the vertical boundary of a component of $\calB$.
Because this is an essential annular simplification, there is a component $B$ of the generalised parallity bundle of the form $A_B\times I$, where $A_B$ is an annulus, and $(A_B\times I)\cap S' = A_B\times \bdy I$. The annuli $A_B\times\{0\}$ and $A_B\times\{1\}$ form sub-annuli of $A\subset S'$; see \reffig{AnnSimpAcrossSolidTorus}. Let $A''$ denote the annulus $A \cut (A_B\times \bdy I)$. 

The vertical boundary $\bdy A_B\times I$ is incompressible by assumption. It has two components, one of which is $A'$ and will replace $A$ after the annular simplification. Denote the other by $V$. Then $\bdy V$ and $\bdy A''$ agree.

Because the annular simplification is essential, $A\cup A'$ bounds a 3-manifold $P$ that forms a product region. This is a solid torus. The thickened annulus $A_B\times I$ forms a regular neighbourhood of $A'$ in that solid torus. The remainder $P \cut(A_B\times I)$ is a solid torus that we denote by $W$.

When performing annular simplifications, we will be isotoping annuli across solid tori. The following sequence of results leads to a bound on the number of edge swaps required under these moves.

\begin{theorem}[Winding number in triangulated solid torus]\label{Thm:SolidTori}
Let $\calT$ be a triangulation of a solid torus $M$. Suppose that some simple closed longitudinal curve $\lambda$ in $\partial M$ is simplicial. Let $C$ be a simplicial simple closed curve in $\partial M$ that intersects $\lambda$ once transversely. Then the winding number of $C$ in $M$ is at most $10 \Delta(\calT) + 3$.
\end{theorem}

\begin{proof} 
We can assume that $C$ runs over at most one edge in each 2-simplex in $\partial M$, as otherwise we may shorten $C$. Its length is therefore at most half the number of triangles in $\partial M$; hence it is at most $(3/2) \Delta(\calT)$. Similarly, $\lambda$ has length at most  $(3/2) \Delta(\calT)$.  Let $w$ be the winding number of $C$. If $w$ is odd, we replace $C$ by a simplicial, but possibly non-embedded curve, with winding number $w+1$, by Dehn twisting about $\lambda$. The resulting curve has length at most $3\Delta(\calT)$. Attach a 2-cell to $M$ along this curve. We can triangulate this 2-cell by using at most $3 \Delta(\calT)$ triangles. Then attach a 3-cell. The number of triangles in the boundary of this 3-cell is at most $9 \Delta(\calT)$. So we can form a triangulation of this 3-manifold using at most $10 \Delta(\calT)$ tetrahedra. This manifold is the lens space $L(p,1)$, where $p = w$ or $w+1$. It is a theorem of Jaco, Rubinstein and Tillmann \cite{JacoRubinsteinTillmann:LensSpaces} that, when $p$ is even, $\Delta(L(p,1)) = p-3$. Hence, $10 \Delta(\calT) \geq p -3 \geq w -3$.
\end{proof}

\begin{lemma}[Triangulation from handle structure, solid torus]\label{Lem:TriangulatingSolidTorus}
Let the setup be as in \reffig{AnnSimpAcrossSolidTorus}, as above. That is:
\begin{itemize}
\item $\calH$ is a pre-tetrahedral handle structure on a compact orientable 3-manifold $M$ containing a closed, transversely oriented surface $S$ that is normal or almost normal.
\item $\calB$ is a maximal generalised parallelity bundle for the component of $M\cut S$ into which $S$ points. 
\item $S'$ is a transversely oriented surface that is obtained from $S$ by a sequence of clean annular simplifications. It is disjoint from the interior of $\calB$, with an essential annular component $A_B\times I$ of $\calB$ meeting $S'$ in $A_B\times \bdy I$. 
\item $V$ is the vertical boundary component of $A_B\times I$ such that $\bdy V$ bounds an annulus $A''$ in $S'$, and $V\cup A''$ bounds a solid torus $W$ with interior disjoint from $A_B\times I$.
\item $S'$ has a cell structure, where each 2-cell is either a component of intersection between $S' \cap S$ and $\calH$, or a 2-cell in $\partial_v\calB$.
\end{itemize}
If $S$ is normal, define $\Delta$ to be the sum of the complexities of the 0-handles of $M\cut S$ that intersect $W$. If $S$ is almost normal, define $\Delta$ to be the sum of the complexities of the 0-handles of $M$ that intersect $W$.

Suppose that the only components of $\calB$ in $W$ are $I$-bundles over discs.
Then the solid torus $W$ admits a triangulation with at most $1320\Delta = 3t\Delta$ tetrahedra, where $t=440$.  Each 2-cell of $S' \cap W$ is a union of simplices in this triangulation, as is each 2-cell of $V$. There are at most 
$440\Delta = t\Delta$ 1-cells in the cell structure on $\bdy W \cut \partial_h \calB$.
\end{lemma}

\begin{proof}
Suppose first that $S$ is normal. Then $M \cut S$ inherits a handle structure $\calH'$ from $\calH$. Consider any 0-handle $H$ of $\calH'$ that intersects $W$ and is not a parallelity handle. Its boundary has a cell structure, obtained by viewing each of the associated handles on $\bdy H$ as 2-cells; the possible cell structures are shown in \reffig{PreTetrahedral}.

Triangulate each 2-cell by adding a vertex to its interior and coning off. Then triangulate the 0-handle by adding a vertex to its interior and coning off.
Colour the resulting tetrahedra red.
In the case of a tetrahedral 0-handle, we use $72$ tetrahedra. In the case of a semi-tetrahedral 0-handle, we use $60$ tetrahedra. Note that a parallelity handle of length three or four does not occur, as it is a parallelity handle. Thus it only remains to consider the case of a product 0-handle (of length three), which uses $36$ tetrahedra. Considering the associated complexity of these types of handles, as in \refdef{Complexity}, we see that the number of red tetrahedra is at most $288$ times the complexity of the 0-handle. 

When two 0-handles in $W$ are incident to the same 1-handle, glue them together along the relevant triangles in their boundary. This process does not require any more tetrahedra in our triangulation. Each component of $\calB$ in $W$ is, by assumption, an $I$-bundle over a disc. Its vertical boundary has already been triangulated by triangles coloured red. We triangulate the horizontal boundary by adding a vertex to its interior and then coning. We then triangulate the whole $I$-bundle by adding a vertex to its interior and coning.
For those tetrahedra meeting a red triangle, adjacent to the vertical boundary, colour the tetrahedron orange. Colour the rest of the tetrahedra, adjacent to the horizontal boundary, yellow. Note each red tetrahedron is adjacent to at most one orange, and each orange tetrahedron is adjacent to at most one yellow. 

For each 2-handle of $\calH'$ that is not a parallelity handle, first triangulate its horizontal boundary by adding a vertex to its interior. Then triangulate the whole 2-handle by adding a vertex to its interior and coning off. Again colour tetrahedra adjacent to the vertical boundary orange, and the others yellow. 

The boundary of each 3-handle has then been triangulated, and so we triangulate the 3-handle by adding a vertex to its interior and coning.
For any resulting tetrahedra that are adjacent to a red tetrahedron, colour the tetrahedron orange. The remaining tetrahedra will be adjacent to a yellow tetrahedron on the boundary of the 3-handle, but also adjacent to an orange tetrahedron in the interior, so colour these yellow. 

Now, there are at most $288\Delta$ red tetrahedra. Each orange tetrahedron is adjacent to a red tetrahedron, but each red tetrahedron is adjacent to at most one orange. Hence there are at most $288\Delta$ orange tetrahedra. Similarly, each yellow tetrahedron is adjacent to an orange tetrahedron, and each orange is adjacent to at most one yellow, so there are at most $288\Delta$ yellow tetrahedra. Thus in total there are $3*288\Delta = 864\Delta$ tetrahedra in the triangulation of the solid torus. 

When $S$ is almost normal, this construction needs to be adjusted a little, because $S$ contains an almost normal piece. Let $H_0$ denote the handle of $\calH$ containing this almost normal piece. 
Note $S$ still cuts handles of $\calH$ into a pre-tetrahedral handle structure, except for two components of $H_0\cut S$ adjacent to the almost normal piece. Then away from $H_0$, the complexity of $M$ agrees with the complexity of $M\cut S$, which is well-defined. The complexity of $H_0\cut S$ is not defined for exactly two components. However, note that if these two components are glued along the almost normal piece of $S$, the result is a pre-tetrahedral handle $H$ with well-defined complexity. Note also that because $S$ is separating, at most one of these two pieces meets $W$. Hence for our bound, we may use the bounds above in the normal case away from $H$, and we will bound the number of tetrahedra in one of the two components of $H_0\cut S$ in terms of the complexity of $H$ to complete the bound.

When the almost normal piece is an octagon, then we triangulate the octagon by adding a vertex to interior and coning off. We then triangulate the 3-ball lying to one side of it, as described above, with red tetrahedra. This uses $84$ tetrahedra:
16 adjacent to the octagon, $7\times 4$ adjacent to 1-handles, $6\times 2$ adjacent to 2-handles, and $(6\times 2)+(4\times 4)$ adjacent to remaining 0-handles. 
The octagon lies in a tetrahedral 0-handle.
The union $H$ of the two components of $H_0\cut S$ adjacent to the octagon must be tetrahedral, and $H$ has complexity at least $1/2$. Thus the number of red tetrahedra is at most $168$ times the complexity of $H$.

When the almost normal piece is a tube, there are several cases. The handle $H$, which is the union of the two components of $H_0\cut S$ adjacent to the tube, could be tetrahedral, semi-tetrahedral, or even a parallelity or product handle, since $S$ may run past the tube in triangles or squares on either side. The solid torus $W$ might lie on the 3-ball side of the tube or the solid torus side of the tube. Finally, the tube might run between two squares, between a triangle and a square, or between two triangles. All these cases need to be considered.

In the case $W$ lies on the 3-ball side of the tube, we triangulate the 3-ball by adding an essential arc to the tube, which cuts it into a disc with either 18 sides (in the case the tube joins two squares), 16 sides (for a triangle and a square), or 14 sides (when the tube joins two triangles). We then cone off the disc and cone off all other 2-cells, then cone the 3-ball, and colour these tetrahedra red. The different cases outlined above give different numbers of tetrahedra. The highest bound on tetrahedra per complexity occurs in the case that the tube connects a square to a square and $H$ is semi-tetrahedral, with complexity at least $1/4$. In this case there are $110$ tetrahedra produced from the coning procedure, so the number of tetrahedra is bounded by at most $440$ times the complexity of $H$.

In the case $W$ lies on the solid torus side of the tube, we triangulate the solid torus by attaching an edge compression disc, and then triangulating the resulting 3-ball. The edge compression disc cuts the tube into a disc with 18, 16, or 14 sides again, depending on whether the tube runs from a square to a square, a square to a triangle, or a triangle to a triangle, respectively. The edge compression disc itself appears twice on the boundary of the resulting 3-ball; each time it is coned into two triangles, hence coned to a total of four tetrahedra. Again considering the other 2-cells in all the cases outlined above, the highest bound on tetrahedra per complexity occurs either in the case that the tube runs between a triangle and a square, or between two triangles, and $H$ is semitetrahedral. In that case, there are $66$ red tetrahedra, so the number of tetrahedra is bounded by at most $264$ times the complexity of $H$. 

As in the normal case, obtaining a triangulation of $W$ can be achieved by triangulating components of $\calB$ that are $I$-bundles over discs, remaining 2-handles, and remaining 3-handles, and this adds at most three times the number of tetrahedra. 
So in all cases, the result is a triangulation $\calT$ of the solid torus $W$ with at most $1320 \Delta = 3t\Delta$ tetrahedra, where $t = 440$ is the maximum number of triangles per complexity on the boundary of a 0-handle. 

Finally, the number of 1-cells in $\bdy W \cut \bdy_h\calB$ is at most the number of triangles on the boundary of the 0-handles in $W \cut \calB$, which is at most $t\Delta$ for $t=440$, as above.
\end{proof}

\begin{proposition}[Edge swap bound, essential annular simplification]\label{Prop:SpinesAcrossSolidTorus}
Let $\calH$, $M$, $S$, $\calB$ be as in \reflem{TriangulatingSolidTorus}, and let $S'$, $A_B\times I$, $V$, $A''$ and the solid torus $W$ be as in that lemma; see \reffig{AnnSimpAcrossSolidTorus}. 

Suppose that the only components of $\calB$ in $W$ are $I$-bundles over discs. Let $S''$ be the result of performing an annular simplification to $S'$, by removing $A = A'' \cup (A_B \times \bdy I)$ from $S'$ and replacing it by $A' = (\bdy A_B \times I) - V$. Let $w$ be the width of $A_B \times I$. Let $\ell$ be the length of $(\bdy A_B \times \bdy I) - \bdy V$.

In the case where $S$ is normal, let $\Delta$ be the sum of the complexities of the 0-handles of $M \cut S$ that intersect $W$. In the case where $S$ is almost normal, let $\Delta$ be the sum of the complexities of the 0-handles of $M$ that intersect $W$.
Then the following hold.
\begin{enumerate}
\item Any cellular spine in $S'$ may be converted to a cellular spine in $S''$ using at most
  \[ 6g(S) + 34t \Delta + 8w +\ell (6g(S) + 50) + 5 \mbox{ edge swaps,} \] 
\item Any cellular spine in $S''$ may be converted to a cellular spine in $S'$ using at most
  \[ 6g(S) +  2\ell + 30t \Delta+ 9 \mbox{ edge swaps, }\]
\end{enumerate}
where in both cases, $t=440$ is the constant of \reflem{TriangulatingSolidTorus}.
\end{proposition}

\begin{proof}
Use \reflem{TriangulatingSolidTorus} to obtain a triangulation for $W$ with at most $3t\Delta$ tetrahedra, where $t$ is the constant of that lemma, and such that the corresponding cell structure on $\bdy W \cut \partial_h \calB$ contains at most $t\Delta$ 1-cells.

We now wish to use \reflem{RemoveAnnulus}. This deals with situation where a cell structure on a surface is modified by changing it within a subsurface that is homeomorphic to an annulus. In our situation, one surface is $S'$, containing the annulus $A =A'' \cup (A_B \times \partial I)$. We remove the interior of $A$, and replace it by the cell structure on $A' = (\partial A_B \times I) - V$. This gives a cell structure $\mathcal{C}''$ on the surface $S''$. This is not quite the canonical cell structure on $S''$, rather it is a refinement of it.

In order to apply \reflem{RemoveAnnulus}, we need to specify essential arcs $\alpha$ and $\alpha'$ that are cellular in the canonical cell structure on $S'$ and $\mathcal{C}''$, respectively. We do this as follows. The annular region $A_B \times I$ has width $w$. So there is some essential arc $\beta$ properly embedded in $A_B$ such that $\beta \times I$ is vertical in the parallelity handles in $A_B \times I$ and where each component of $\beta \times \bdy I$ has length $w$. Denote the components of $\beta\times \bdy I$ by $\beta_0$ and $\beta_1$; see \reffig{SpineAcrossSolidTorus}.

\begin{figure}
  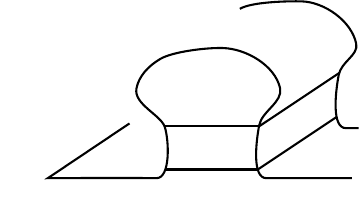
  \caption{Curves $\gamma$, $\beta_0$, $\beta_1$ in the proof of \refprop{SpinesAcrossSolidTorus}.}
  \label{Fig:SpineAcrossSolidTorus}
\end{figure}

Join the two points $(\beta \times \bdy I) \cap A''$ by an embedded arc $\gamma$ in $A''$. We may ensure that $\gamma$ is cellular in the cell structure on $S'$ and simplicial in the triangulation of $W$. We may also ensure that it avoids the interior of $\bdy_h \mathcal{B}$ and so its length in the cell structure is at most the number of 1-cells in $\bdy W \cut \bdy_h \mathcal{B}$, which is at most
$t \Delta$ by \reflem{TriangulatingSolidTorus}. Let $\alpha$ be $\beta_0\cup \beta_1\cup \gamma$. This is cellular in $S'$ and has length at most $t\Delta + 2w$. Let $\alpha'$ be the component of $\bdy \beta \times I \subset A'$ with the same endpoints as $\alpha$. Note since $\alpha'$ is a fibre in the vertical boundary, it has length $1$. 

We need to find an upper bound on the intersection number between $\alpha'$ and the image of $\alpha$ under the isotopy from $S'$ to $S''$.
Let $C$ be the simple closed curve in $\bdy W$ obtained from $\gamma$ by attaching the component of $\bdy \beta \times I \subset V$ to which it is incident.
This has winding number at most $10\times 3t\Delta + 3=30t\Delta+3$ in $W$, by \refthm{SolidTori} and \reflem{TriangulatingSolidTorus}. 
Hence, the simple closed curve $\alpha \cup \alpha'$ has winding number at most $30t\Delta+3$ in the solid torus $W \cup (A_B \times I)$.
So, after the isotopy taking $S'$ to $S''$, the image of $\alpha$ and $\alpha'$ have intersection number at most $30t \Delta + 3$.

By \reflem{RemoveAnnulus}, a cellular spine in $S'$ may be converted to a cellular spine in $\mathcal{C}''$ using at most
\[
6g(S) + 2\ell + 4\cdot(t\Delta + 2w)  + (30t \Delta+3) + 2
\]
edge swaps.

It may then be converted to a cellular spine in the canonical handle structure on $S''$ using at most a further $6g(S)\ell + 48 \ell$ edge swaps, as follows. Recall from \reflem{AnnularSimpChangeHS} that to get from $\mathcal{C}''$ to the canonical handle structure on $S''$, we remove the edges that lie in the boundary of the annulus $A' = (\partial A_B \times I) - V$. This has the effect of combining each 0-handle in $A'$ with the pair of 0-handles adjacent to it, and doing the same thing to the 1-handles. If this process creates any regions that are not discs, then they are replaced by discs. Therefore, let $D_1, \dots, D_n$ be the 0-handles and 1-handles of $S''$ that intersect the interior of $A'$. The number of these is at most the length $\ell$ of the annulus. The length of the boundary of each $D_i$ is at most $24$ by \reflem{FinitenessForNearlyNormal}. Hence, the number of edges swaps needed to move the spine off the interior of these discs is at most $6g(S) \ell + 48 \ell$, by \refcor{SlidingOffDiscs}.

Thus for part~(1), we have used at most
\[ 6g(S) + 34t \Delta + 8w + \ell(6g(S)+50) + 5 \]
edge swaps. 

For part~(2), a cellular spine in the canonical handle structure $S''$ is already cellular in $\mathcal{C}''$.
By \reflem{RemoveAnnulus}, it may be converted to a cellular spine in the canonical handle structure on $S'$ using at most
\[ 6g(S) +  2\ell + 4 + 30t \Delta+ 5 \]
edge swaps.
Note the bounds are different here, because when we apply \reflem{RemoveAnnulus} in this situation, what is relevant is the length of $\alpha'$, which is $1$, rather than the length of $\alpha$, which is at most $t \Delta(W) + 2w$.
\end{proof}

Note that there is an asymmetry between the statements (1) and (2) in the above theorem. In particular, when a spine in $S''$ is transferred to one in $S'$, the width of the annular bundle is not relevant. This will be important and useful later.

At this point, we can also deal with trivial annular simplifications.

\begin{lemma}[Edge swap bound, trivial annular simplification]\label{Lem:EdgeSwapTrivialAnnSimp}
Let $\calH$, $M$, $S$, $\calB$ and $S'$ be as in Lemma \ref{Lem:TriangulatingSolidTorus}. Let $S''$ be obtained from $S'$ by performing a trivial annular simplification. This has the effect of removing an annulus $A'$ in $S'$ and replacing it by an annulus $A''$ that is a vertical boundary component of $\calB$. Since the annular simplification is trivial, both boundary components of $A'$ bound discs in $S'$, which are nested. Let $D$ be the larger of these discs and let $\ell$ be the length of its boundary. Give $S'$ and $S''$ the cell structures associated with their canonical handle structures. Then the following hold.
\begin{enumerate}
\item Any cellular spine in $S'$ may be converted to a cellular spine in $S''$ using at most
  \[ 6g(S)(\ell+1) + 50\ell \mbox{ edge swaps.} \]
\item Any cellular spine in $S''$ may be converted to a cellular spine in $S'$ using at most $6g(S) + 2 \ell$ edge swaps.
\end{enumerate}
\end{lemma}

\begin{proof}
The canonical handle structure of $S''$ is obtained from that of $S'$ as in \reflem{AnnularSimpChangeHS}. Following this process, the associated cell structure on $S''$ is obtained by the following two steps.
\begin{enumerate}
\item Remove the handles in the interior of $A'$ and replace them with the cell structure on $A''$, leaving the 1-cells on the boundary $\bdy A'=\bdy A''$. That is, $\bdy A'$ has adjacent 0-handles and 1-handles in an alternating fashion, which combine in pairs in the $I$-bundle structure on $A''$. Although in step~(2) of \reflem{AnnularSimpChangeHS} the 0-handles and 1-handles are combined to extend beyond $A''$, we initially leave their boundaries as 1-cells on $\bdy A''$.  This gives a cell structure $\mathcal{C}''$ on $S''$, consisting of alternating 0-handles and 1-handles, cut off by 1-cells in $\bdy A''$.
\item Now remove the 1-cells in $\partial A''$, completing step~(2) of \reflem{AnnularSimpChangeHS}. Any disc regions that are created become 2-cells. However, if any regions are not discs, then these are combined with the disc subsets of $S''$ that they surround to form 0-handles, as in step~(3) of \reflem{AnnularSimpChangeHS}.
\end{enumerate}

Let $\Gamma'$ be a cellular spine for $S'$. Using Lemma~\ref{Lem:SlidingOffADisc}, we can move $\Gamma'$ off $D$ using at most $6g(S) + 2 \ell$ edge swaps. It is then cellular in $\mathcal{C}''$. To make it cellular with respect to the canonical handle structure of $S''$, we need to move it off the 1-cells in $\partial A''$. Let $D_1, \dots, D_n$ be the 0-handles and 1-handles of $S''$ that intersect the interior of $A''$. The number of these is at most the length $\ell$ of $\partial D$. As in the proof of \reflem{MovingSpineIsotopyAcrossD}, the length of the boundary of each $D_i$ is at most $24$, as follows. For $D_i$ coming from 1-handles, the length is always 4. For $D_i$ coming from 0-handles, \reflem{FinitenessForNearlyNormal} implies there are at most 12 intersections of its boundary with 1-handles, and there will be at most 12 intersections with 2-handles between 1-handles, giving length at most 24.
So, the number of edge swaps needed to move the spine off the interior of these discs is at most $6g(S)\ell + 48 \ell$ by Corollary \ref{Cor:SlidingOffDiscs}. It is then cellular in $S''$.

Now let $\Gamma''$ be a cellular spine for $S''$. We refine the cell structure of $S''$ by making $\partial D$ cellular. This keeps $\Gamma''$ cellular. We then move $\Gamma''$ off $D$ using at most $6g(S) + 2\ell$ edge swaps. It is then cellular in $S'$.
\end{proof}

\subsection{The width of annular bundles} \label{Sec:Width}

To apply \refprop{SpinesAcrossSolidTorus}, we will need to bound the width $w$ in that lemma. This subsection contains results bounding width.

\begin{lemma}[Existence of core curve]\label{Lem:HighWidthAnnulus}
For $1 \leq j \leq k$, let $A_j$ be an annulus and let $A'_j$ be the result of removing some disjoint discs and homotopically trivial annuli from $A_j$. Let $\mathcal{C}$ be a cell structure for $\bigcup_j A'_j$. Suppose certain 2-cells in $\calC$ are marked with an $X$. Define numbers as follows.
\begin{itemize}
\item Let $\ell_j$ be the length of $\partial A'_j - \partial A_j$.
\item Let $w_j$ denote the length of the shortest cellular curve in $A'_j$ joining the components of $\partial A_j$. 
\item Let $k$ be the maximal number of edges in the boundary of any 2-cell, counted with multiplicity.
\item Let $n_j$ be the number of cells in $\calC$ marked with an $X$.
\end{itemize}
Then, provided $\sum_j (n_j(k-2) + \ell_j) < \sum_j w_j$, there is a core curve of some $A_j$ lying in $A'_j$ that lies in just the 1-cells and 2-cells of $\mathcal{C}$, that intersects the 1-cells transversely,
and that misses all the 2-cells marked with an $X$ and all the 1-cells in $\partial A'_j$.
\end{lemma}

\begin{proof}
Since $\sum_j (n_j(k-2) + \ell_j) < \sum_j w_j$, there is some $j$ such that $n_j(k-2) + \ell_j < w_j$. Let $A = A_j$ and let $A' = A'_j$. Let $n = n_j$, $\ell = \ell_j$ and $w = w_j$. 

First, subdivide $\mathcal{C}$ to a cell structure $\mathcal{C}'$, without introducing any new vertices, so that each 2-cell has at most $3$ edges. We define path metric on $A'$ by declaring that each 1-cell has length $1$, each 2-cell with three edges is an equilateral Euclidean triangle and each 2-cell with one or two edges is a spherical hemisphere. Define $f\from A' \to [0,\infty]$
by letting $f(x)$ be the distance from $x$ to a specific boundary component of $A$; we set the distance to be infinite if the boundary component of $A$ cannot be reached from the point $x$ within $A'$, for example if $x$ lies in a disc enclosed by a removed annulus.
Let $L_i$ be $f^{-1}(i +(1/3))$ for integers $i$ satisfying $0 \leq i < w-1$.

Since we have chosen the inverse image of a real number that is neither an integer nor a half-integer, $L_i$ is a union of disjoint properly embedded 1-manifolds. It separates one component of $\partial A$ from the other. Each edge intersects at most one $L_i$. Hence, as each 2-cell of $\mathcal{C}'$ has at most $3$ sides, each 2-cell of $\mathcal{C}'$ intersects at most one $L_i$. So, each 2-cell of $\mathcal{C}$ intersects at most $(k-2)$ of the curves $L_i$. Therefore, the number of the $L_i$ going through a 2-cell with an $X$ is at most $n(k-2)$. The number intersecting $\partial A' - \partial A$ is at most $\ell$. Since we are assuming that $n(k-2) + \ell < w$, there is some $L_i$ that is disjoint from the 2-cells marked with an $X$ and from the 1-cells in $\partial A'$. It is a closed 1-manifold separating the two components of $\partial A$. It must therefore contain the required core curve.
\end{proof}

\begin{lemma}[Width of incoherent essential annular components]\label{Lem:WidthOneSurface}
Let $\calH$ be a pre-tetrahedral handle structure of $M = S \times [0,1]$, where $S$ is not a torus. Let $\calB(\calH)$ be a maximal generalised parallelity bundle for $\calH$.
\begin{itemize}
\item If $\calH$ contains a normal fibre that is not normally parallel to a boundary component, then let $S$ be one such fibre with least weight.
\item If every normal fibre in $\calH$ is normally parallel to a boundary component, then let $S$ be an almost normal fibre with least weight.
\end{itemize}
Suppose that $S$ is disjoint from the incoherent components of $\calB(\calH)$ that are neither $I$-bundles over discs nor boundary-trivial.
Let $\calB$ be a maximal generalised parallelity bundle for the handle structure obtained by cutting along $S$. Then the total width of the incoherent essential annular components of $\calB$ that are incident to $S$ is at most
$7640 \Delta(\calH)$.
\end{lemma}

\begin{remark}
In the above lemma, we require $S$ to be disjoint from all essential incoherent annular components of $\calB(\calH)$, but we allow $S$ to intersect $I$-bundles over discs and boundary-trivial components. This will hold, for example, for coherently bundled $\calH$.
\end{remark}

\begin{proof}
Let $\calB_I^{\mathrm{coh}}(\calH)$ and {$\calB_{\calA}$} be as in \reflem{B_ADisjointB_IH}. By that lemma, {$\calB_\calA$} is disjoint from $\calB_I^{\mathrm{coh}}(\calH)$. By assumption, $S$ is disjoint from the incoherent components of $\calB(\calH)$ that are neither $I$-bundles over discs nor boundary-trivial. For each component of {$\calB_\calA$}, pick a horizontal boundary component. Let $A_1 \cup \dots \cup A_k$ be the union of these annuli. We will show that their total width is at most $7640 \Delta(\calH)$.

We now form {$\calB_\calA^-$} from {$\calB_\calA$} by removing its intersection with $\calB_D(\calH)$ and its intersection with $\calB_\bdy(\calH)$, and by removing any non-parallelity handles of $\calB$. The result is disjoint from $\calB(\calH)$, and consists only of parallelity handles of $\calB$, so is an $I$-bundle. For $1 \leq j \leq k$, let $A'_j = A_j \cap {\calB^-_\calA}$. Then $A'_j$ is obtained from $A_j$ by removing some discs and homotopically trivial annuli. As in \reflem{HighWidthAnnulus}, let $\ell_j$ be the length of $\partial A'_j - \partial A_j$. Let $w_j$ be the length of a shortest curve in $A'_j$ joining the two components of $\partial A_j$. 
We will show that $\sum_j w_j \leq 7640 \Delta(\calH)$. This will imply that the total width of $A_1 \cup \dots \cup A_k$ is at most $7640 \Delta(\calH)$. 

Since $A'_1 \cup \dots \cup A'_k$ misses the parallelity handles of $\calB(\calH)$, it lies entirely in the non-parallelity handles of $\calH$.

\begin{figure}
  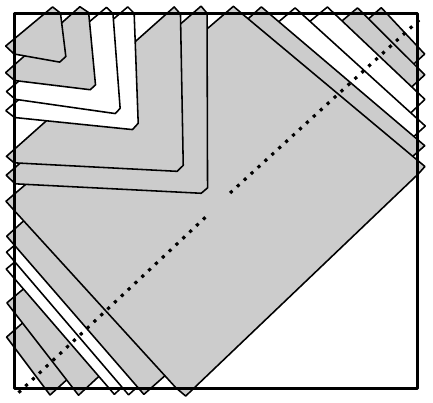
  \caption{The two outermost triangles and squares for each type in a 0-handle are marked by an $X$.}
  \label{Fig:NormalX}
\end{figure}

We give $S$ a cell structure, by declaring that each component of intersection between $S$ and a handle of $\calH$ is a 2-cell. There is one exception to this, which is when $S$ has a tubed piece. In this case, we add in a further 1-cell, which cuts the annular tube to a disc. Note that the boundary of each 2-cell then has length at most $18$, where the length is the number of 1-cells in the boundary.

We now place an $X$ on various of these 2-cells, first in 0-handles, then 1-handles, then 2-handles.
Within each non-parallelity 0-handle of $\calH$, for each type of triangle and square of $S$ in such a 0-handle, place an $X$ on the four outermost ones in the collection; see \reffig{NormalX}.
Similarly, if $S$ contains an almost normal piece, place an $X$ on this. A tetrahedral 0-handle without an almost normal piece has at most $4*5=20$ 
cells marked with an $X$, a semi-tetrahedral 0-handle has at most $12$ cells marked with an $X$, and a product handle has at most four. Multiplying by the minimal complexity of the cell, as in \refdef{Complexity}, there are at most $48 \Delta(\calH)$ such $X$'s that are not arising from almost normal pieces.  (The extremal case is a semi-tetrahedral 0-handle containing a square and two triangle types.)

Within each non-parallelity 1-handle of $\calH$, there are at most three types of disc of $S$, where all discs of the same type are normally parallel. Place an $X$ on the outermost four discs of each type.
In a 1-handle arising as the dual of a triangular face, there are at most 12 such $X$'s. In a 1-handle arising as the dual of a bigon, there are at most four such $X$'s. Note that a tetrahedral 0-handle meets four 1-handles coming from triangles, so it meets at most $48$ such $X$'s, or at most $96$ times the complexity. A semi-tetrahedral 0-handle meets two 1-handles from triangles, and two from bigons, so at most $32$ such $X$'s, or at most $128$ times the complexity. A non-parallelity product handle has three 1-handles from bigons, so meets at most a number of $X$'s equal to $96$ times the complexity. Thus there are at most $128\Delta(\calH)$ such $X$'s. 

Finally, there are the non-parallelity 2-handles. Each is incident to at least one non-parallelity 1-handles. 
So the number of non-parallelity 2-handles is at most $3$ times the number of non-parallelity 1-handles, which is no more than $24\Delta(\calH)$. 
Within such a 2-handle, all components of intersection with $S$ are normally parallel. We place an $X$ on the outermost four. Thus there are at most $3*24\Delta(\calH)*4 = 288\Delta(\calH)$ of these. 

In total, the number of $X$'s we have placed is at most 
\[48\Delta(\calH) + 1 +  128\Delta(\calH) + 288\Delta(\calH) = 464\Delta(\calH)+1 \leq 472 \Delta(\calH).\]

The total length of $(\partial A'_1 - \partial A_1) \cup \dots \cup (\partial A'_k - \partial A_k)$ is at most $88 \Delta(\calH)$, by \reflem{LengthBoundVerticalBoundaryAlmostNormal}.
\reflem{HighWidthAnnulus} implies that if $\sum_j w_j$ is more than
\[ \sum_j(n_j(k-2) + \ell_j) = (472\Delta(\calH)) *( 18-2 ) + 88\Delta(\calH) = 7640\Delta(\calH), \]
then there is a core curve $\gamma$ of one of these punctured annuli that lies in just the 1-cells and 2-cells of the cell structure and that misses all the 2-cells marked with an $X$ and the boundary of {$\bdy_h \calB_\calA^-$}. Thus $\gamma$ lies deep inside the parallelity bundle, away from outermost triangles and squares in each 0-handle. 

Suppose by way of contradiction that such a $\gamma$ exists. Because $\gamma$ does not meet outermost triangles and squares in each 0-handle, it is adjacent to two annuli $A_-$ and $A_+$, lying on the positive and negative sides of $S$, such that $A_- \cap S = \bdy A_-$ and $A_+ \cap S  = \bdy A_+$ and that are vertical in the parallelity bundle for $M \cut S$. These are parallel to annuli $\tilde A_-$ and $\tilde A_+$ in $S$, by the incompressibility of $S$. Let $\gamma_-$ and $\gamma_+$ be the curves $\partial A_- - \gamma$ and $\partial A_+ - \gamma$. See \reffig{ProofWidthOneSurface}.

\begin{figure}
  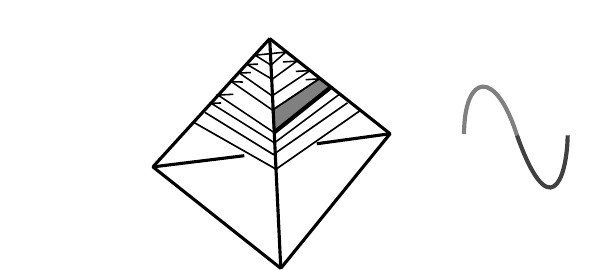
  \caption{Left: a core curve $\gamma$ misses all triangles and squares labeled $X$. Middle: Therefore it is adjacent to only triangles and squares that lie in the interior of the parallelity bundle; the $I$-bundles over $\gamma$ form annuli $A_-$ and $A_+$. Right: $A_-$ and $A_+$ are parallel to annuli $\tilde A_-$ and $\tilde A_+$ on $S$.}
  \label{Fig:ProofWidthOneSurface}
\end{figure}

Note that, near the two components of $\partial \tilde A_-$, $\tilde A_-$ emanates from the same side of $A_-$, because otherwise $A_- \cup \tilde A_-$ would be a Klein bottle in $S \times [0,1]$. Similarly, near the two components of $\partial \tilde A_+$, $\tilde A_+$ emanates from the same side of $A_+$. Hence, $\tilde A_-$ and $\tilde A_+$ emanate from opposite sides of $\gamma$ in $S$. 

We claim that $\gamma_-$ and $\gamma_+$ are disjoint curves in $S$. Suppose that, on the contrary, they intersect. Since $\gamma$ avoids the 2-cells labelled X, $\gamma_+$ also lies in a component of {$\mathcal{B}_\calA$} on the negative side of $S$. Let $H$ be its horizontal boundary component that contains $\gamma_+$. Since we are assuming that $\gamma_-$ and $\gamma_+$ intersect, then $\gamma_-$ also lies in $H$. Hence, $H$ is an annular subsurface of $S$ that contains the two curves $\gamma_-$ and $\gamma_+$ but does not contain the curve $\gamma$ between them. Therefore, $S \cut \gamma$ is an annulus. One way to see this is to note that the boundary component of $S \cut \gamma$ on the $\gamma_-$ side is homotopic to $\gamma_-$, and this is homotopic to $\gamma_+$ since they both are core curves of the annulus $H$, and $\gamma_+$ is homotopic to the other boundary component of $S \cut \gamma$. Hence, $S$ is a torus, contrary to assumption.

As a result of this claim, the annuli $\tilde A_-$ and $\tilde A_+$ intersect only along $\gamma$. Therefore, as in the proof of \reflem{B_ADisjointB_IH}, we may remove $\tilde A_- \cup \tilde A_+$ from $S$ and reglue its boundary components to form a new fibre $S'$. Then $S$ is the normal sum of $S'$ and the torus formed from $\tilde A_-$ and $\tilde A_+$ by gluing its boundary components. In particular, $S'$ is not normally parallel to a boundary component of $S \times [0,1]$. If $S$ was normal, then so is $S'$, and it has smaller weight than $S$, which is a contradiction. If $S$ was almost normal, then $S'$ is either almost normal or normal, and again it has smaller weight than $S$. Again, this is a contradiction.
\end{proof}

We also need a version of this for surfaces that are normally cylindrical on exactly one side.

\begin{lemma}[Width, $S$ cylindrical one side]\label{Lem:WidthCylindricalOneSide}
Let $\calH$ be a pre-tetrahedral handle structure of $M = S \times [0,1]$. Let $S$ be a normal fibre that is normally cylindrical on exactly one side. Let $\calB(\calH)$ be a maximal generalised parallelity bundle for $\calH$. Suppose that $S$ is disjoint from the incoherent components of $\calB(\calH)$ that are not $I$-bundles over discs.
Let $\calB$ be a maximal generalised parallelity bundle for the handle structure obtained by cutting along $S$. Then the total width of the incoherent essential annular components of $\calB$ that are incident to $S$ is at most $7640 \Delta(\calH)$.
\end{lemma}

\begin{proof}
Suppose that width of the incoherent essential annular components of $\calB$ is more than $7640\Delta(\calH)$.
The argument above gives that there are annuli $A_-$ and $A_+$ on opposite sides of $S$ with $A_- \cap S = \bdy A_-$ and $ A_+ \cap S = \bdy A_+$, with essential boundary curves and that are vertical in the parallelity bundle for $M \cut S$. This contradicts the assumption that $S$ is normally cylindrical on exactly one side.
\end{proof}

\section{Interpolating spines}\label{Sec:InterpolatingSpines}

\textbf{Road map:} At this stage, we know how to interpolate between almost normal and normal surfaces, using nearly normal surfaces.
We know how to transfer spines on the surfaces, and we have bounds on the number of edge swaps required to transfer a spine for each generalised isotopy move. In this section, we put that information together to complete the proof of \refthm{TriangulationProductOneVertex}. 
The proof involves four results on transferring spines across $S\times[0,1]$ that have somewhat subtly different hypotheses, namely Propositions~\ref{Prop:SpinesAcrossAProduct} and~\ref{Prop:EdgeSwapBoundSimplified}, and Theorems~\ref{Thm:MainTheoremProductsNoNormalFibre} and \ref{Thm:MainTheoremProducts}. Before diving into the details, we summarise the argument, and how these fit together. 

Recall that \refthm{TriangulationProductOneVertex} gives upper and lower bounds on the minimal number of tetrahedra in a triangulation of $S\times [0,1]$ with given 1-vertex triangulations of $S\times\{0\}$ and $S\times\{1\}$. It will be a fairly rapid consequence of \refthm{MainTheoremProducts}, which provides an upper bound on the number of edge swaps required to take a cellular spine in $S\times\{0\}$ to one in $S\times\{1\}$ when $S\times[0,1]$ has a coherently bundled pre-tetrahedral handle structure. Theorem~\ref{Thm:MainTheoremProducts} is proved by induction on the complexity of the handle structure. The inductive step implies that whenever the handle structure contains a normally acylindrical fibre that is not normally parallel to a boundary component, we may cut along it, giving two coherently bundled handle structures with smaller complexity. Thus \refthm{MainTheoremProducts} follows quickly from a similar result, \refthm{MainTheoremProductsNoNormalFibre}, with the additional hypothesis that the only normally acylindrical fibres are the ones that are boundary parallel.

Theorem~\ref{Thm:MainTheoremProductsNoNormalFibre} is also proved by induction, but the argument is more delicate. The proof divides into several cases; in most we consider the existence of a normal fibre that is not normally parallel to the boundary. Necessarily, this is normally cylindrical. Depending on the case, we may pick such a fibre that is innermost (as in \refdef{InnermostCylindricalSurface}) or of least weight. If there is no such fibre, then we pick an almost normal fibre with least weight, which exists by \refthm{AlmostNormalExists}. Inductively, we get a collection of normal fibres. We know how to interpolate between these normal fibres, using almost normal and nearly normal surfaces.
We also have upper bounds on the number of edge swaps required to transfer a spine in one fibre to the next. 
In Propositions~\ref{Prop:SpinesAcrossAProduct} and~\ref{Prop:EdgeSwapBoundSimplified}, we use these to bound the total number of edge swaps in terms of the complexity of the handle structure and the width of certain annular bundles. In \refprop{SpinesAcrossAProduct}, these annular bundles are components of the generalised parallelity bundles $\calB_i$ given by \reflem{IterationGenIsotopy}. However, we really need the generalised parallelity bundles $\calB$ for $M\cut S$, where $S$ is a normal surface in $M$. In \refprop{EdgeSwapBoundSimplified}, we obtain an upper bound on the number of edge swaps in terms of the width of the annular components of $\calB$. The bounds are then used in \refthm{MainTheoremProductsNoNormalFibre}.

\begin{proposition}[Edge swap bounds under generalised isotopy]
  \label{Prop:SpinesAcrossAProduct}
Let $\calH$ be a pre-tetrahedal handle structure for $M = S \times [0,1]$. 
Let $S$ be a normal or almost normal fibre with a transverse orientation.
Let $S= S_0, \dots, S_n = S'$ be a sequence of surfaces as in Lemma \ref{Lem:IterationGenIsotopy}.
Let $M_i$ be the manifold between $S_i$ and $S_{i+1}$. Let $\mathcal{B}_i$ be a maximal generalised parallelity bundle for the manifold between $S_i$ and $S_{i+1}$. 
In the case where $S$ is normal, let $\calH'$ be the resulting handle structure for the component of $M \cut S$ into which $S$ points. When $S$ is almost normal, let $\calH'$ equal $\calH$.
\begin{enumerate}
\item Let $\Gamma$ be a cellular spine for $S$. Let $w$ denote the total width of the incoherent essential annular components of $\calB_0 \cup \dots \cup \calB_{n-1}$. 
Then there is a sequence of at most $k  \Delta(\calH') + 8w$ edge swaps taking the cellular spine $\Gamma$ to a spine $\Gamma'$ that is cellular in $S'$. 
\item Let $\Gamma'$ be a cellular spine for $S'$. Then there is a sequence of at most $k \Delta(\calH')$ edge swaps taking $\Gamma'$ to a spine $\Gamma$ that is cellular in $S$.
\end{enumerate}
In both cases, $k = 4224 g(S) + 32472$, or $4224 g(S) + 17512 + 34t$, where $t=440$ is the constant of \refprop{SpinesAcrossSolidTorus}.
\end{proposition}

\begin{proof}
For each $i$, there is a sequence of generalised isotopy moves taking $S_i$ to $S_{i+1}$, as in Lemma \ref{Lem:IterationGenIsotopy}. Let $S_i = S_i^0, \dots, S_i^{m(i)} = S_{i+1}$ be these surfaces.
Each of these surfaces $S_i^j$ inherits its canonical handle structure. We will work with the associated cell structure for this, which we denote $\mathcal{C}_i^j$. For each surface $S_i^j$, we will pick a spine $\Gamma_i^j$ that is cellular with respect to $\mathcal{C}_i^j$. For each $j$, the product region between $S_i^j$ and $S_i^{j+1}$ provides a homeomorphism $S_i^j \rightarrow S_i^{j+1}$ that is well-defined up to isotopy. We will ensure that the image of $\Gamma_i^j$ under this homeomorphism will be related to $\Gamma_i^{j+1}$ by a controlled number of edge contractions and expansions. Note that depending whether we are proving (1) or (2) of the proposition,
we either create $\Gamma_i^{j+1}$ from $\Gamma_i^j$ or we create $\Gamma_i^j$ from $\Gamma_i^{j+1}$.

For $i > 0$, let $\calH_i$ be the handle structure on $M_i$. When $i = 0$ and $S$ is normal, let $\calH_0$ be the handle structure on the manifold $M_0$ between $S_0$ and $S_1$. When $i = 0$ and $S$ is almost normal, let $\calH_0$ denote the handle structure on the component of $M \cut S_1$ containing $S$.
Thus, $\sum_i \Delta(\calH_i) \leq \Delta(\calH')$ by Lemma \ref{Lem:ComplexityDecomposition}. Many of the bounds below will be in terms of the length of the vertical boundary of $\calB_i$. This length is at most $88 \Delta(\calH_i)$ by \reflem{BoundOnLengthVerticalBoundary} or \reflem{LengthBoundVerticalBoundaryAlmostNormal}. So the sum of these lengths over all $i$ is at most $88 \Delta(\calH')$.

We consider the various types of move in turn.

\medskip
\emph{Case 1.} $S_i^{j+1}$ is obtained from $S_i^j$ by a generalised isotopy move across an $I$-bundle over a disc, as in \refdef{GenIsotopyMoveAcrossB}. 
\medskip

Denote the $I$-bundle over the disc by $E\times I$, where the length of {$\partial E$}
 in $S_i^j$ is $\ell$. By \reflem{MovingSpineGenIsotopyAcrossB}, if $S_i^j$ is equipped with a cellular spine $\Gamma_i^j$, then we may build a cellular spine $\Gamma_{i}^{j+1}$ on $S_i^{j+1}$ by performing at most $6(2\ell+1)g(S)+92\ell-64 \leq 18 \ell g(S) + 92\ell$ edge swaps. By the same lemma, if $S_i^{j+1}$ has a cellular spine $\Gamma_{i}^{j+1}$, then we may build a cellular spine $\Gamma_{i}^j$ on $S_{i}$ by performing at most
$6 g(S) + 8 \ell +40 \leq 6\ell g(S) + 48 \ell$ edge swaps. Note that the sum of all such $\ell$ over all these moves is at most the length of the vertical boundary of $\calB_0, \dots, \calB_{n-1}$, which is at most $88\Delta(\calH')$ by \reflem{BoundOnLengthVerticalBoundary} or \reflem{LengthBoundVerticalBoundaryAlmostNormal}. Therefore, the total number of edge swaps, over all these generalised isotopy moves across $I$-bundles over discs, is at most $(1584g(S) + 8096) \Delta(\calH')$ when going from $S$ to $S'$, and at most $(528 g(S)+ 4224) \Delta(\calH')$ when going in the other direction.

\medskip
\emph{Case 2.} $S_i^{j+1}$ is obtained from $S_i^j$ by a clean annular simplification across an essential annular bundle.
\medskip

In this situation, there is an annular region $A \times I$ of $\calB_i$ such that $(A \times I) \cap S_i^j = A \times \bdy I$. Let $w'$ denote the width of $A\times I$. A component of $(M \cut S_i^j) \cut (A \times I)$ is a solid torus $W$. The only components of $\calB_i$ lying in $W$ are $I$-bundles over discs. Let $A'$ be the annulus $\bdy W \cut (\bdy A \times I)$. The annular simplification removes $A' \cup (A \times \bdy I)$ from $S_i^j$ and replaces it by $V' = (\bdy A \times I) - \bdy W$, giving the surface $S_i^{j+1}$. Let $\ell$ be the length of the annulus $V'$.

By \refprop{SpinesAcrossSolidTorus}, if $S_i^j$ is equipped with a cellular spine, then we may build a cellular spine on $S_i^{j+1}$ by a sequence of at most
$6g(S)+ 34t\Delta+8w' + \ell(6g(S)+50) + 5$ edge swaps, where $t=440$ is the constant from \refprop{SpinesAcrossSolidTorus}. This is at most $12g(S)\ell + 55\ell+ 34t\Delta + 8w'$. 
If instead $S_i^{j+1}$ is equipped with a cellular spine, then we may build a cellular spine on $S_i^j$ by a sequence of at most 
$6g(S)+2\ell+30t\Delta+9 \leq 6g(S)\ell + 11\ell + 30t\Delta$
edge swaps. Again, the sum of all such $\ell$ over all these moves is at most $88\Delta(\calH')$ by \reflem{BoundOnLengthVerticalBoundary} or \reflem{LengthBoundVerticalBoundaryAlmostNormal}.
The sum of $\Delta$ over all such solid tori $W$, one for each annular simplification, is at most $\Delta(\calH')$. Therefore, the total number of edge swaps, over all these annular simplifications, is at most 
$(1056 g(S) + (34t + 4840)) \Delta(\calH')  + 8w$ when going from $S$ to $S'$ and at most $(528 g(S) + (30t + 968)) \Delta(\calH')$ when going in the other direction.

\medskip
\emph{Case 3.} $S_i^{j+1}$ is obtained from $S_i^j$ by a trivial annular simplification. 
\medskip

This replaces an annulus in $S_i^j$ with a vertical boundary component of $\calB_i$. Let $\ell$ be the length of this vertical boundary component.
By \reflem{EdgeSwapTrivialAnnSimp}, we may convert a cellular spine in $S_i^j$ to one in $S_i^{j+1}$ using at most
$6g(S)(\ell+1) + 50\ell\leq 12g(S)\ell+50\ell$ edge swaps. Alternatively, we
may convert a cellular spine in $S_i^{j+1}$ to one in $S_i^j$ using at most $6g(S) + 2 \ell \leq 6g(S) \ell + 2 \ell$ edge swaps. Again, the sum of all such $\ell$ over all these moves is at most $88 \Delta(\calH')$ by \reflem{BoundOnLengthVerticalBoundary} or \reflem{LengthBoundVerticalBoundaryAlmostNormal}. Thus, the total number of edge swaps is at most
$(1056g(S) + 4400)\Delta(\calH')$ when going from $S$ to $S'$, and at most $(528g(S) + 176)\Delta(\calH')$ when going from $S'$ to $S$. 

\medskip
\emph{Case 4.} $S_i^{j+1}$ is obtained from $S_i^j$ by a parallelisation isotopy. 
\medskip

Lemma~\ref{Lem:EdgeSwapBoundParallelisation} states that we may convert a cellular spine in one surface to a cellular spine in the other using at most $6g(S) + 2\ell$ edge swaps, which is at most $(6g(S)+2)\ell$. So, in total, the number of edge swaps is at most $(528g(S) + 176)\Delta(\calH')$

\medskip
\emph{Case 5.} $S_i^{j+1}$ is obtained from $S_i^j$ by a compression isotopy. 
\medskip

By \reflem{CompressionIsotopyEffect}, the compression isotopy takes the canonical handle structure for $S_i^j$ to that for $S_i^{j+1}$. So, we may transfer a spine from one surface to the other without using any edge swaps.

\medskip
\emph{Case 6.} $S_i^{j+1}$ is obtained from $S_i^j$ by a tube compression. 
\medskip

By \reflem{EffectTubeComp}, the tube compression takes the canonical handle structure for $S_i^j$ to that for $S_i^{j+1}$. So, again we may transfer a spine from one surface to the other without using any edge swaps.

\medskip

So, adding the total number of edge swaps in each of Cases~1 to~6,
we obtain the required upper bounds on the number of edge swaps.
\end{proof}

The above proposition bounds edge swaps in terms of the width of the essential incoherent annular components of the parallelity bundles $\calB_0, \dots, \calB_{n-1}$. We now replace this with a bound that is solely in terms of the parallelity bundle for $S$.

\begin{proposition}[Edge swap bounds, simplified version]\label{Prop:EdgeSwapBoundSimplified}
Let $M=S\times[0,1]$, and let $\calH$ be a pre-tetrahedal handle structure for $M$. 
Let $S$ be a normal fibre with a transverse orientation.
Let $S= S_0, \dots, S_n = S'$ be a sequence of surfaces as in Lemma \ref{Lem:IterationGenIsotopy}.
Let $\calB$ be a maximal generalised parallelity bundle for the component of $M \cut S$ into which $S$ points.
Let $w$ be the sum of the widths of the essential incoherent annular components of $\calB$ that are incident to $S$.
Let $\Gamma$ be a cellular spine for $S$. Then $\Gamma$ can be converted to a cellular spine for $S'$ using at most
$k \Delta(\mathcal{H}) + 122240\, \Delta(\calH) + 8w$ edge swaps. Here, $k$ is the constant from \refprop{SpinesAcrossAProduct}.
\end{proposition}

\begin{proof}
For $0 \leq i \leq n-1$, let $M_i$ be the manifold between $S_i$ and $S_{i+1}$. Let $\calH_i$ be the handle structure on $M_i$. Now, when $S$ is isotoped a little into $M \cut S$ in the direction specified by its transverse orientation, the resulting surface is acylindrical in $M \cut S$ on the side into which it does not point, since it is normally parallel to the boundary. Hence, by Proposition \ref{Prop:CylindricalWrongSide}, each of the surfaces $S_i$ is acylindrical in $M \cut S$ on the side into which it does not point. In particular,  for $i > 0$, $S_i$ is acylindrical on that side in the manifold $M_{i-1} \cup M_i$.
By the way $S_1$ is constructed, specifically conclusion (4) and (7) in Lemma \ref{Lem:IterationGenIsotopy}, the essential incoherent annular components of
$\calB_0$ are a subset of the essential incoherent annular components of $\calB$. Hence, their total width is at most $w$.

For $i > 0$, suppose that $\calB_i$ has at least one essential incoherent annular component. Then the hypotheses of Lemma~\ref{Lem:WidthCylindricalOneSide} apply to the surface $S_i$ in $M_{i-1} \cup M_i$ for the following reason. We have shown above that $S_i$ is acylindrical on the $M_{i-1}$ side. By conclusion~(7) in Lemma~\ref{Lem:IterationGenIsotopy}, the isotopy taking $S_i$ to $S_{i+1}$ moves across the essential incoherent components of $\calB_i$. Hence, $S_i$ is cylindrical on exactly one side in $M_{i-1} \cup M_i$. Any incoherent component of $\calB_{i-1}$ that is not an $I$-bundle over a disc has been isotoped across in the isotopy taking $S_{i-1}$ to $S_i$. Hence, $S_i$ is disjoint from these components. 
So by Lemma~\ref{Lem:WidthCylindricalOneSide}, applied to the surface $S_i$ in $M_{i-1} \cup M_i$ for $i > 0$, the total width of the incoherent essential annular components of $\calB_i$ is at most $7640 \, (\Delta(\calH_{i-1}) + \Delta(\calH_i))$. Summing this over all $i$, the total width of the incoherent essential annular components of $\calB_0, \dots, \calB_{n-1}$ is 
at most $15280\,\Delta(\mathcal{H}) + w$. Hence, by \refprop{SpinesAcrossAProduct}, the number of edge swaps that one must apply to the spine in $S$ to make it cellular in $S'$ is at most $k \Delta(\mathcal{H}) + 122240\, \Delta(\calH) + 8w$. 
\end{proof}

Recall from \refdef{CoherentlyBundled} that $\calH$ is coherently bundled if no vertical boundary component of its maximal generalised parallelity bundle $\calB$ is an incompressible annulus with both boundary components on $S\times\{0\}$ or $S\times\{1\}$. It follows that $\calH$ is coherently bundled if and only if the only components of $\calB$ with horizontal boundary disjoint from $S\times\{0\}$ or disjoint from $S\times\{1\}$ are $I$-bundles over discs or boundary-trivial; see \refthm{IncompressibleHorizontalBoundary}.

\begin{theorem}[Edge swaps across product, no acylindrical fibre]\label{Thm:MainTheoremProductsNoNormalFibre}
Let $S$ be an orientable, closed surface with genus at least two.
Let $\mathcal{H}$ be a coherently bundled pre-tetrahedral handle structure of $S \times [0,1]$, such that every normally acylindrical normal fibre in $\mathcal{H}$ is normally parallel to a boundary component.
Let $\Gamma$ be a cellular spine for $S \times \{ 0 \}$. Then there is a sequence of at most $c \Delta(\mathcal{H})$ edge swaps taking $\Gamma$ to a spine $\Gamma'$ that is cellular with respect to $S\times\{1\}$. Here, $c = 9 k$, where $k$ is the constant from \refprop{SpinesAcrossAProduct}.
\end{theorem}

\begin{proof}
We prove \refthm{MainTheoremProductsNoNormalFibre} by induction on $\Delta(\mathcal{H})$. 
Note that although $\Delta(\calH)$ is not necessarily an integer, it is a non-negative rational number with bounded denominator, and so an inductive proof is permitted.
The induction starts with $\Delta(\mathcal{H}) = 0$, in which case $S \times \{ 0 \}$ and $S \times \{ 1 \}$ are normally parallel. Then $\Gamma$ becomes a spine for $S \times \{ 1 \}$ using no swaps.

We now consider the inductive step. By assumption, every normal fibre in $M = S \times [0,1]$ is normally cylindrical on at least one side, other than the two fibres that are normally parallel to the boundary.

\medskip
\emph{Case 1.} There exists a normal fibre that is normally cylindrical on the $S\times\{0\}$ side only. This case is illustrated schematically in \reffig{ProductsNoNormalFibreCase12A}, left.
\medskip

\begin{figure}
  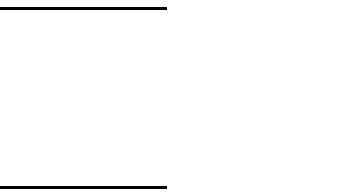
  \hspace{.2in}
  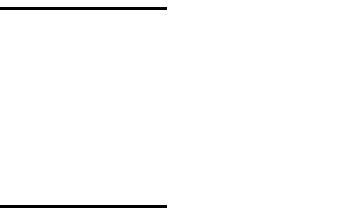
  \caption{Left: Case 1. Right: Case 2A.}
  \label{Fig:ProductsNoNormalFibreCase12A}
\end{figure}

Then let $S$ be one that is innermost, as in \refdef{InnermostCylindricalSurface}. This divides $M$ into two manifolds $M_+$ and $M_-$, where $M_+$ contains $S \times \{ 1 \}$.
These have handle structures $\mathcal{H}_+$ and $\mathcal{H}_-$ such that $\Delta(\mathcal{H}_-) + \Delta(\mathcal{H}_+) = \Delta(\mathcal{H})$ by \reflem{ComplexityDecomposition}.

Observe first that $\mathcal{H}_+$ is coherently bundled by \reflem{AcylindricalOneSide}. Because $S$ is innermost, no fibre in $\mathcal{H}_+$ is normally acylindrical other than those that are boundary parallel, by \reflem{CutAlongInnermost}. Finally, $\Delta(\mathcal{H}_+) < \Delta(\mathcal{H})$. Hence, the inductive hypothesis applies to $\calH_+$. Thus, using at most $c \Delta(\mathcal{H}_+) $ edge swaps, we convert a cellular spine in $S$ into a spine that is cellular in $S \times \{ 1 \}$.

As for $\calH_-$, apply generalised isotopy moves to $S$, all in the $S \times \{ 0 \}$ direction, satisfying the conclusions of Lemma \ref{Lem:IterationGenIsotopy}. 
By \refcor{IsotopeCylindricalOneSide}, the final surface in this sequence is normally parallel to $S \times \{ 0 \}$.

We do not have a good bound on the width of the incoherent essential annular regions in $\mathcal{H}_-$. However, note that a bound on width is not required when going from $S \times \{ 0 \}$ to $S$: by \refprop{SpinesAcrossAProduct}~(2), applied to $\mathcal{H}_-$, we obtain a cellular spine for $S$ from $\Gamma$ using at most $k \Delta(\mathcal{H}_-)$ edge swaps, which is less than $c \Delta(\mathcal{H}_-)$ edge swaps.
So in total the number of edge swaps taking $\Gamma$ to a cellular spine in $S \times \{ 1 \}$ is at most $c\Delta(\calH_+) + c\Delta(\calH_-) = c\Delta(\mathcal{H})$.

\medskip
\emph{Case 2.} Every normal fibre is normally cylindrical on the $S \times \{ 1 \}$ side at least, other than those that are boundary parallel, or there are no normal fibres that are not boundary parallel. 

\medskip
\emph{Case 2A.} Every normal fibre is normally cylindrical on the $S \times \{ 0 \}$ side also, other than those that are boundary parallel, or there are no normal fibres that are not boundary parallel. This case is illustrated schematically in \reffig{ProductsNoNormalFibreCase12A}, right.
\medskip

Pick a normal fibre $S$ with least weight that is not boundary parallel, or if there is no normal fibre that is not boundary parallel, then pick an almost normal fibre with least weight, which exists by \refthm{AlmostNormalExists}. Again $M\cut S$ has components $M_+$ and $M_-$, with $M_+$ containing $S\times\{1\}$.
Let $\calB_\pm$ be a maximal generalised parallelity bundle for $M_\pm$. By Proposition \ref{Prop:IsotopeFromLeastWeight}, we can interpolate from $S$ to $S \times \{ 0 \}$  using generalised isotopy moves satisfying the conclusions of Lemma~\ref{Lem:AlmostNormGenIsotopyDisjointB} and~\ref{Lem:GeneralisedIsotopyExtraProperties} (in the case where $S$ is almost normal) or~\ref{Lem:NormGenIsotopyDisjointB} (in the case where $S$ is normal). 
We view this sequence of moves as a sequence as in Lemma~\ref{Lem:IterationGenIsotopy} (with $n=1$ there). Similarly, we can interpolate from $S$ to $S \times \{ 1 \}$  using generalised isotopy moves satisfying the conclusions of Lemma~\ref{Lem:AlmostNormGenIsotopyDisjointB} and~\ref{Lem:GeneralisedIsotopyExtraProperties} (in the case where $S$ is almost normal) or~\ref{Lem:NormGenIsotopyDisjointB} (in the case where $S$ is normal). 
By \reflem{WidthOneSurface}, which applies because $\calH$ is coherently bundled, the total width of the incoherent essential annular components of $\calB_+$ is at most $7640\, \Delta(\mathcal{H})$. Hence, by \refprop{SpinesAcrossAProduct} parts~(1) and~(2), there is a sequence of at most
\[ k\Delta(\calH) +  k \Delta(\calH) + 8 \times 7640\, \Delta(\calH) < c \Delta(\mathcal{H}) \]
edge swaps taking $\Gamma$ to a cellular spine in $S \times \{ 1 \}$. We are using here the fact that $8 \times 7640 < 7k$.

\medskip
\emph{Case 2B.} There is some normal fibre that is normally cylindrical on the $S \times \{ 1 \}$ side only. This case is illustrated in \reffig{ProductsNoNormalFibreCase2B}. 
\medskip

\begin{figure}
\begingroup%
  \makeatletter%
  \providecommand\color[2][]{%
    \errmessage{(Inkscape) Color is used for the text in Inkscape, but the package 'color.sty' is not loaded}%
    \renewcommand\color[2][]{}%
  }%
  \providecommand\transparent[1]{%
    \errmessage{(Inkscape) Transparency is used (non-zero) for the text in Inkscape, but the package 'transparent.sty' is not loaded}%
    \renewcommand\transparent[1]{}%
  }%
  \providecommand\rotatebox[2]{#2}%
  \newcommand*\fsize{\dimexpr\f@size pt\relax}%
  \newcommand*\lineheight[1]{\fontsize{\fsize}{#1\fsize}\selectfont}%
  \ifx\svgwidth\undefined%
    \setlength{\unitlength}{144bp}%
    \ifx\svgscale\undefined%
      \relax%
    \else%
      \setlength{\unitlength}{\unitlength * \real{\svgscale}}%
    \fi%
  \else%
    \setlength{\unitlength}{\svgwidth}%
  \fi%
  \global\let\svgwidth\undefined%
  \global\let\svgscale\undefined%
  \makeatother%
  \begin{picture}(1,0.60000002)%
    \lineheight{1}%
    \setlength\tabcolsep{0pt}%
    \put(0,0){\includegraphics[width=\unitlength,page=1]{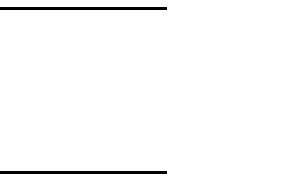}}%
    \put(0.5716011,0.28816941){\color[rgb]{0,0,0}\makebox(0,0)[lt]{\lineheight{1.25}\smash{\begin{tabular}[t]{l}$S$ innermost\end{tabular}}}}%
    \put(0.57291642,0.54454169){\color[rgb]{0,0,0}\makebox(0,0)[lt]{\lineheight{1.25}\smash{\begin{tabular}[t]{l}$S\times\{1\}$\end{tabular}}}}%
    \put(0.5747765,0.0073386){\color[rgb]{0,0,0}\makebox(0,0)[lt]{\lineheight{1.25}\smash{\begin{tabular}[t]{l}$S\times\{0\}$\end{tabular}}}}%
    \put(0,0){\includegraphics[width=\unitlength,page=2]{ProductsNoNormalFibreCase2B.pdf}}%
    \put(0.00000009,0.46641669){\color[rgb]{0,0,0}\makebox(0,0)[lt]{\lineheight{1.25}\smash{\begin{tabular}[t]{l}$M_+$\end{tabular}}}}%
    \put(0.00000009,0.20599963){\color[rgb]{0,0,0}\makebox(0,0)[lt]{\lineheight{1.25}\smash{\begin{tabular}[t]{l}$M_-$\end{tabular}}}}%
    \put(0.00000008,0.09662475){\color[rgb]{0,0,0}\makebox(0,0)[lt]{\lineheight{1.25}\smash{\begin{tabular}[t]{l}coherently bundled\end{tabular}}}}%
  \end{picture}%
\endgroup%

  \caption{Case 2B.}
  \label{Fig:ProductsNoNormalFibreCase2B}
\end{figure}

Let $S$ be one that is innermost in $M$. Let $M_-$ be the component of $M \cut S$ containing $S  \times \{ 0 \}$ and let $\mathcal{H}_-$ be the handle structure that it inherits. It is coherently bundled by \reflem{AcylindricalOneSide}. 
Note that in $M_-$, every fibre is normally cylindrical in $M_-$ on the $S \times \{ 0 \}$ side at least, other than those that are boundary-parallel, by \reflem{CutAlongInnermost}.

\medskip
\emph{Case 2B(i).} Every normal fibre in $M_-$ is normally cylindrical in $M_-$ on both sides, other than those that are boundary parallel. This case is illustrated in \reffig{ProductsNoNormalFibreCase2Bi-ii}, left.
\medskip

\begin{figure}
  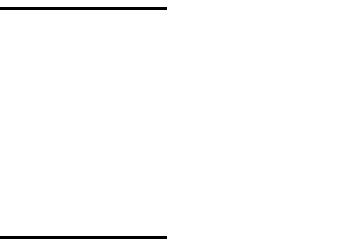
  \vspace{.1in}
  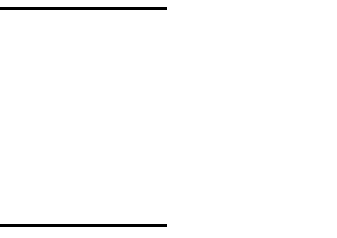
  \caption{Left: Case 2B(i). Right: Case 2B(ii).}
  \label{Fig:ProductsNoNormalFibreCase2Bi-ii}
\end{figure}

Then pick a normal fibre $S'$ in $M_-$ that is not boundary parallel and that is of least weight. If there is no such fibre, then let $S'$ be an almost normal fibre in $M_-$ of least weight. By Proposition~\ref{Prop:IsotopeFromLeastWeight}, we can interpolate from $S'$ to $S \times \{ 0 \}$ using generalised isotopy moves as in Lemma~\ref{Lem:AlmostNormGenIsotopyDisjointB} and~\ref{Lem:GeneralisedIsotopyExtraProperties} (in the case where $S'$ is almost normal) or~\ref{Lem:NormGenIsotopyDisjointB} (in the case where $S'$ is normal). By the same proposition, we can also interpolate from $S'$ to $S$ using generalised isotopy moves also as in those lemmas. Let $\mathcal{B}$ be a maximal generalised parallelity bundle for $M_- \cut S'$. The total width of the essential incoherent annular components of $\mathcal{B}$ is at most $7640 \, \Delta(\mathcal{H}_-)$ by \reflem{WidthOneSurface}, which again applies because $\calH$ is coherently bundled.
Hence, the total number of edge swaps taking $\Gamma$ to a cellular spine in $S$ is at most 
$2k \Delta(\mathcal{H}_-) + 8 \times 7640 \, \Delta(\calH_-)$ by \refprop{SpinesAcrossAProduct}, again with $n=1$ in that proposition. 

We can also perform generalised isotopy moves to $S$ in the $S \times \{ 1 \}$ direction, as in Lemma~\ref{Lem:IterationGenIsotopy}. The resulting normal surface admits no further generalised isotopy moves. So it is normally parallel to $S \times \{ 1 \}$. (Recall that we are in Case 2.) If $\mathcal{B}'$ is a maximal generalised parallelity bundle for $M \cut S$, then the total width of the incoherent essential annular components of $\mathcal{B}'$ is at most $7640\,\Delta(\mathcal{H})$ by \reflem{WidthCylindricalOneSide}. 
Hence, by Proposition~\ref{Prop:EdgeSwapBoundSimplified}, the number of edge swaps that one must apply to the spine in $S$ to make it cellular in $S \times \{ 1 \}$ is at most $k \Delta(\mathcal{H}) + 122240\, \Delta(\calH) + 8 \times 7640\,\Delta(\mathcal{H})$.
So in total, the number of edge swaps is at most
$9 k \Delta(\mathcal{H}) = c\Delta(\mathcal{H})$.
Here, we are using that $244480 \leq 6k$.

\medskip
\emph{Case 2B(ii).} There is a normal fibre in $M_-$ that is normally cylindrical in $M_-$ on the $S \times \{ 0 \}$ side only. This case is illustrated in \reffig{ProductsNoNormalFibreCase2Bi-ii}, right. 
\medskip

Let $S'$ be one that is innermost in $M_-$.
We can perform generalised isotopy moves taking $S'$ to $S \times \{ 0 \}$. This is because, as mentioned above, every normal fibre in $M_-$ is normally cylindrical on the $S \times \{ 0 \}$ side at least, and so \refprop{IsotopeEverythingCylindrical} applies. Also, by \refprop{IsotopeEverythingCylindrical}, we can get from $S'$ to $S\times\{1\}$ in $M$ using generalised isotopy moves, because every normal surface in $M$ is cylindrical on this side, as we are in Case~2. We need to find an upper bound on the total width of the incoherent essential annular regions lying on the $S \times \{ 1 \}$ side of $S'$. Any such region must contain an incoherent essential annular region for $S$, since $S'$ is acylindrical in $M \cut S$ on the $S$ side. The total width of these annular regions for $S$ is at least that of $S'$. By \reflem{WidthCylindricalOneSide}, the total width for $S$ is at most $7640\, \Delta(\mathcal{H})$, hence the same bound applies to $S'$. So, by \refprop{SpinesAcrossAProduct}~(2) and Proposition \ref{Prop:EdgeSwapBoundSimplified}, the number of edge swaps that one must apply to take $\Gamma$ from $S\times\{0\}$ to $S'$ to $S\times\{1\}$ is at most $2k\Delta(\mathcal{H}) + 122240\, \Delta(\calH) + 8 \times 7640\,\Delta(\mathcal{H}) < c \Delta(\mathcal{H})$.
\end{proof}

The following is the main technical theorem in the paper.

\begin{theorem}[Edge swaps across product]\label{Thm:MainTheoremProducts}
Let $S$ be a closed orientable surface with genus at least two.
Let $\mathcal{H}$ be a coherently bundled pre-tetrahedral handle structure of $S \times [0,1]$. Let $\Gamma$ be a cellular spine for $S \times \{ 0 \}$. Then there is a sequence of at most $c \, \Delta(\mathcal{H})$ edge swaps taking $\Gamma$ to a spine $\Gamma'$ that is cellular with respect to $S \times \{ 1 \}$. Here, $c$ is the constant from \refthm{MainTheoremProductsNoNormalFibre}.
\end{theorem}

\begin{proof}
Let $\mathcal{H}$ be a coherently bundled pre-tetrahedral handle structure of $S \times [0,1]$. Let $S_1, \dots, S_n$ be a maximal collection of disjoint normal fibres, none of which is normally parallel to a boundary component, no two of which are normally parallel and all of which are normally acylindrical.
Note this collection may be empty. 
Suppose the surfaces are labelled so that they appear in the order $S \times \{ 0 \} = S_0, S_1, \dots, S_n, S_{n+1} = S \times \{ 1 \}$ in the product structure. Let $\mathcal{H}_i$ be the handle structure inherited by the submanifold between $S_i$ and $S_{i+1}$. Then by \reflem{AcylindricalTwoHandleStructures}, $\mathcal{H}_i$ contains no normal fibre that is normally acylindrical and not normally parallel to a boundary component. By \reflem{CutProductAlongAcylindrical}, each $\mathcal{H}_i$ is coherently bundled.

We build cellular spines $\Gamma_0, \dots, \Gamma_{n+1}$ for $S_0, \dots, S_{n+1}$ as follows. We start with $\Gamma_0 = \Gamma$. Each $\calH_i$ is a coherently bundled pre-tetrahedral handle structure of $S\times [0,1]$ such that every normally acylindrical fibre is normally parallel to a boundary component. Thus, inductively, we may apply \refthm{MainTheoremProductsNoNormalFibre} to $\calH_i$ and $\Gamma_i$ to obtain $\Gamma_{i+1}$ in $S_{i+1}$ by at most $c \Delta(\mathcal{H}_i)$ edge swaps. We set $\Gamma' = \Gamma_{n+1}$. The total number of edge swaps is at most $c \Delta(\mathcal{H}_0) + \dots + c \Delta(\mathcal{H}_n)$, which is $c \Delta(\mathcal{H})$, by \reflem{ComplexityDecomposition}.
\end{proof}

We are now able to prove \refthm{TriangulationProductOneVertex}.

\begin{proof}[Proof of \refthm{TriangulationProductOneVertex}]
Suppose that $\calT_0$ and $\calT_1$ are 1-vertex triangulations of $S$. As in the proof of \refprop{UpperBound}, we can triangulate $S \times [0,1]$ so that $S \times \{ 0 \}$ and $S \times \{ 1 \}$ both have the triangulation $\calT_0$, using at most $18g(S)$ tetrahedra. A sequence of 2-2 Pachner moves relating $\calT_0$ and $\calT_1$ specifies a sequence of attachments of tetrahedra onto $S \times \{ 1 \}$. We end with a triangulation of $S \times [0,1]$ satisfying the required properties. It follows that the minimal number of 2-2 Pachner moves relating $\calT_0$ to $\calT_1$ gives an upper bound on the minimal number of tetrahedra.

Conversely, suppose that we have a triangulation $\calT$ of $S \times [0,1]$ so that $\calT$ restricted to $S\times\{0\}$ is $\calT_0$ and $\calT$ restricted to $S\times\{1\}$ is $\calT_1$. We attach onto each of $S \times \{ 0 \}$ and $S \times \{ 1 \}$ the following triangulation of $S \times I$. Each triangle in $S \times \{ 0 \}$, say, determines a prism in $S \times I$. This prism can be triangulated using three tetrahedra, as shown in \reffig{TriangulatePrism}.

\begin{figure}
\includegraphics{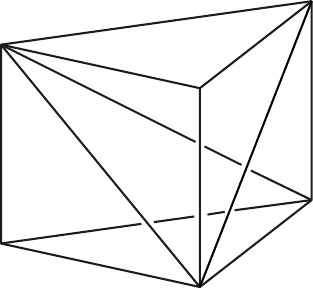}
  \caption{The triangulation of the product of a 2-simplex and an interval.}
  \label{Fig:TriangulatePrism}
\end{figure}

Gluing these prisms along their vertical faces gives a triangulation of $S \times I$. If we attach these onto $\calT$, one on each boundary component, we end with a new triangulation $\calT'$. This has the same triangulations on the boundary as $\calT$ did. However, $\calT'$ now has the property that for any tetrahedron, its intersection with $S \times \{ 0 ,1 \}$ is either a vertex, an edge or a face. So, when we dualise $\calT'$ to form a handle structure $\mathcal{H}$, this is pre-tetrahedral, for the following reason. We have one handle of $\mathcal{H}$ for every simplex of $\calT'$ that does not lie wholly in $S \times \{ 0,1\}$. So when a tetrahedron of $\calT'$ intersects $S \times \{ 0,1\}$ in the empty set or in a vertex, the resulting 0-handle is tetrahedral. When a tetrahedron intersects $S \times \{ 0,1\}$ in an edge, the resulting 0-handle is semi-tetrahedral. When a tetrahedron intersects $S \times \{ 0,1\}$ in a face, the resulting 0-handle is a product handle with length $3$.
Note that $\Delta(\mathcal{H}) \leq \Delta(\calT) + 12g(S)$, as follows.
Suppose the top of the product in \reffig{TriangulatePrism} is glued to $S\times\{0\}$. Then the three vertices on the top of the product give 3-handles, whereas the three vertices on the bottom do not. So, we have the following complexities: The tetrahedron on the bottom becomes a product 0-handle with length 3 incident to a single 3-handle, with complexity $1/8$. The tetrahedron in the middle becomes a semi-tetrahedral 0-handle incident to two 3-handles, with complexity $1/4+2/8$. The tetrahedron on the top becomes a tetrahedral 0-handle incident to three 0-handles, with complexity $1/2+ 3/8$. 
So we have added complexity $3/2$ for each triangle of S. There are less than  $4g(S)$ triangles of S by \reflem{NumberTrianglesAndEdges}. We need to attach these triangulated products onto both $S\times\{0\}$ and $S\times\{1\}$. So in total, we add complexity $8 g(S) * 3/2 = 12g(S).$

Now, $\mathcal{H}$ has empty parallelity bundle. For if it has any parallelity handles, then it would have to have a parallelity 2-handle. This would correspond to an edge of $\calT'$ not in $S \times \{ 0,1 \}$ but with endpoints lying in $S \times \{ 0 ,1\}$. There is no such edge in our triangulation. In particular, $\calH$ is coherently bundled. 

Pick a cellular spine $\Gamma$ in $S \times \{ 0 \}$ that is isotopic to the dual of $\calT_0$. Applying \refthm{MainTheoremProducts}, we obtain a cellular spine $\Gamma'$ in $S \times \{ 1 \}$ that is obtained from $\Gamma$ by at most $c \Delta(\mathcal{H})$ edge swaps. By \reflem{EdgeSwapBound}, $\Gamma'$ is therefore obtained from $\Gamma$ by at most $24 c g(S)\Delta(\calH)$ edge contractions and expansions.

The spine $\Gamma'$ intersects each edge of $\calT_1$ in at most $2$ points. Hence, the total number of intersections with the edges of $\calT_1$ is at most $12g(S)$ points by \reflem{NumberTrianglesAndEdges}.
Hence, by \reflem{SpineToTriangulation}, one can change $\Gamma'$ into the dual of $\calT_1$ using at most $96 g(S)^2$ edge contractions and expansions. So, we have related the dual of $\calT_0$ to the dual of $\calT_1$ using edge contractions and expansions, where the number of these is linearly bounded above by $\Delta(\calT)$. 
We can then convert this to a sequence of 2-2 Pachner moves joining $\calT_0$ to $\calT_1$, using the following lemma.
\end{proof}

\begin{lemma}
Let $\calT$ and $\calT''$ be 1-vertex triangulations of a closed orientable surface of genus $g$. Suppose that the spines $\Gamma$ and $\Gamma''$ dual to $\calT$ and $\calT''$ differ by a sequence of $k$ edge expansions and contractions. Then $\calT$ and $\calT''$ differ by a sequence of 2-2 Pachner moves with length at most $(8g(S) - 2)k$.
\end{lemma}

\begin{proof}
Let $\Gamma = \Gamma_0, \dots, \Gamma_k = \Gamma''$ be the sequence of spines relating $\Gamma$ to $\Gamma''$. Dual to each $\Gamma_i$ is a 1-vertex polygonal decomposition of $S$. This is an expression of $S$ as a union of polygons, each with at least 3 sides, with their edges identified in pairs, so that the resulting cell structure has a single 0-cell. An edge expansion relating two of the spines corresponds to dividing one of the polygons along a diagonal. We call this operation a \emph{diagonal decomposition}. So, we obtain a sequence of 1-vertex polygonal decompositions taking $\calT$ to $\calT''$, where each is obtained from its predecessor by a diagonal decomposition or the reverse of a diagonal decomposition. This sequence has length $k$. From the 1-vertex polygonal decomposition dual to $\Gamma_i$, we can obtain a 1-vertex triangulation $\calT_i$ by picking a vertex of each polygon and coning that polygon from the vertex.

We will show that $\calT_{i}$ and $\calT_{i+1}$ differ by a sequence of  2-2 Pachner moves with length at most $(8g(S) - 2)$. Suppose that $\Gamma_{i+1}$ is obtained from $\Gamma_i$ by an edge expansion. Then a polygon $P$ dual to a vertex of $\Gamma_i$ is subdivided into two polygons. In $\calT_i$, $P$ is triangulated some way with no vertex in the interior of $P$. In $\calT_{i+1}$, $P$ is triangulated in some other way, again with no vertex in the interior. But, by  \cite{CulikWood}, any two such triangulations of a polygon with $m$ sides differ by a sequence of 2-2 Pachner moves with length at most $(2m -2)$. In our case, $P$ has at most $4g(S)$ sides. The argument in the case where $\Gamma_{i}$ is obtained from $\Gamma_{i+1}$ by an edge expansion is analogous. The required bound on the number of Pachner moves taking $\calT$ to $\calT''$ immediately follows.
\end{proof}

\section{Proof of the main theorem}\label{Sec:MainProof}

We now have all the ingredients to give the proof of \refthm{Main}. 

\begin{proof}[Proof of \refthm{Main}]
As discussed in \refsec{Intro}, the quantities~(2) and~(3) of \refthm{Main} are known to lie within a bounded ratio of each other by the \v{S}varc-Milnor lemma. We showed in \refthm{StableTranslationLength} that quantities~(3) and~(4) lie within a bounded ratio of each other. We showed in \refprop{UpperBound} that $\Delta((S \times [0,1])/\phi)$ is at most a constant times $\ell_{\MCG(S)}(\phi)$. Hence, it remains to show that $\ell_{\MCG(S)}(\phi)$
is at most a constant times $\Delta((S \times [0,1])/\phi)$. In fact, in \refprop{SpineGraphMCG}, it was shown that the spine graph $\mathrm{Sp}(S)$ and the mapping class group $\MCG(S)$ are quasi-isometric. So, it suffices to show that $\ell_{\mathrm{Sp}(S)}(\phi)$
is at most a constant times $\Delta((S \times [0,1])/\phi)$.

Let $M_n = (S\times [0,1]) / \phi^n$. In particular, $M_n$ is an $n$-fold cover of $M_1$. Let $\calT_1$ be a triangulation for $M_1$ with $\Delta(\calT_1) = \Delta(M_1)$. Let $\calH_1$ be the dual handle structure, so $\Delta(\calH_1)=\Delta(\calT_1)$. Let $S_1$ be a normal fibre in $\calH_1$ that has least weight in its isotopy class. By \reflem{AcylindricalExists}, $S_1$ exists and is normally acylindrical. Choose a spine $\Gamma$ that is cellular in $S_1$.

Now consider lifts to the finite cyclic covers $M_n$. The triangulation $\calT_1$ lifts to a triangulation $\calT_n$, with dual handle structure $\calH_n$. The surface $S_1$ lifts to a normal fibre $S_n$ in $\calH_n$. The spine $\Gamma$ lifts to a cellular spine identical to $\Gamma$ on $S_n$. By cutting along all lifts of $S_1$ in $\calH_n$, we deduce that $\Delta(\calH_n) = n \Delta(\calH_1) = n\Delta(M_1)$.

We claim that $S_n$ is normally acylindrical. Suppose it is not and that there is annulus $A$ in $\calH_n$ with $\partial A = A \cap S_n$ being essential curves in $S_n$, and with $A$ vertical in the parallelity bundle of $M_n \cut S_n$, and near $\partial A$, emanating from the same side of $S_n$. We may orient the $n$ lifts of $S$ to $M_n$ coherently. Hence, when $M_n$ is cut along these lifts, each component of the resulting 3-manifold has one boundary component oriented inwards and one oriented outwards. These lifts intersect $A$ in a collection of parallel core curves. The outermost two are inconsistently oriented. Hence, there are two adjacent core curves in $A$ that are inconsistently oriented. Between them lies an annulus that projects homeomorphically to an annulus in $M_1 \cut S_1$ that makes $S_1$ normally cylindrical, which is a contradiction. This proves the claim.

Let $\calH'_n$ denote the handle structure obtained from cutting $\calH_n$ along $S_n$. By \reflem{ComplexityDecomposition}, $\Delta(\calH'_n)=\Delta(\calH_n)$.

By \reflem{CutFibredAlongAcylindrical}, $\calH'_n$ is coherently bundled. Thus \refthm{MainTheoremProducts} implies there is a sequence of at most $c\Delta(\calH'_n) = c\Delta(\calH_n)$ edge swaps taking $\Gamma$ in $S_n \times \{0\}$ to a cellular spine $\Gamma'$ in $S_n\times \{1\}$. By \reflem{EdgeSwapBound}, this gives a bound of at most ${24} c g(S)\Delta(\calH_n)$ edge contractions and expansions to take $\Gamma$ to $\Gamma'$. 

Apply the gluing map $\phi^n$ to obtain a spine $\phi^n(\Gamma)$ that is cellular in $S_n\times \{1\}$. By \reflem{SpinesSameComplex}, at most $48 g(S)^2 L$ edge contractions and expansions are required to take $\Gamma'$ to $\phi^n(\Gamma)$, where $L$ is the number of 1-cells in $\Gamma$. Note $L$ is independent of $n$ since $\Gamma$ is the same for all $n$.
Thus the total number of edge contractions and expansions we have used is at most
\[ {24} c \,g(S) \Delta(\calH_n) + 48 g(S)^2 L = {24} c\,n\,g(S)\Delta(M_1) + 48 g(S)^2 L. \]

At this point, we do not have a bound on $L$. However, we know it is independent of $n$. Hence there exists $N$ such that for all $n\geq N$, $48 g(S)^2 L \leq {24} c\,n\,g(S)\Delta(M_1)$. For such $n$, the total number of edge expansions and contractions required to take $\Gamma$ to $\Gamma'$ to $\phi^n(\Gamma)$ 
is at most ${48}\,c\,n\,g(S)\Delta(M_1)$. 

By \refthm{StableTranslationLength}, there exists a constant $k>0$ depending only on $g(S)$ such that the translation length of $\phi^n$ in the spine graph is at least $k\,n$ times the translation length of $\phi$. Thus
\[ k\,n\,\ell_{\Sp(S)}(\phi) \leq \ell_{\Sp(S)}(\phi^n) \leq {48}\,c\,n\,g(S)\Delta(M_1), \]
or
\[ \Delta(M_1) \geq \frac{k}{{48}\,c\,g(S)} \ell_{\Sp(S)}(\phi). \qedhere \]
\end{proof}

\bibliographystyle{amsplain}
\bibliography{references}

\end{document}